\documentclass[12pt,leqno]{amsart}

\usepackage[notref,notcite]{showkeys}
\usepackage{amsmath,stmaryrd}
\usepackage{amssymb}
\usepackage{amsthm}
\usepackage[english]{babel}
\usepackage{epsf}
\usepackage{dsfont}
\usepackage{graphicx}
\usepackage{esint} 
\usepackage{dsfont}
\usepackage[pagebackref]{hyperref} 
\hypersetup{pdfpagemode=FullScreen,  colorlinks=true}  
\usepackage{enumitem} 
\usepackage{soul,xcolor}
\usepackage{listings}
\usepackage{xparse}
\usepackage{epstopdf} 
\usepackage{enumitem}

\setstcolor{red}

\theoremstyle{plain}
\newtheorem{Th}{Theorem}[section]
\newtheorem{Lemma}[Th]{Lemma}
\newtheorem{Cor}[Th]{Corollary}
\newtheorem{Prop}[Th]{Proposition}

\theoremstyle{definition}
\newtheorem*{Def*}{Hypothesis \HH}
\newtheorem*{DefCO}{Hypothesis \COF}
\newtheorem{Def}[Th]{Definition}

\newtheorem{Rem}[Th]{Remark}

\numberwithin{equation}{section}

\def\Xint#1{\mathchoice
{\XXint\displaystyle\textstyle{#1}}%
{\XXint\textstyle\scriptstyle{#1}}%
{\XXint\scriptstyle\scriptscriptstyle{#1}}%
{\XXint\scriptscriptstyle\scriptscriptstyle{#1}}%
\!\int}
\def\XXint#1#2#3{{\setbox0=\hbox{$#1{#2#3}{\int}$ }
\vcenter{\hbox{$#2#3$ }}\kern-.6\wd0}}

\def\dashint{\Xint-}

\def\Yint#1{\mathchoice
    {\YYint\displaystyle\textstyle{#1}}%
    {\YYint\textstyle\scriptstyle{#1}}%
    {\YYint\scriptstyle\scriptscriptstyle{#1}}%
    {\YYint\scriptscriptstyle\scriptscriptstyle{#1}}%
      \!\iint}
\def\YYint#1#2#3{{\setbox0=\hbox{$#1{#2#3}{\iint}$}
    \vcenter{\hbox{$#2#3$}}\kern-.51\wd0}}
\def\longdash{{-}\mkern-3.5mu{-}} 

\def\fiint{\Yint\longdash}

\DeclareMathOperator{\Tr}{Tr}
\DeclareMathOperator{\dist}{dist}
\DeclareMathOperator{\diver}{div}
\DeclareMathOperator{\supp}{supp}

\DeclareMathOperator{\Jac}{Jac}

\newcommand{\R}{\mathbb R}
\renewcommand{\S}{\mathbb S}

\newcommand{\A}{\mathcal A}
\newcommand{\B}{\mathcal B}
\newcommand{\C}{\mathcal C}
\newcommand{\wt}{\widetilde}

\newcommand{\ep}{\hfill $\square$}
\newcommand{\1}{{\mathds 1}}
\newcommand{\I}{\RN{1}}
\newcommand{\II}{\RN{2}}
\newcommand{\III}{\RN{3}}

\newcounter{hcountno}
\newcommand{\hcount}[1]{\refstepcounter{hcountno} \label{#1}}
\newcommand{\HH}{(\hyperref[H1]{$\mathcal H$})} 
\newcommand{\COF}{(\hyperref[H2]{$\mathcal{COF}$})}

\ExplSyntaxOn
\NewDocumentCommand{\RN}{m}
 {
  \textup{ \int_to_Roman:n { #1 } }
 }
\ExplSyntaxOff

\makeatletter
\renewcommand*\env@matrix[1][*\c@MaxMatrixCols c]{%
  \hskip -\arraycolsep
  \let\@ifnextchar\new@ifnextchar
  \array{#1}}
\makeatother

\title{The Regularity Problem in Domains with Lower Dimensional Boundaries}

\author[Dai]{Zanbing Dai}
\address{Zanbing Dai. School of Mathematics, University of Minnesota, Minneapolis, MN 55455, USA}
\email{dai00003@umn.edu}

\author[Feneuil]{Joseph Feneuil}
\address{Joseph Feneuil. Mathematical Sciences Institute, Australian National University, Acton, ACT, Australia}
\email{joseph.feneuil@anu.edu.au}

\author[Mayboroda]{Svitlana Mayboroda}
\address{Svitlana Mayboroda. School of Mathematics, University of Minnesota, Minneapolis, MN 55455, USA}
\email{svitlana@math.umn.edu}

\thanks{S. Mayboroda was partly supported by the NSF
RAISE-TAQS grant DMS-1839077 and the Simons foundation grant 563916, SM. J. Feneuil was partially supported by the Simons foundation grant 601941, GD and by the European Research Council via the project ERC-2019-StG 853404 VAREG}

\begin{document}

\maketitle

\begin{abstract}

In the present paper we establish the solvability of the Regularity boundary value problem in domains with lower dimensional boundaries (flat and Lipschitz) for operators whose coefficients exhibit small oscillations analogous to the Dahlberg-Kenig-Pipher condition. 

The proof follows the classical strategy of showing bounds on the square function and the non-tangential maximal function. The key novelty and difficulty of this setting is the presence of multiple non-tangential derivatives. To solve it, we consider a cylindrical system of derivatives and establish new estimates on the ``angular derivatives''. 
\end{abstract}

\maketitle

\tableofcontents

\section{Introduction}

There are three principal types of boundary value problems for elliptic operators with rough ($L^p$) data: Dirichlet, Neumann, and Regularity. The Dirichlet problem consists of establishing the existence and uniqueness of solutions with a given trace on the boundary, the Neumann problem corresponds to prescribing the flux, that is, the normal derivative on the boundary, again, in $L^p$. The Regularity problem postulates that the tangential derivative of the trace of the solution is known, once again, in some $L^p$ space. As such, it can be seen as a companion of the Neumann problem in which the tangential rather than the normal derivative of the solution is given, or as a version of the Dirichlet problem corresponding to the smoother boundary data. 

The Dirichlet problem has received a lot of attention in the past 30-40 years and we will not be able to even briefly mention all the references in the subject. Its well-posedness was established, in particular, for $t$-independent operators on all Lipschitz domains \cite{jerison1981dirichlet, kenig2000new, hofmann2015regularity}, for the the Laplacian on all uniformly rectifiable sets with mild topological conditions \cite{dahlberg1977estimates, hofmann2014uniformb, azzam2021semi, azzam2020harmonic}, which was then extended to the sharp class of the so-called Dahlberg-Kenig-Pipher (DKP) operators \cite{kenig2001dirichlet, dindos2007lp, hofmann2021uniform} and for their analogues in domains with lower-dimensional boundaries \cite{david2019dahlberg, feneuil2018dirichlet}.

The Neumann and Regularity problems in $L^p$ proved to be much more challenging. In particular, concerning the latter, up until recently the only known results pertained to either $t$-independent scenario \cite{kenig1993neumann} or a ``small constant" DKP case \cite{dindovs2017boundary}. The breakthrough article \cite{mourgoglou2021regularity} by Mourgoglou and Tolsa was the first one to consider the regularity problem on domains beyond Lipschitz graphs: they proved the solvability of the regularity problem for the Laplacian on domains with uniformly rectifiable boundaries and some mild topology. Just in the past few months the first ``big constant" DKP result was announced, by two different arguments, by Dindo{\v s}, Hofmann, Pipher \cite{dindos2022regularity} in the half plane and Lipschitz domains, and simultaneously, by Mourgoglou, Poggi, Tolsa \cite{mourgoglou2022lp} on domains with uniformly rectifiable boundaries.

The present paper is devoted to the setting of domains with lower dimensional boundaries. It establishes the solvability of the regularity problem in the complement of $\R^d$, or more generally, of a Lipschitz graph, for an appropriate analogue of the ``small constant" DKP coefficients. The higher co-dimensional setting presented numerous new challenges, particularly, due to the presence of ``torsion", the derivatives which roughly speaking turn the solution around a thin boundary which are not present in the traditional $(n-1)$-dimensional case. Respectively, we had to invent new structural properties of the operators which on one hand, are amenable to the analysis in desired geometric scenarios, and on the other, still allow for a control of the second derivatives of a solution in a square function. All this will be discussed in detail below. 

Let us also mention that in the setting of the domains with lower dimensional boundaries we are bound to work with degenerate elliptic operators, whose coefficients grow as powers of the distance to the boundary. This provides a curious new motivation point. Our operators, as explained below, essentially look like $-{\rm div} \dist(\cdot, \partial\Omega)^\beta \nabla$ with a suitable power $\beta$ depending on the dimension of the set and of the boundary. This is reminiscent of the Caffarelli-Silvestre extension operator which allows one to view the fractional Laplacian $(-\Delta)^\gamma$, $\gamma\in (0,1)$, on $\R^d$ as a Dirichlet-to-Neumann map for the operator $-{\rm div}\, {\rm dist} (\cdot, \R^d)^\beta \nabla $ on $\R^{d+1}$, where $\beta = 1-2\gamma$ (see \cite{caffarelli2007extension} and also an extension to higher powers by A. Chang and co-authors in \cite{chang2017class}).  Respectively, the mapping properties of the Dirichlet-to-Neumann map become the mapping properties of the fractional Laplacian. By the same token, one could view the Dirichlet-to-Neumann map of our operators as an embodiment of a new concept of differentiation or integration on rough lower-dimensional sets, and in this vein the appropriate estimates correspond exactly to the solution of the Regularity and Neumann problems. This paper is the first step in the direction.

Let us now turn to definitions and statements of the main results. Let $0<d<n$ be two integers. If $d=n-1$, the domain $\Omega$ is the half-space $\R^{n}_+ := \{(x,t) \in \R^{d} \times (0,\infty)\}$ and if $d<n-1$, then $\Omega := \R^n \setminus \R^d := \{(x,t) \in \R^d \times (\R^{n-d}\setminus \{0\}) \}$. In the rest of the article, $t$ will be seen as a horizontal vector, and thence $t^T$ will correspond to the vertical vector. It is technically simpler and more transparent to work in $\R^{n}_+$ and $\R^n \setminus \R^d$ rather than a more general graph domain, but the goal is to treat the class of coefficients which would automatically cover the setting of Lipschitz domains via a change of variables -- see Corollary~\ref{cLip}. 

We take an operator $L:=-{\diver}|t|^{d+1-n}\mathcal{A}\nabla$ and the first condition that we impose is of course the ellipticity and boundedness of $\mathcal A$: there exists $\lambda>0$ such that for $\xi, \zeta\in \mathbb{R}^n$, and $(x,t) \in \Omega$,
\begin{align}\label{ELLIP}
\lambda|\xi|^2\leq \mathcal{A}(x,t)\xi\cdot \xi\ \ \text{and}\ \
|\mathcal{A}(x,t)\xi\cdot \zeta|\leq \lambda^{-1}|\xi||\zeta|.
\end{align}
We write \eqref{ELLIP}$_\lambda$ when we want to refer to the constant in \eqref{ELLIP}. Then, we say that $u\in W^{1,2}_{loc}(\Omega)$ is a weak solution to $Lu=0$ if  for any $\varphi\in C^\infty_0(\Omega)$, we have
\begin{align}\label{WSOL}
\iint_{\Omega}\mathcal{A} \nabla u\cdot \nabla \varphi\frac{dt}{|t|^{n-d-1}}dx=0.
\end{align} 

When $d=n-1$ these are the classical elliptic operators and when $d<n-1$ the weight given by the power of distance to the boundary is necessary and natural: if the coefficients are not degenerate, the solutions do not see the lower dimensional sets. For instance, a harmonic function in $\R^n\setminus \R^d$ is the same as a harmonic function in $\R^n$  for sufficiently small $d$. All this is discussed in detail in \cite{david2017elliptic} where we develop the elliptic theory for the operators at hand.  In particular, in the aforementioned work we construct the elliptic measure $\omega_L^X$ associated to $L$ so that for any continuous and compactly supported boundary data $g$, the function
\begin{equation} \label{defug}
u(X) := \int_{\R^d} g(y) \, d\omega^X(y)
\end{equation}
is a weak solution to $Lu=0$, which continuously extends to $\overline{\Omega}$ by taking the values $u= g$ on $\partial \Omega = \R^d$.
 
With this at hand, we turn to the definition of the Regularity problem. The averaged non-tangential maximal function $\wt{N}$ is defined for any function $u\in L^{2}_{loc}(\Omega)$ as
\begin{align}\label{DEFNT}
\wt{N}(u)(x)=\sup_{(z,r)\in\Gamma(x)}u_{W}(z,r),
\end{align}
where $\Gamma(x)$ is the cone $\{(z,r)\in \R^{d+1}_+, \, |z-x|< r\}$,  and $u_{W}(z,r)$ is the $L^2$-average
\begin{align*}
u_{W}(z,r):=\bigg (\fiint_{W(z,r)}|u(y,s)|^2 dy\, ds\bigg )^{\frac{1}{2}}, 
\end{align*}
over the Whitney box
\begin{align}\label{DEFWHIT}
    W(z,r):=\{(y,s)\in \Omega, \, |y-z| < r/2,\,  r/2\leq |s|\leq 2r\}.
\end{align} 
Observe that when $d<n-1$, a Whitney cube is a bounded, annular region, so in particular, the higher co-dimensional Whitney cubes $W(z,r)$ are invariant under rotation around the boundary. We say that the {\bf Regularity problem is solvable in $\mathbf{L^p}$} if for any $g\in C^\infty_0(\R^d)$, the solution given by \eqref{defug} verifies 
\begin{equation} \label{Rp}
\|\wt N(\nabla u)\|_{L^p(\R^d)} \leq C \|\nabla g\|_{L^p(\R^d)}
\end{equation}
with a constant $C>0$ that is independent of $g$. If the Regularity problem is solvable in $L^p$, then we deduce by density that for any $g\in L^1_{loc}(\R^d)$ such that $\|\nabla g\|_{L^p(\R^d)} < \infty$, there exists a solution to $Lu=0$ 
subject to \eqref{Rp} which converges non-tangentially to $g$. The proof of this fact is non-trivial, but classical. See for instance Theorem 3.2 of \cite{kenig1993neumann} for the proof of the non-tangential convergence from the bound \eqref{Rp}, and since the space $\{g\in L^1_{loc}(\R^d), \, \|\nabla g\|_{L^p(\R^d)} < \infty\}$ is homogenous and only equipped of a semi-norm, we need density results analogous to Lemma 5.7, Remark 5.10, Lemma 5.11 in \cite{david2017elliptic}. 

Going further, we say that a function $f$ satisfies {\bf the Carleson measure condition} if $\sup_{W(z,s)}|f|^2\frac{dsdz}{s}$ is a Carleson measure on $\Omega$, that is,  there exists a constant $M\geq 0$ such that
\begin{equation} \label{defCMM}
\sup_{x\in \mathbb{R}^d,r>0} \dashint_{z\in B(x,r)}\int_0^{r} \sup_{W(z,s)}|f|^2\frac{dsdz}{s}\leq  M.
\end{equation}
We write $f\in CM$, or $f\in CM(M)$ when we want to refer to the constant in \eqref{defCMM}.  It is fairly easy to check that $f\in L^\infty(\Omega)$, and we even have
\begin{equation} \label{PpCAL}
\|f\|_{L^\infty(\Omega)}\leq CM^{1/2} \qquad \text{ whenever } f \in CM(M),
\end{equation}
with a constant that depends only on $d$ and $n$. 

The main result of the present paper is as follows.

\begin{Th}\label{THRE3MA} Let $0 \leq d<n$ be two integers.  
For any $\lambda >0$, there exists a small parameter $\kappa >0$ and a large constant $C$,  both depending only on $\lambda$, $d$, and $n$, with the following property. Consider an elliptic operator $L:=-\diver [|t|^{d+1-n} \A \nabla]$ that satisfies \eqref{ELLIP}$_\lambda$ and such that $\A$ can be decomposed as $\A = \B + \C$, $\B$ is a block matrix
\begin{equation} \label{formofB}
\B = \begin{pmatrix} B_1 & B_2 \frac{t}{|t|} \\ \frac{t^T}{|t|}B_3 & b_4 I\end{pmatrix},
\end{equation}
 where $B_1$, $B_2$, $B_3$, and $b_4$ are respectively a $d\times d$ matrix, a $d$-dimensional vertical vector\footnote{Since $t$ is a horizontal vector, $B_2\frac{t}{|t|}$ is seen as a matrix product giving a $d\times (n-d)$ matrix.}, a $d$-dimensional horizontal vector\footnote{That is $\frac{t^T}{|t|}B_3$ is a $(n-d)\times d$ matrix.}, a scalar function, and 
\begin{equation} \label{NB+CareCM}
|t||\nabla B_1| + |t||\nabla B_2| + |t||\nabla B_3| + |t||\nabla b_4| + |\C| \in CM(\kappa).
\end{equation}
Then the Regularity problem is solvable in $L^2(\R^d)$, that is
\begin{equation} \label{N<Tr}
\|\wt N(\nabla u_g)\|_{L^2(\R^d)} \leq C \|\nabla g\|_{L^2(\R^d)}
\end{equation}
whenever $g \in C^\infty_0(\R^d)$ and $u_g$ is a solution to $Lu=0$ given by in \eqref{defug}.
\end{Th}

Note that when $d=n-1$ our result corresponds to the main result in  \cite{dindovs2017boundary} by Dindo{\v s}, Pipher, and Rule. In this case, the coefficients of $\B$ satisfy the so-called Dahlberg-Kenig-Pipher (DKP) condition with a small constant and the addition of $\C$ is made possible by the perturbation results \cite{kenig1995neumann, dai2021carleson}. The DKP condition is sharp, that is, its failure could result in the failure of solvability of the Dirichlet problem \cite{fefferman1991theory} and hence, a failure of solvability of the Regularity problem by \cite{fefferman1991theory}. 

In the setting of the domains with lower dimensional boundaries the special structure \eqref{formofB} is new. It is dictated by 
the aforementioned need to control the ``torsion" of the coefficients, that is, not only to control the oscillations of the coefficients in the transversal direction to the boundary, but also to make sure that they are well-behaved, in a very peculiar sense, in the angular coordinate in cylindrical coordinates naturally induced by $\R^n\setminus \R^d$. Roughly speaking, we want to have an almost isometry to some constant coefficient matrix as far as the $t$ direction is concerned. 

One good test for whether our class of coefficients is sound structure-wise is whether it allows for a change of variables that would yield the results on rougher, e.g., Lipschitz, domains. After all, this was an initial motivation for the DKP Carleson conditions on the coefficients in half-space back when Dahlberg suggested them. To this end, consider $d<n-1$ and take a Lipschitz function $\varphi :\, \R^d \mapsto \R^{n-d}$. Let $\Omega_\varphi:= \{(x,t)\in \R^n, \, t \neq \varphi(x)\}$. We set $\sigma:= \mathcal H^{d}|_{\partial \Omega_\varphi}$ to be the $d$-dimensional Hausdorff measure on the graph of $\varphi$, which is the boundary of $\Omega_\varphi$, and we construct the ``smooth distance'' 
\[D_\varphi(X):= \left( \int_{\partial \Omega_\varphi} |X-y|^{-d-\alpha} \, d\sigma(y) \right)^{-\frac1\alpha}, \quad \alpha >0.\]
The quantity $D_\varphi(X)$ is equivalent to $\dist(X,\partial \Omega_\varphi)$, see Lemma 5.1 in \cite{david2019dahlberg}, so the operator $L_\varphi := -\diver [D_\varphi^{d+1-n} \nabla]$ falls under the elliptic theory developed in \cite{david2017elliptic}. Moreover, it was proved that the Dirichlet problem for such an operator $L_\varphi$ is solvable in $L^p$ in a complement of a small Lipschitz graph \cite{feneuil2018dirichlet} and much more generally, in a complement of a uniformly rectifiable set \cite{david2020harmonic, feneuil2020absolute}. It is also explained in the aforementioned works why $D_\varphi$ as opposed to the Euclidean distance has to be used in this context. Using the results from \cite{feneuil2018dirichlet}, one can prove  solvability of the Dirichlet problem in $L^2$.  Here we establish solvability of Regularity problem.

\begin{Cor}\label{cLip}
Let $\varphi: \mathbb{R}^d\rightarrow \mathbb{R}^{n-d}$ be a Lipschitz function, and set $\Omega_\varphi$ and $L_\varphi$ as above. There exists $\kappa>0$ such that if $\|\nabla\varphi\|_{L^\infty(\mathbb{R}^d)}\leq \kappa$, then the Regularity problem is solvable in $L^2(\partial \Omega_\varphi)$.
\end{Cor}

The reader can consult Section~\ref{sLip} for the proof and the detailed definitions. 

\subsection{Remarks on the proof of Theorem \ref{THRE3MA}.}
At this point let us return to the Main result, Theorem \ref{THRE3MA}, and discuss some highlights of the proof along with the particular challenges of the higher co-dimensional setting.

Similarly to the strategy used in codimension 1, we want to prove that for any $g$ smooth enough and $u_g$ constructed as in \eqref{defug}, we have
\begin{equation} \label{S<Nz}
\|S(\nabla u_g) \|_{L^2(\R^d)} \leq C \|g\|_{L^2(\R^d)} + C \kappa \|\wt N(\nabla u_g)\|_{L^2(\R^d)}
\end{equation}
and
\begin{equation} \label{N<Sz}
\|\tilde N(\nabla u_g)\|_{L^2(\R^d)} \leq C \|S(\nabla u_g) \|_{L^2(\R^d)},
\end{equation}
for some $C>0$. Here, $S$ is a square function that will be defined in \eqref{DEFSQ} below. We can see that when $\kappa$ is small the two estimates above would formally imply the bound \eqref{N<Tr}. They are the crux of the matter and the core of the argument. 
However, even in this passage there are considerable additional difficulties. Nothing guarantees that $\|\wt N(\nabla u)\|_{L^2(\R^d)}$ is finite, and if we do not know {\em a priori} whether $\|\wt N(\nabla u)\|_{L^2(\R^d)}$ is finite, we cannot use \eqref{S<Nz}--\eqref{N<Sz} to deduce that $\|\tilde N(\nabla u)\|_{L^2(\R^d)} \leq C \|g\|_{L^2(\R^d)}$. For that reason we cannot simply concentrate on  \eqref{S<Nz}--\eqref{N<Sz}, but rather have to prove  local versions of those estimates, where all the terms are guaranteed to be finite, and we then carefully take a limit to directly establish 
\begin{equation} \label{Nfinite}
\|\tilde N(\nabla u)\|_{L^2(\R^d)} \leq C \|\nabla g\|_{L^2(\R^d)} < +\infty.
\end{equation}
Unfortunately, taking the limit is already far from trivial, because the term $\|\nabla g\|_{2}$ is obtained roughly by taking the limit of $\|\nabla u(x,\epsilon)\|_{2}$, and to ensure convergence, we had to assume that $\|\nabla \mathcal A\|_\infty <+\infty$ as in \cite{kenig1993neumann}, and then obtain \eqref{Nfinite} for all $\mathcal A$ by interchanging two limits. In the classical  case of codimension 1, the situation is considerably easier because more tools are available to us (for instance layer potential representations). 

The principal issue though are still the estimates on the quantity $S(\nabla u)$, \eqref{DEFSQ}. Clearly, it involves two derivatives, and in principle we do not have enough regularity of the coefficients ($\C$ is not necessarily continuous) to be able to directly bound the second derivatives of the solution, not to mention the actual refined estimates that we are targeting. This led us to a separate paper devoted to the Carleson perturbation theory for the Regularity problem \cite{dai2021carleson} (cf. \cite{kenig1995neumann} when $d=n-1$). However, even with that and even for $\A=\B$ we could not follow the route paved for $d=n-1$ in \cite{dindovs2017boundary}. We finally realized that these arguments are not well adapted to the cylindrical structure of our space and the additional, quite involved, structural considerations are necessary. Let us try to give some ideas here.

\subsubsection{Cylindrical Coordinate Derivatives} \label{SSCCS}
As we mentioned, we shall use $S(\nabla u)$ as an intermediate quantity in our computations, and so we will need to estimate second derivatives. However, taking the second derivatives in the cartesian system of coordinates will not be adapted to our context, and we prefer to consider ``cylindrical derivatives'' defined below.

We notice that there are three difference types of directions. One is the tangential direction, which goes alone the boundary $\mathbb{R}^d\times \{t=0\}$. The second one is the angular direction, which rotates around the boundary, and the last one is the radial direction that moves away from the boundary. We write $\nabla_x=(\partial_{1}, \partial_{2},...,\partial_{d})$ and $\nabla_t=(\partial_{d+1}, \partial_{d+2},...,\partial_n)$, where $\partial_i=\vec e_i\cdot \nabla$ and $\vec e_i\in \mathbb{R}^n$ denotes the vector with a $1$ in the $i$-th coordinate and $0$'s elsewhere.

\begin{Def}\label{DEFVRA}
The radial directional derivative $\partial_r$ is defined as:
\begin{align}\label{DEFNU}
\partial_r:=\sum_{\alpha=d+1}^n\frac{t_\alpha}{|t|}\partial_\alpha.
\end{align}
For each $d+1\leq i, j\leq n$, the directional derivative $\partial_{\varphi_{ij}}$  is defined as:
\begin{align}\label{DEFVAR}
\partial_{\varphi_{ij}}:=-\frac{t_i}{|t|}\partial_j+\frac{t_j}{|t|}\partial_i.
\end{align}
\end{Def}

The important property of $\partial_\varphi$ is that
\begin{equation} \label{dphit=0}
\partial_\varphi |t| = 0
\end{equation}
To lighten the notation, we write $\partial_\varphi$ for any angular directional derivative. We will mention $i,j$ explicitly when it is necessary. Furthermore we define the angular gradient $\nabla_\varphi$ as a vector derivative whose components are all angular directional derivatives $(\partial_{\varphi_{ij}})_{d+1\leq i,j\leq n}$ and
\[|\nabla_\varphi u|^2= \frac12 \sum_{i,j=d+1}^n |\partial_{\varphi_{ij}} u|^2.\]

Note that $\partial_{\varphi_{ii}}=0$ for all $d+1\leq i\leq n$ and $\partial_{\varphi_{ij}}=\partial_{\varphi_{ij}}$ for all $d+1\leq i,j\leq n$. Also, we can easily check that the tangential, angular, and radial directions are perpendicular to each other. More importantly, for any $u\in W^{1,2}_{loc}$, we have the identity that $|\nabla_t u|^2=|\partial_r u|^2+|\nabla_\varphi u|^2$ almost everywhere (see Proposition \ref{PFGEV}).  Consequently, it suffices to establish estimates for the average non-tangential maximal functions of $\nabla_x$, $\nabla_\varphi$ and $\partial_r$. In the rest of the article, we will write
\begin{align}\label{DEFNAB}
\overline \nabla=(\nabla_x, \nabla_\varphi, \partial_r).
\end{align}

One of the main reasons for using the cylindrical coordinate system is that the operator $L=-\diver[|t|^{d+1-n}\mathcal{A}\nabla]$ can be written in terms of $\partial_x, \partial_\varphi,$ and $\partial_r$ (see Proposition \ref{PpOPER}) when the coefficient matrix $\mathcal{A}$ is in the form of (\ref{coe.afor}). The expression (\ref{eqPOPER}) not only simplifies the computations, but also helps us to better understand the geometric structure of the operator $L$.

\begin{Rem} \label{Rdrrdrphi}
The notation $\partial_r$, $\partial_\varphi$, ... might be confusing at first, as these are {\bf not} derivatives in a new system of coordinates. We will not use a change of variable to turn our system of coordinates from a cartesian to a cylindrical one. Instead, $\partial_r$ and $\partial_\varphi$ denote linear combinations of derivatives in cartesian coordinates, or derivatives along some curves (i.e., $r$ and $\varphi$ are not ``new variables''). They are used for properly grouping the derivatives. In particular, we do not need to properly define a bijection $(x,t) \mapsto (x,r,\varphi)$ or its Jacobian.
\end{Rem}

\subsubsection{Commutators}
The common point between $\partial_x$ and $\partial_\varphi$ is that they both cancel out the weight $|t|^{d+1-n}$, so they will be handled in a similar manner by commuting them with the operator $L$; the estimates on the last derivative $\partial_r$ will then be obtained by using the equation (Proposition \ref{PpOPER}). The difference between the two differential operators $\partial_x$ and $\partial_\varphi$  is that $\partial_x$ commute with $\nabla$ and $\overline{\nabla}$, and $\partial_\varphi$ do not commute with the radial and other angular derivatives, but fortunately, everything will work out at the end because the commutators have zero average on $W(z,r)$.
The computations pertaining to commutators are performed in Section \ref{Sangder}, for instance Proposition \ref{PPLL1} gives that
\[ [\partial_r, \partial_\varphi]:= \partial_r \partial_\varphi - \partial_\varphi \partial_r = -\frac{\partial_\varphi}{|t|}.\]

\subsubsection{Local bounds} 
We want to prove local versions of \eqref{S<Nz}--\eqref{N<Sz}. Before introducing the notation, let us mention that a weak solution is in $W^{2,2}_{loc}$ whenever $\nabla A \in L^\infty_{loc}$, this is a well known fact which we proved again in Proposition \ref{LW22loc}.

We have already defined the non-tangential maximal function in \eqref{DEFNT}, and the square function of $v\in W^{1,2}_{loc}(\Omega)$ is defined as:
\begin{align}\label{DEFSQ}
S(v)(x):=\bigg (\iint_{\widehat {\Gamma}_a(x)}|\nabla v(y,s)|^2 \,\frac{dyds}{|s|^{n-2}}\bigg )^{\frac{1}{2}},
\end{align}
where 
\[\widehat {\Gamma}(x)=\{(y,s)\in \mathbb{R}^n\setminus \mathbb{R}^d: |y-x|\leq |s|\}\]
is a higher-codimension cone with vertex $x\in \R^d$.  We write
\begin{align}\label{eq.sqsum}
    S(\overline\nabla u)^2:= \sum_{i=1}^d S(\partial_{x_i} u)^2+ \sum_{d<i,j\leq n} S(\partial_{\varphi_{ij}} u)^2 + S(\partial_r u)^2,
\end{align}
and the square functions of $\nabla_x u$ and $\nabla_\varphi u$ are defined in a similar manner. 

For a function $0 \leq \Psi \leq 1$, the definitions of the localized square functions and the non-tangential maximal functions are 
\begin{align}\label{DEFLSFG}
    S(v|\Psi)(x):=\bigg (\iint_{\widehat \Gamma(x)}|\nabla v|^2\Psi \frac{dsdy}{|s|^{n-d}}\bigg )^{1/2}
\end{align}
and 
\begin{align*}
    \wt N(v|\Psi)(x)=\sup_{(z,r)\in \Gamma(x)}(v|\Psi)_{W,a}(z,r)
\end{align*}
where $(v|\Psi)_{W,a}$ is defined on $\mathbb{R}^{d+1}_+$ by
\begin{align*}
    (v|\Psi)_{W,a}(z,r):=\bigg (\frac{1}{|W_a(z,r)|}\iint_{W_a(z,r)}|v|^2\Psi dyds\bigg )^{1/2}.
\end{align*}

``Good'' cut-off functions will satisfy the following hypothesis.
\begin{DefCO}
We say that a function $\Psi$ satisfies \COF \hcount{H2}  if $\Psi$ is a cut-off function, that is if $\Psi \in C^\infty(\overline{\Omega})$, $0 \leq \Psi \leq 1$, $\Psi$ is radial - i.e. there exists $\psi \in C^\infty(\R^{d+1}_+)$ such that $\Psi(x,t) = \psi(x,|t|)$ - and we have the bound
\[ |t||\nabla \Psi| \leq K \quad \text{ and } \quad \1_{\supp \nabla \Psi} \in CM(K).\]
We write \COF$_K$ when we want to refer to a constant for which $|t|\nabla \Psi| \leq K$ and $\1_{\supp \nabla \Psi} \in CM(K)$, and $K$ will always be chosen $\geq 1$.
\end{DefCO}
We show that if $\Psi$ is a ``good'' cut-off function, then for any weak solution $u\in W^{2,2}_{loc}(\Omega)$ to the equation $Lu=0$, we have
\[\|S(\overline\nabla u|\Psi)\|^2_{2}
\leq C_1\kappa \|\wt{N}(\nabla u|\Psi)\|^2_{2}+\|\Tr_{\Psi} (\nabla_x u)\|^2_2+\text{``error terms''},
\]
where $\Tr_{\Psi} (\nabla_x u)$ is an approximation of trace of $\nabla_x u$ that depends on how far is $\supp \Psi$ to $\partial \Omega$. The precise statement can be found in Lemma \ref{LSLNF}. In addition, for a reduced class of ``good'' cut-off function we will obtain the local $N\leq S$
\[\|\wt{N}(\nabla u|\Psi^{3})\|^2_2\lesssim \|S(\overline\nabla u|\Psi)\|^2_{2} + \text{``error terms''},\]
where an exact estimate is given in Lemma \ref{LENSTP}. The ``error terms'' that we mentioned above go to zero once we extend local estimates to global ones. The careful definitions of the ``good" cutoffs, a delicate splitting of the derivatives, and an enhanced structure of the operator are all important for the algebra of the computations.
Afterwards, when $\kappa$ is small, by taking $\Psi \uparrow 1$, we are able obtain the estimate 
\begin{equation} \label{N<limTr}
\|\wt N(\nabla u)\|_2\lesssim \lim_{\epsilon \to 0} \|\Tr_\epsilon (\nabla_x u)\|_2
\end{equation}
 whenever $u$ is an energy solution (see Theorem \ref{THMAIN1a}). Finally, with this at hand, two natural questions now arise. Does the limit $\lim_{\epsilon} \|\Tr_\epsilon (\nabla_x u_g)\|_2$ exists and does it converge to $\|\nabla g\|_2$?

\subsubsection{Approximation Results}
We want to follow the strategy that Kenig and Pipher used in \cite{kenig1993neumann}. The idea is to construct a sequence of coefficients $\{\A^j\}_{j\in \mathbb{N}}$ such that $\A^j \equiv \A$ on $\{|t|>1/j\}$ and $\A^j$ is Lipschitz up to the boundary. In particular $\A_j$ converges pointwise to $\A$, which guarantees the convergence of the solution $u^j_g$ to $u_g$ (see Theorem \ref{PaprRe}). Meanwhile, since $\A^j$ is continuous up to the boundary, $\|\Tr_\epsilon(\nabla_x u^j_g)\|_2$ converges indeed to $\|\nabla g\|_2$ because $\nabla_x u_j$ is continuous/smooth up to the boundary. 
We can swap the two limits (in $\epsilon$ and in $j$), because \eqref{N<limTr} entails 
a uniform convergence of the traces in $j$. 

However, the construction of the $\A^j$ used by Kenig and Pipher does not immediately transfer to our higher codimensional setting. In addition, we only succeeded to obtain global bounds on $\nabla \nabla_x u$ (and not on all the second derivatives, like we could do in the codimension 1 setting), and this forced us to prove Theorem \ref{THMAIN1a} before doing the approximation. For that reason, even if we globally follow the spirit of Kenig and Pipher's method, we cannot say that our argument is a simple adaptation of \cite{kenig1993neumann}.

\subsubsection{Self-improvement} All the arguments that we presented will allow us to prove the $L^2$-solvability of the Regularity problem  for a reduced class of operators, and then we will ``self improve'' it to Theorem \ref{THRE3MA}. The reduced class of operators on which most of intermediate results will be written is given as follows.

\begin{Def*}
We say that the operator $L:= - \diver(|t|^{d+1-n} \A\nabla)$ satisfies the assumption \HH \hcount{H1} if 
\begin{itemize}
\item $L$ is uniformly elliptic, that is there exists $\lambda\in (0,1)$ such that 
\begin{equation} \label{defellip} 
\lambda|\xi|^2\leq \mathcal{A}(x,t)\xi\cdot \xi\ \ \text{and}\ \
|\mathcal{A}(x, t)\xi\cdot \zeta|\leq \lambda^{-1} |\xi||\zeta| \qquad \text{ for } (x,t) \in \Omega, \, \xi,\zeta \in \R^n;
\end{equation}
\item the matrix $\A$ can be written as
\begin{align}\label{coe.afor}
\mathcal{A}(x,t)=
  \left( {\begin{array}{cc}
   \mathcal{A}_1(x,t)& \mathcal{A}_2(x,t) \frac{t}{|t|}\\
0 & Id_{(n-d)\times (n-d)} \\
  \end{array} } \right),
\end{align}
where $\A_1$ is a $d\times d$-matrix function, and $\A_2$ is vertical vector of length $d$\footnote{That is, $\mathcal{A}_2(x,t) \frac{t}{|t|}$ is a matrix operation which gives a $d\times (n-d)$ matrix};
\item There exists $\kappa>0$ such that 
\begin{equation} \label{defCMAx} 
|t| |\nabla \A_1| + |t| |\nabla \A_2| \in CM(\kappa).
\end{equation}
\end{itemize}
We write \HH$_{\lambda,\kappa}$ when we want to refer to the constants in \eqref{defellip}, and \eqref{defCMAx}.  The constant $\kappa$ will ultimately be small. 
\end{Def*}

Keep in mind that we consider the operators satisfying \HH{} at first, partially because some of our intermediate results can not be stated with the assumptions from Theorem \ref{THRE3MA} (for instance we need $u\in W^{2,2}_{loc}$ for Lemma \ref{LlastS<N}, and so cannot consider Carleson perturbation $\C$ for this result), but also because we want to simplify the proofs (for instance, our proofs would work with $\A$ in the form \eqref{formofB} instead of \eqref{coe.afor}, but many extra computations would be needed in Sections \ref{SN<S} and \ref{SS<N}). That is, we sacrificed the optimality of the intermediate results in order to shorten our proof.

We prove in Section \ref{SThred} the following result, which seems at first glance weaker than Theorem \ref{THRE3MA}.

\begin{Th}\label{THRE2MA}
Take $\lambda, M>0$. There exists $\kappa \in (0,1)$ small enough (depending only on $\lambda$, $d$, and $n$) such that if $L:=-\diver (|t|^{d+1-n}\mathcal{A}\nabla)$ is an elliptic operator satisfying \HH$_{\lambda,\kappa}$, then for any boundary data $g\in C^\infty_0(\R^d)$, the solution $u$ to $Lu=0$ constructed as in \eqref{defug} or equivalently by using Lax-Milgram theorem (see Lemma \ref{LaxMilgram}) verifies
\begin{align}\label{eqTHM00z}
\|\wt{N}(\nabla u)\|_{L^2(\R^d)} \leq C\|\nabla g\|_{L^2(\R^d)},
\end{align}
where $C>0$ depends only on $\lambda$, $d$, and $n$. 
\end{Th}

Then, using the theory of Carleson perturbations for the Regularity problem \cite{dai2021carleson, kenig1995neumann} we improve the above result in Section \ref{Sselfimpro}  and we get Theorem \ref{THRE3MA}, as desired.

\section{Equation in Cylindrical Coordinates} \label{Sangder}

In Subsection \ref{SSCCS}, we introduced a set of directional derivatives adapted to the cylindrical structure of $\Omega$ (when $d<n-1$). The gradient $\overline \nabla = (\nabla_{x}, \nabla_\varphi, \partial_r)$ in cylindrical coordinate has a norm equivalent to the one of the classical gradient (see Proposition \ref{PFGEV}), which makes $\overline \nabla$ equivalent to $\nabla$ for estimates on first order derivatives. 
We compute the expression of our elliptic operator in the cylindrical system of derivatives (see Proposition \ref{PpOPER}).

For the second order derivatives in cylindrical coordinates, we will need to know the commutators between $\nabla_{x}$, $\nabla_\varphi$, and $\partial_r$, which we compute in Proposition \ref{PPLL1} and Proposition \ref{CorDAR}. We observe that the non trivial commutators will always involve the angular derivative $\nabla_\varphi$. In order to deal with them, we shall crucially rely on Proposition \ref{PDECC}, which uses the fact that the angular directional derivative $\partial_{\varphi} u(x, r\theta)$ has zero mean on the unit sphere for almost every $(x,r)\in \mathbb{R}^{d+1}_+$. From there, we will be able to use the Poincar\'e inequality and recover second order derivatives (that will eventually be controlled). 

Recall that, as mentioned in Remark \ref{Rdrrdrphi},  $r$ and $\varphi$ are not ``new variables in a cylindrical system'', and $\partial_r$ and $\partial_{\varphi_{ij}}$ are just a linear combination of Euclidean derivatives.

\begin{Prop}\label{PFGEV}
Let $\partial_{\varphi_{ij}}$ and $\partial_r$ be directional derivatives defined in Definition \ref{DEFVRA}, and let $\overline \nabla u$ be the cylindrical gradient defined in \eqref{DEFNAB}. We have 
\[\nabla u \cdot \nabla v = \nabla_x u \cdot \nabla_x v + (\partial_r u)(\partial_r v) +\frac{1}{2}\sum_{i,j=d+1}^n (\partial_{\varphi_{ij}} u)(\partial_{\varphi_{ij}} v) = \overline \nabla u \cdot \overline \nabla v\] 
whenever it makes sense (for instance for $u\in W^{1,2}_{loc}(\Omega)$ and  almost every $(x,t) \in \Omega$). In particular, we have $|\nabla u|^2 = |\overline \nabla u|^2$.
\end{Prop}

\begin{proof}
We just need to prove 
\[\nabla_t u \cdot \nabla_t v := \sum_{\alpha=d+1}^n (\partial_{t_\alpha} u)(\partial_{t_\alpha} v) =  (\partial_r u)(\partial_r v) +\frac{1}{2}\sum_{i,j=d+1}^n (\partial_{\varphi_{ij}} u)(\partial_{\varphi_{ij}} v).\] 
According to the definition of $\partial_{\varphi_{ij}}$ in (\ref{DEFVAR}), we have
\begin{align*}
    \sum_{i,j=d+1}^n(\partial_{\varphi_{ij}} u)(\partial_{\varphi_{ij}} v) =\sum_{i,j=d+1}^n\Big \{\frac{t_i^2}{|t|^2}(\partial_{t_j} u)(\partial_{t_j} v) -2\frac{t_it_j}{|t|^2}(\partial_{t_i} u)(\partial_{t_j} v)+\frac{t^2_j}{|t|^2}(\partial_{t_i} u)(\partial_{t_i} v)\Big \}.
\end{align*}
The first term on the righthand side equals $\nabla_t u \cdot \nabla_t v$ since $\sum_{i=d+1}^n t_i^2/|t|^2=1$. For the same reason, the last term is also $\nabla_t u \cdot \nabla_t v$.  We can factorize the second term of the righthand side into the product of a sum in $i$ and a sum in $j$, and we easily observe from the definition (\ref{DEFNU}) that the middle term is indeed $-2(\partial_r u)(\partial_r v)$. The proposition follows. 
\end{proof}

The second proposition establishes an integration by parts for the angular and radial derivatives. 

\begin{Prop}\label{IBBdrdphi}
Let $u,v\in \C^\infty(\R^n)$ be such that either $u$ or $v$ is compactly supported in $\Omega$. We have the identities
\[\iint_\Omega (\partial_r u) \, v \, |t|^{d+1-n} \, dt\, dx = - \iint_\Omega u \, (\partial_r v) \, |t|^{d+1-n} \, dt\, dx\] 
and
\[\iint_\Omega (\partial_{\varphi} u) \, v \, |t|^{d+1-n} \, dt\, dx = - \iint_\Omega u \, (\partial_\varphi v) \, |t|^{d+1-n} \, dt\, dx\] 
where $\partial_\varphi$ stands for any of the $\partial_{\varphi_{ij}}$, $d+1\leq i,j\leq n$.
\end{Prop}

\begin{proof}
If one writes the integrals in cylindrical coordinates, the integration by parts for $\partial_r$ is immediate once you notice that we imposed the boundary condition $uv = 0$ when $r=0$. 

The second identity is also expected, but let us write is formally. Take $d+1\leq i,j\leq n$ and we have by definition of $\partial_{\varphi_{ij}}$ that
\[I := \iint_\Omega (\partial_{\varphi_{ij}} u) \, v \, |t|^{d+1-n} \, dt\, dx 
=  \iint_\Omega v \Big[ (\partial_iu) \, t_j |t|^{d-n} - (\partial_ju) \, t_i |t|^{d-n}  \Big] \, dt\, dx \]
We use the integration by part to remove $\partial_i$ and $\partial_j$ from $u$, and we get
\[\begin{split}
I = - \iint_\Omega u (\partial_{\varphi_{ij}} v)\,  |t|^{d+1-n} \, dt\, dx - \iint_\Omega uv \Big[\, \partial_i(t_j |t|^{d-n}) - \, \partial_j (t_i |t|^{d-n})  \Big] \, dt\, dx.
\end{split}\]
It is easy to check that $\partial_i(t_j |t|^{d-n}) - \, \partial_j (t_i |t|^{d-n}) = 0$ in $\Omega$, thus the proposition follows.
\end{proof}

The following proposition rearranges the derivatives, in order to use $\partial_r$ and $\partial_\varphi$ instead of the $t$-derivatives in the expression of $L$.
 
\begin{Prop}\label{PpOPER}
Let $L=-\diver(|t|^{d+1-n} \mathcal{A} \nabla)$ be such that
\begin{align*}
\mathcal{A}(x,t)=
  \left( {\begin{array}{cc}
   \mathcal{A}_1(x,t)& \mathcal{A}_2(x,t) \frac{t}{|t|}\\
   \frac{t^T}{|t|} \mathcal{A}_3(x,t) & b(x,t)Id_{(n-d)\times (n-d)} \\
  \end{array} } \right),
\end{align*} 
where $t\in \mathbb{R}^{n-d}$ is seen as a $d$-dimensional horizontal vector, and where $\A_1$, $\A_2$, $\A_3$, and $b$ are respectively a $d\times d$ matrix, a $d$-dimensional vertical vector, a $d$-dimensional horizontal vector, and a scalar function. 

Then:
\[L=-|t|^{d+1-n}\Big [\diver _x(\mathcal{A}_1\nabla_x)+\diver _x(\mathcal{A}_2\partial_r) + \partial_r (\mathcal A_3 \nabla_x)
    +\partial_r(b\partial_r)+\frac{1}{2}\sum_{i,j=d+1}^n\partial_{\varphi_{ij}}(b\partial_{\varphi_{ij}}) \Big ].\]
        
In particular, if $L$ satisfies \HH, then
\begin{equation}\label{eqPOPER}
    L=-|t|^{d+1-n}\Big [\diver _x(\mathcal{A}_1\nabla_x)+\diver _x(\mathcal{A}_2\partial_r) +\partial^2_r +\frac{1}{2}\sum_{i,j=d+1}^n\partial^2_{\varphi_{ij}} \Big ].
\end{equation} 
\end{Prop}

\begin{proof}
We first decompose as
\begin{multline}\label{eqLL17}
    L=-\diver _x(|t|^{d+1-n}\mathcal{A}_1\nabla_x )
    -\diver _x(|t|^{d-n-1}\mathcal{A}_2 \frac{t}{|t|}\nabla_t)
    -\diver _t(|t|^{d-n-1} \frac{t^T}{|t|} \mathcal{A}_3\nabla_x)\\-\diver _t(|t|^{d+1-n}b\nabla_t)=:L_1+L_2+L_3+L_4.
\end{multline}
Since the weight $|t|^{d+1-n}$ is independent of $x$, one has $L_1=-|t|^{d+1-n}\diver _x(\mathcal{A}_1\nabla_x)$ and $L_2 = -|t|^{d+1-n}\diver _x(\mathcal{A}_2 \frac{t}{|t|}\nabla_t) = -|t|^{d+1-n}\diver _x(\mathcal{A}_2\partial_r)$, since by definition $\partial_r$ is $\frac{t}{|t|}\nabla_t$. 

Recall that $\A_3$ is a horizontal vector and $\nabla_x$ is a vertical vector differential operator, so $\A_3\nabla_x$ is a scalar (differential operator). In conclusion,
\[L_3 = -\diver_t(|t|^{d-n}t^T) \A_3\nabla_x - |t|^{d-n-1} \frac{t}{|t|} \cdot \nabla_t( \A_3\nabla_x) \\
= 0 - |t|^{d-n-1} \partial_r ( \A_3\nabla_x).\]

At this point, it remains to treat $L_4$. The integration by parts entails that, for $u,v\in C^\infty_0(\Omega)$,
\[\begin{split}
\iint_\Omega (L_4 u) v \, dt\,dx & = \iint_\Omega b \nabla_t u \cdot \nabla_t v\, |t|^{d+1-n} \, dt\, dx \\
& = \iint_\Omega b (\partial_r u)(\partial_r v) \, |t|^{d+1-n} \, dt\, dx + \frac{1}{2}\sum_{i,j=d+1}^n \iint_\Omega b (\partial_{\varphi_{ij}} u) (\partial_{\varphi_{ij}} v) \, |t|^{d+1-n} \, dt\, dx
\end{split}\]
by Proposition \ref{PFGEV}. Using the integration by part for $\partial_r$ and $\partial_{\varphi_{ij}}$ given by Proposition \ref{IBBdrdphi}, we deduce
\[\iint_\Omega (L_4 u) v \, dt\,dx = - \iint_\Omega \partial_r(b\partial_r u) \, v \, |t|^{d+1-n} \, dt\, dx -  \frac{1}{2}\sum_{i,j=d+1}^n \iint_\Omega b \partial_{\varphi_{ij}}(b\partial_{\varphi_{ij}} u)  v \, |t|^{d+1-n} \, dt\, dx\]
Since the above equality is true for all $u,v \in C^\infty_0(\Omega)$, we conclude
\[L_4 = -|t|^{d+1-n}\Big [\partial_r(b\partial_r)+\frac{1}{2}\sum_{i,j=d+1}^n\partial_{\varphi_{ij}}(b\partial_{\varphi_{ij}}) \Big ].\]
The proposition follows.
\end{proof}

In the next results, we want to compute commutators. We immediately have that $[\partial_x,\partial_r] = 0$ and $[\partial_x,\partial_\varphi] = 0$. The normal derivative $\partial_r$ and the angular directional derivative $\partial_{\varphi}$ do not commute, therefore we want to compute their commutator. 

\begin{Prop}\label{PPLL1}
Let $\partial_{\varphi}$ and $\partial_r$ be the derivatives defined in Definition \ref{DEFVRA}. Then we have 
\[[\partial_r, \partial_{\varphi}]:= \partial_r \partial_{\varphi} - \partial_\varphi \partial_r = -\frac{\partial_{\varphi}}{|t|}.\]
\end{Prop}

\begin{proof}
Fix a angular directional derivative $\partial_{\varphi_{ij}}$.
We use the expressions of $\partial_{\varphi_{ij}}$ and $\partial_r$ given in Definition \ref{DEFVRA} to write
\begin{multline}\label{eqLL02}
\partial_r \partial_{\varphi_{ij}}=\sum_{\alpha=d+1}^n\frac{t_\alpha}{|t|}\partial_\alpha \partial_{\varphi_{ij}}=\sum_{\alpha=d+1}^n \frac{t_\alpha}{|t|} \Big [   \partial_{\varphi_{ij}}\partial_\alpha- \partial_\alpha\Big (\frac{t_i}{|t|}\Big )\partial_j
+ \partial_\alpha\Big (\frac{t_j}{|t|}\Big )\partial_i\Big ]\\
=\sum_{\alpha=d+1}^n\Big [\partial_{\varphi_{ij}}\Big (\frac{t_\alpha}{|t|}\partial_\alpha\Big )-\partial_{\varphi_{ij}}\Big (\frac{t_\alpha}{|t|}\Big )\partial_\alpha-\frac{t_\alpha}{|t|}\partial_\alpha\Big (\frac{t_i}{|t|}\Big )\partial_j
+\frac{t_\alpha}{|t|}\partial_\alpha\Big (\frac{t_j}{|t|}\Big )\partial_i\Big ].
\end{multline}
We notice that the first term on the last line of (\ref{eqLL02}) is exactly $\partial_{\varphi_{ij}}\partial_r$ after summing over all $d+1\leq \alpha\leq n$.
The third and forth terms of (\ref{eqLL02}) are similar, and are both zero. Indeed,  
\[-\sum_{\alpha=d+1}^n\frac{t_\alpha}{|t|}\partial_\alpha\Big (\frac{t_i}{|t|}\Big )\partial_j=-\sum_{\alpha=d+1}^n
    \frac{t_\alpha}{|t|}\Big (\frac{\delta_{i\alpha}}{|t|}
    -\frac{t_it_\alpha}{|t|^3}\Big )\partial_j = - \frac{t_i}{|t|} \partial_j  +  \frac{t_i}{|t|} \sum_{\alpha=d+1}^n \frac{t_\alpha}{|t|^2} \partial_j   =0.\]

The second term on the last line of (\ref{eqLL02}) can be handled as follows:
\begin{multline}\label{eqLL03}
    -\sum_{\alpha=d+1}^n\partial_{\varphi_{ij}}\Big (\frac{t_\alpha}{|t|}\Big )\partial_\alpha=-\sum_{\alpha=d+1}^n\Big [-\frac{t_i}{|t|}\partial_j\Big (\frac{t_\alpha}{|t|}\Big )+\frac{t_j}{|t|}\partial_i\Big (\frac{t_\alpha}{|t|}\Big )\Big ]\partial_\alpha\\
    =\sum_{\alpha=d+1}^n\Big [\frac{t_i}{|t|}\Big (\frac{\delta_{j\alpha}}{|t|}-\frac{t_jt_\alpha}{|t|^3}\Big )-\frac{t_j}{|t|}\Big (\frac{\delta_{i\alpha}}{|t|}-\frac{t_it_\alpha}{|t|^3}\Big )\Big ]\partial_\alpha=\frac{t_i}{|t|^2}\partial_j-\frac{t_j}{|t|^2}\partial_i=-\frac{\partial_{\varphi_{ij}}}{|t|}.
\end{multline}
By combining our observations all together, the proposition follows. 
\end{proof}

Different angular derivatives do not commute either, and we give their commutator below.

\begin{Prop}\label{prop.vaco}
We trivially have $[\partial_{\varphi_{ij}},\partial_{\varphi_{\alpha\beta}}] = 0$ when $i,j,\alpha,\beta$ are all different. If $i,j,k$ are all different, we have
\[    [\partial_{\varphi_{ij}},\partial_{\varphi_{ik}}]= - [\partial_{\varphi_{ji}},\partial_{\varphi_{ik}}] = \frac{1}{|t|} \partial_{\varphi_{jk}}.\]
\end{Prop}

\begin{proof}
The identity $[\partial_{\varphi_{ij}},\partial_{\varphi_{ik}}]= - [\partial_{\varphi_{ji}},\partial_{\varphi_{ik}}]$ comes from the fact that $\partial_{\varphi_{ij}} = - \partial_{\varphi_{ji}}$. For the second identity, we brutally compute. We use the definitions of the angular derivatives, and develop the expressions to obtain 8 terms that we pair as follows: 
\begin{multline*}
[\partial_{\varphi_{ij}},\partial_{\varphi_{ik}}]  = \left[\frac{t_i}{|t|} \partial_j \Big( \frac{t_i}{|t|} \partial_k \Big) -  \frac{t_i}{|t|} \partial_k \Big( \frac{t_i}{|t|} \partial_j \Big)  \right] - \left[\frac{t_i}{|t|} \partial_j \Big( \frac{t_k}{|t|} \partial_i \Big) -  \frac{t_k}{|t|} \partial_i \Big( \frac{t_i}{|t|} \partial_j \Big)  \right] \\
 - \left[\frac{t_j}{|t|} \partial_i \Big( \frac{t_i}{|t|} \partial_k \Big) -  \frac{t_i}{|t|} \partial_k \Big( \frac{t_j}{|t|} \partial_i \Big)  \right] + \left[\frac{t_j}{|t|} \partial_i \Big( \frac{t_k}{|t|} \partial_i \Big) -  \frac{t_k}{|t|} \partial_i \Big( \frac{t_j}{|t|} \partial_i \Big)  \right] \\
 := T_1 + T_2 + T_3 + T_4.
\end{multline*}
By using the product rule for every term and the fact that $i,j,k$ are pairwise different, we easily get that $T_1 = T_4 = 0$ and
\[ T_2 = \frac{t_k}{|t|^2} \partial_j \quad \text{ and } T_3 = -\frac{t_j}{|t|^2} \partial_k.\]
We conclude that 
\[[\partial_{\varphi_{ij}},\partial_{\varphi_{ik}}]  = -\frac{t_j}{|t|^2} \partial_k + \frac{t_k}{|t|^2} \partial_j = \frac1{|t|} \partial_{\varphi_{jk}} \]
as desired.
\end{proof}

Now it is time to compute the commutator $[L, \partial_{\varphi}]$, which is a crucial step for establishing local bounds between the square functions and the non-tangential maximal functions. We will explain more when we start building up these estimates. We compute the commutator when $L$ satisfies \HH; we could compute the commutator for general elliptic operator $L$, but we do not need it, so we spare ourselves the extra complications.

\begin{Prop}\label{CorDAR}
Let $\mathcal{A}$ be a $n\times n$ matrix in the form of (\ref{coe.afor}), then for any $v \in W^{2,2}_{loc}(\Omega)$
\[
[L,\partial_{\varphi}](v) = |t|^{d-n}\Big[ \diver_x (\A_2 \partial_\varphi) + 2 \partial_r\partial_{\varphi} v \Big]  + \diver _x(|t|^{d+1-n}(\partial_{\varphi}\mathcal{A})\nabla v) .
\]
Here we identity $\partial_\varphi \mathcal A$ with its non-trivial submatrix, that is the first $d$ rows.
\end{Prop}

\begin{proof}
Fix an angular directional derivative $\partial_\varphi$. We rearrange the derivatives to avoid using any $t$-derivatives, and Proposition \ref{PpOPER} entails that
\[
    L=-|t|^{d+1-n}\Big [\diver _x(\mathcal{A}_1\nabla_x)+\diver _x(\mathcal{A}_2\partial_r)+ \partial_r^2 +\frac{1}{2}\sum_{i,j=d+1}^n\partial_{\varphi_{ij}}^2 \Big ] =:L_1+L_2+L_3 + L_4.\]

We note that $[L,\partial_\varphi]=\sum_{\alpha=1}^4 [L_\alpha,\partial_\varphi].$ So we will compute each $[L_\alpha, \partial_\varphi]$ individually. Let us start from the easiest one $[L_1, \partial_\varphi]$. Since $\nabla_x$ and $\partial_\varphi$ commute and $\partial_\varphi |t| = 0$, we have 
\begin{align}\label{cod.eq00}
    [L_1, \partial_\varphi]=\diver _x(|t|^{d+1-n}(\partial_\varphi \mathcal{A}_1)\nabla_x).
\end{align}

We turn to the operator $L_2$. By product rule, one has
\begin{multline}\label{cod.eq01}
    L_2\partial_\varphi=-\diver _x(|t|^{d+1-n}\mathcal{A}_2\partial_r \partial_\varphi)
    =-\diver _x(|t|^{d+1-n}\mathcal{A}_2\partial_\varphi \partial_r) - \diver _x(|t|^{d+1-n}\mathcal{A}_2[\partial_r,\partial_\varphi])\\
    =-\diver _x(|t|^{d+1-n}\partial_\varphi(\mathcal{A}_2\partial_r))+\diver_x(|t|^{d+1-n}(\partial_\varphi\mathcal{A}_2)\partial_r) + |t|^{d-n} \diver _x(\mathcal{A}_2 \partial_\varphi),
\end{multline}
where we used Proposition \ref{PPLL1} to compute the commutator.
The first term on the last line of (\ref{cod.eq01}) is exactly $\partial_\varphi L_2$ because $\partial_\varphi |t|^{d+1-n}\equiv 0$ and $\partial_x$ and $\partial_\varphi$ commute. Thus, (\ref{cod.eq01}) becomes
\begin{align}\label{cod.eq02}
    [L_2, \partial_\varphi]=\diver_x(|t|^{d+1-n}(\partial_\varphi\mathcal{A}_2)\partial_r) - |t|^{d-n} \diver _x(\mathcal{A}_2\partial_\varphi).
\end{align}
For simplicity, we group $ [L_1, \partial_\varphi]$ and  $[L_2, \partial_\varphi]$. We have for any $v\in W^{2,2}_{loc}(\Omega)$ that
\[[L_1, \partial_\varphi](v) +  [L_2, \partial_\varphi](v) =  \diver _x\Big(|t|^{d+1-n}(\partial_{\varphi}\mathcal{A})\nabla  v \Big) + |t|^{d-n} |\diver _x(\mathcal{A}_2\partial_\varphi)|.
\]

As for the commutator between $L_3$ and $\partial_\varphi$, we use Proposition \ref{PPLL1} multiple times to write
\[\partial_r^2 \partial_\varphi = \partial_r \partial_\varphi \partial_r - \partial_r \Big( \frac{\partial_\varphi}{|t|} \Big) = \partial_\varphi \partial_r^2 - \frac1{|t|}\partial_\varphi \partial_r - \frac1{|t|} \partial_r \partial_\varphi + \frac1{|t|^2} \partial_\varphi = \partial_\varphi \partial_r^2 - 2 \partial_r \partial_\varphi\]
and thus deduce 
\[[L_3,\partial_\varphi] = 2 |t|^{d+1-n} \partial_r \partial_\varphi. \]

It remains to establish that $[L_4,\partial_\varphi] = 0$. We take $d+1\leq \alpha,\beta\leq n$ so that $\partial_\varphi = \partial_{\varphi_{\alpha\beta}}$. We invoke the fact that $\partial_{\varphi_{ij}} = - \partial_{\varphi_{ij}}$ and then Proposition \ref{prop.vaco}  to obtain 
\[\begin{split}
- 2|t|^{n-d-1}[L_4,\partial_\varphi] & = - \sum_{i \neq \alpha} \Big(\partial_{\varphi_{\beta i}}^2 \partial_{\varphi_{\beta\alpha}} - \partial_{\varphi_{\beta\alpha}}\partial_{\varphi_{\beta i}}^2  \Big) + \sum_{j \neq \beta} \Big( \partial_{\varphi_{\alpha j}}^2 \partial_{\varphi_{\alpha\beta}} -  \partial_{\varphi_{\alpha\beta}} \partial_{\varphi_{\alpha j}}^2 \Big) \\
& = - \sum_{i \neq \alpha} \Big(\partial_{\varphi_{\beta i}}  \partial_{\varphi_{i\alpha}} - \partial_{\varphi_{\alpha i}} \partial_{\varphi_{\beta i}}  \Big) + \sum_{j \neq \beta} \Big( \partial_{\varphi_{\alpha j}} \partial_{\varphi_{j\beta}} -  \partial_{\varphi_{\beta j}} \partial_{\varphi_{\alpha j}} \Big) \\
& =  - \sum_{i \neq \alpha,\beta} \Big(\partial_{\varphi_{\beta i}}  \partial_{\varphi_{i\alpha}} - \partial_{\varphi_{\alpha i}} \partial_{\varphi_{\beta i}}  \Big) + \sum_{j \neq \alpha,\beta} \Big( \partial_{\varphi_{\alpha j}} \partial_{\varphi_{j\beta}} -  \partial_{\varphi_{\beta j}} \partial_{\varphi_{\alpha j}} \Big).
\end{split}\]
because $\partial_{\varphi_{kk}} = 0$.  We can freely change $j$ in $i$ in the second sum, and after recalling again that $\partial_{\varphi_{ij}} = - \partial_{\varphi_{ij}}$, we observe that two sums in the right-hand side above cancel with each other. We conclude that $[L_4,\partial_\varphi] = 0$, which finishes the proof of the proposition.
\end{proof}

Finally, we will need the following version of the Poincar\'e inequality.

\begin{Prop}\label{PDECC}
Let $u\in W^{1,2}_{loc}(\Omega)$ and let $\Phi \in C^\infty_0(\Omega,\R^+)$ be a radial function. Then   $\partial_\varphi u$ has zero mean on sphere, that is, for almost every $(x,r) \in \R^{d+1}_+$, we have 
\begin{align*}
    (\partial_{\varphi} u)_{\mathbb{S}^{n-d}}(x, r):=\dashint_{\mathbb{S}^{n-d}}\partial_{\varphi} u(x, r\theta)d\sigma(\theta)=0,
\end{align*}
where $\sigma$ is the surface measure on the unit sphere $\mathbb{S}^{n-d}$.

As a result, we have 
\[\iint_\Omega |\partial_\varphi u|^2 \Phi \, dt \, dx \leq C \iint_\Omega |t|^2|\partial_\varphi^2 u|^2 \Phi \, dt\, dx,\]
where $C>0$ is a universal constant. 
\end{Prop}

\begin{proof} 
Let $\phi(x, r)\in C_0^\infty(\mathbb{R}^{d+1}_+)$ and set $\Phi(x,t):=\phi(x, |t|)$. Observe that 
\[ \iint_{\R^{d+1}_+} \left| \int_{\mathbb S^{n-d}} \partial_\varphi u \, d\theta \right| |\phi| \, r^{n-d-1} \, dr \, dx \leq  \iint_{\R^n} |\partial_\varphi u| |\Phi| \, dx \, dt \leq C_\phi \]
which proves by Fubini's theorem that $\int_{\mathbb S^{n-d}} \partial_\varphi u \, d\theta$ and thus $(\partial_\varphi u)_{\mathbb S^{n-d}}$ exists for almost every $(x,r) \in \R^{d+1}_+$. Notice now that $\partial_{\varphi} \Phi = (\partial_\varphi |t|)(\partial_r \phi)  \equiv 0$ because $\partial_{\varphi}|t|\equiv 0$ and $|\partial_r \phi|<\infty$. Therefore the integration by parts (see Proposition \ref{IBBdrdphi}) entails that
\begin{multline}\label{ECC.eq02}
    \iint_{\mathbb{R}^{d+1}_+}(\partial_{\varphi} u)_{\mathbb{S}^{n-d-1}}\, \phi \, r^{n-d-1} \, dr\, dx=\frac{1}{\sigma(\mathbb{S}^{n-d-1})}\iint_{\Omega}(\partial_{\varphi} u)\Phi\, dt\, dx \\
    =-\frac{1}{\sigma(\mathbb{S}^{n-d-1})} 
    \iint_{\Omega} u \partial_{\varphi} \Phi\,   \, dt\, dx =0.
\end{multline}
Since the identity (\ref{ECC.eq02}) holds for every $\phi\in C_0^\infty(\mathbb{R}^{d+1}_+)$, it is enough to conclude that $(\partial_{\varphi} u)_{\mathbb{S}^{n-d-1}}(x ,r)=0$ for almost every $(x,r)\in \mathbb{R}^{d+1}_+$. 

\medskip
 
Let us turn to the second part of the Proposition. Without loss of generality, we can assume that $d\leq n-2$ (because otherwise angular derivatives do not exist) and  $\partial_\varphi = \partial_{\varphi_{ij}}$ with $i=n$ and $j=n-1$. Write a running point of $\R^n$ as $(x,t',t_{n-1},t_n) \in \R^d \times \R^{n-d-2} \times \R \times \R$. We consider a function $\psi \in \R^{n-1}_+ := \{(x,t',r)\in \R^{n-2} \times (0,\infty)\}$, and then $\Psi(x,t) := \psi(x,t', |(t_{n-1},t_n)|)$.
The same argument as before shows that for almost every $(x,t',r) \in \R^{n-1}_+$, the function $\theta \to \partial_\varphi u(x,t', r\theta)$ lies in $L^2(\S^1,d\sigma)$ and
\[(\partial_\varphi u)_{\mathbb S^1}(x,t',r) := \fint_{\S_1} \partial_\varphi u(x,t', r\theta) \, d\sigma(\theta) = 0.\]
However, $\S^1$ is just the unit circle, so we have the bijection
\[\rho:\, z\in [0,2\pi) \mapsto \theta = (\cos(z),\sin(z)) \in \S^1\]
and we even have $d\sigma(\theta) = dz$. Moreover,
\[\begin{split}
\frac{\partial}{\partial z} [u(x,t',r\theta)] & =  -r\sin(\theta) \frac{\partial}{\partial t_{n-1}} u(x,t',r\theta) + r\cos(\theta) \frac{\partial}{\partial t_{n}} u(x,t',r\theta) = r \partial_\varphi u(x,t',r\theta)
\end{split}\]
and similarly
\[\frac{\partial^2}{\partial z^2} [u(x,t',r\theta)] = r^2 \partial_\varphi^2 u(x,t',r\theta).\]
We deduce that, for almost every $(x,t',r) \in \R^{n-1}_+$,
\[\fint_0^{2\pi} \frac{\partial}{\partial z}[ u(x,t', r\rho(z))]  \, dz = (\partial_\varphi u)_{\mathbb S^1}(x,t',r)  = 0\]
and then, by the Poincar\'e inequality on $[0,2\pi]$, 
\[\begin{split}
\int_{\S^1} |\partial_\varphi u(x,t', r\theta)|^2 \, d\sigma(\theta) & = r^{-2} \int_0^{2\pi} \Big| \frac{\partial}{\partial z} [u(x,t',r\rho(z))] \Big|^2 dz \\
& \leq C r^2 \int_0^{2\pi} \Big| \frac{\partial^2}{\partial z^2} [u(x,t',r\rho(z))] \Big|^2 dz = Cr^2 \int_{\S^1} |\partial_\varphi^2 u(x,t', r\theta)|^2 \, d\sigma(\theta).
\end{split}\]
We conclude by integrating over $(x,t',r) \in \R^{n-1}_+$. Since a radial function $\Phi$ depends on $t_{n-1}$ and $t_n$ only via the norm $|(t_{n-1},t_n)|$, we get
\[\begin{split}
\iint_{\Omega} |\partial_\varphi u|^2 \Phi \, dt\, dx & = \int_{\R^{n-2}} \int_{0}^\infty \Phi \left( \int_{\S^1}  |\partial_\varphi u(x,t', r\theta)|^2 \, d\sigma(\theta) \right) \, r \, dr\, dt'\, dx \\
& \lesssim \int_{\R^{n-2}} \int_{0}^\infty \Phi \left( \int_{\S^1}  |\partial_\varphi^2 u(x,t', r\theta)|^2 \, d\sigma(\theta) \right) \, r^3 dr\, dt'\, dx \\
& \quad = \iint_{\Omega} |t|^2 |\partial_\varphi^2 u|^2 \Phi \, dt\, dx.
\end{split}\]
The lemma follows.
\end{proof}

\section{\texorpdfstring{$N\leq S$}{TEXT} Local Estimates, Part 1: Integration by Parts} 

\label{SN<S}

We want to bound of the non-tangential maximal function by the square functional. In this section, we prove preliminary estimates that will be improved to the desired $N<S$ estimate in the next section by using a ``good $\lambda$'' argument.

\medskip

We observe first that if $v\in L^2_{loc}(\Omega)$ and $\Psi$ is a cut-off function, we have by a simple application of Fubini's theorem that
\begin{equation} \label{S=int1}
\iint_\Omega |v|^2 \Psi \, \frac{dt}{|t|^{n-d-2}} \, dx \approx \int_{\R^d} \left(\iint_{(y,t) \in \widehat \Gamma(x)} |v(y,t)|^2 \Psi(y,t) \, \frac{dt}{|t|^{n-2}} \, dy \right) \, dx
\end{equation}
so in particular, for any $v\in W^{1,2}_{loc}(\Omega)$
\begin{equation} \label{S=int2}
\iint_\Omega |\nabla v|^2 \Psi \, \frac{dt}{|t|^{n-d-2}} \, dx \approx \|S(v|\Psi)\|^2_{L^2(\R^d)}.
\end{equation}
The constants in \eqref{S=int1} and \eqref{S=int2} depends only on $d$ and $n$.

Moreover, the Carleson measure condition is well adapted to the averaged non-tangential maximal function, in that we have
\begin{equation} \label{Carleson}
\iint_\Omega |v|^2 |f|^2 \Psi \, \frac{dt}{|t|^{n-d}} \, dx \leq C M \|\wt N(u|\Psi)\|^2_{L^2(\R^d)}
\end{equation}
whenever $f\in CM(M)$. The statement in this particular context can be found as Proposition 4.3 in \cite{feneuil2018dirichlet}, but the proof is an easy consequence of the classical Carleson inequality.

\begin{Lemma}\label{lmGNSE}
In this lemma, $\partial_v$ stands for either a tangential derivative $\partial_x$ or an angular derivative $\partial_\varphi$.
For any function $u\in W^{2,2}_{loc}(\Omega)$, any cut-off function $\Psi\in C^\infty_0(\Omega,[0,1])$ satisfying \COF$_K$, any real constant $\alpha$, and any $\delta \in (0,1)$, we have
\begin{equation}\label{GNSE.eq00}
    \left| \iint_\Omega |\partial_v u-\alpha|^2\partial_r(\Psi^3) \frac{dtdx}{|t|^{n-d-1}} \right| \leq  \delta \|\wt N(\partial_v u-\alpha|\Psi^3)\|^2_2 + C(1+\delta^{-1}K) \|S(\overline{\nabla} u|\Psi)\|^2_2,
\end{equation}
where $C>0$ depends only on $n$.
\end{Lemma}

\begin{proof}
To lighten the notation, we write $V$ for $\partial_v u$. First, by the integration by parts (Proposition \ref{IBBdrdphi}), we have
\[\mathcal T: = \iint_\Omega |V-\alpha|^2\partial_r (\Psi^3) \frac{dtdx}{|t|^{n-d-1}}
= -2\iint_\Omega (V-\alpha)(\partial_rV)\Psi^3\frac{dtdx}{|t|^{n-d-1}}\]
We introduce $1 = \partial_r |t|$, and we proceed to another integration by parts in order to write
\begin{multline*}
    \mathcal{T} =-2\iint_\Omega (V-\alpha)(\partial_r V)\Psi^3 (\partial_r|t|) \frac{dtdx}{|t|^{n-d-1}}
    = 2\iint_\Omega |\partial_r V|^2\Psi^3\frac{dtdx}{|t|^{n-d-2}} \\
    +2\iint_\Omega (V-\alpha)(\partial_r V) \partial_r(\Psi^3)\frac{dtdx}{|t|^{n-d-2}}
    +2\iint_\Omega (V-\alpha)(\partial_r^2 V)\Psi^3\frac{dtdx}{|t|^{n-d-2}}
    :=\RN{1}+\RN{2}+\RN{3}.
\end{multline*}
Thanks to \eqref{S=int2}, the term $\RN{1}$ is bounded by the square function $\|S(V|\Psi^3)\|^2_2$. 
Since $\Psi$ satisfies \COF$_K$, the cut-off function $|t|\nabla \Psi \in CM(K)$ and so the Cauchy-Schwarz inequality and the Carleson inequality \eqref{Carleson} imply
\begin{align*}
    \RN{2}\leq C K^\frac12 \|\wt N(V-\alpha|\Psi^3)\|_2 \|S(V|\Psi)\|_2 \leq \frac{\delta}3 \|\wt N(V-\alpha|\Psi^3)\|^2_2+ C\delta^{-1} K \|S(\overline\nabla u|\Psi)\|^2_2
\end{align*}
for any $\delta \in (0,1)$. As for the term $\RN{3}$, we have
\begin{multline*}
    \RN{3}= 2 \iint_\Omega (V-\alpha)\Big ([\partial_r^2, \partial_v]u\Big )\Psi^3\frac{dtdx}{|t|^{n-d-2}} + 2 \iint_\Omega (V-\alpha)(\partial_v \partial^2_r u) \Psi^3\frac{dtdx}{|t|^{n-d-2}} :=\RN{3}_1+\RN{3}_2.
\end{multline*}
Since $\partial_v |t| = 0$ whenever $\partial_v = \partial_x$ or $\partial_v = \partial_\varphi$, an integration by parts yields that
\[  \RN{3}_2= - 2 \iint_\Omega (\partial_v^2 u) (\partial^2_r u) \Psi^3\frac{dtdx}{|t|^{n-d-2}} - 2 \iint_\Omega (V - \alpha) (\partial^2_r u) \partial_v (\Psi^3) \frac{dtdx}{|t|^{n-d-2}} := \RN{3}_{21} + \RN{3}_{22}.\]
The term $\RN{3}_{21}$ is easily bounded by $C\|S(\overline \nabla u|\Psi^3)\|_2^2$, and similarly to $\RN{2}$, since $|t||\partial_v \Psi| \in  CM(K)$, we have that
\[|\RN{3}_{22}| \leq  \frac{\delta}3 \|\wt N(V-\alpha|\Psi^3)\|^2_2+ C\delta^{-1} K \|S(\overline\nabla u|\Psi)\|_2^2.\]  

It remains to bound $\RN{3}_1$. Since $\partial_x$ and $\partial_r$ commute, the commutator $[\partial^2_r,\partial_x]$ is zero, and hence - when $\partial_v = \partial_x$ - we have $\RN{3}_1 = 0$. Using Proposition \ref{PPLL1} multiple times gives that 
\[[\partial_r^2, \partial_\varphi] = \partial_r [\partial_r, \partial_\varphi] + [\partial_r,\partial_\varphi]\partial_r = -\frac{2}{|t|}\partial_r \partial_\varphi.\]
So when $\partial_v = \partial_\varphi$, we have 
\begin{multline*}
\RN{3}_1 = - 4 \iint_\Omega (\partial_\varphi u - \alpha) (\partial_r \partial_\varphi u) \Psi^3 \frac{dtdx}{|t|^{n-d-1}}  \\
= 4 \iint_\Omega (\partial_r \partial_\varphi u )  (\partial_\varphi u) \Psi^3 \frac{dtdx}{|t|^{n-d-1}} + 4 \iint_\Omega (\partial_\varphi u -\alpha)  (\partial_\varphi u) \partial_\varphi (\Psi^3) \frac{dtdx}{|t|^{n-d-1}} := \RN{3}_{11} + \RN{3}_{12}
\end{multline*}
by the integration by part given in Proposition \ref{IBBdrdphi}. Observe that Proposition \ref{PDECC} and \eqref{S=int2} infer that
\begin{equation} \label{AphiSphi}
\iint_\Omega |\partial_\varphi u|^2 \Phi \frac{dt\, dx}{|t|^{n-d}} \leq C \|S(\partial_\varphi u |\Phi)\|_{2}^2
\end{equation}
So by the Cauchy-Schwarz inequality, we have
\[\RN{3}_{11} \leq \|S(\partial_\varphi u |\Psi^3)\|_{2} \left( \iint_\Omega |\partial_\varphi u|^2 \Psi^3 \frac{dt\, dx}{|t|^{n-d}} \right)^\frac{1}{2}  \lesssim \|S(\overline\nabla  u |\Psi^3)\|_{2}^2\]
and similarly to $\RN{2}$ and $\RN{3}_{22}$, 
\begin{multline*}
\RN{3}_{12} \leq  \delta \|\wt N(\partial_\varphi u-\alpha|\Psi^3)\|^2_2+ C(\delta K)^{-1} \left( \iint_\Omega |\partial_\varphi u|^2 \Psi \frac{dt\, dx}{|t|^{n-d}} \right)^\frac{1}{2}  \\ 
\leq  \frac{\delta}3 \|\wt N(\partial_\varphi u-\alpha|\Psi^3)\|^2_2+ C\delta^{-1}  K\|S(\overline\nabla  u |\Psi)\|_{2}^2.
\end{multline*}
The lemma follows.
\end{proof}

Now, we prove the analogue of the previous lemma for the radial derivative, and we shall use that $u$ is solution to $Lu=0$.

\begin{Lemma}\label{LLRAD}
Let $L$ be an elliptic operator satisfying \HH$_{\lambda,\kappa}$. 
For any weak solution $u\in W^{1,2}_{loc}(\Omega)$ to $Lu=0$, any cut-off function $\Psi\in C^\infty_0(\Omega,[0,1])$ satisfying \COF$_K$, any real constant $\alpha$, and any $\delta \in (0,1)$, we have
\begin{multline*}
\Big |\iint_{\Omega}|\partial_r u-\alpha|^2\partial_r(\Psi^3)\frac{dtdx}{|t|^{n-d-1}}\Big |\leq \delta \|\wt N(\partial_r u-\alpha|\Psi^3)\|^2_2 + C(1+\delta^{-1}K^2) \kappa \|\wt N(\nabla u|\Psi)\|^2_2 \\ + C(1+\delta^{-1}K^2) \|S(\overline{\nabla} u|\Psi)\|^2_2,
\end{multline*}
and  
\begin{equation} \label{bdBr1}
\Big |\iint_{\Omega}|\partial_r u|^2\partial_r(\Psi^3)\frac{dtdx}{|t|^{n-d-1}}\Big |\leq (\delta+\delta^{-1}K^2\kappa) \|\wt N(\partial_r u|\Psi)\|^2_2 + C(1+\delta^{-1}K^2) \|S(\overline{\nabla} u|\Psi^3)\|^2_2,
\end{equation}
$C$ depends only on $\lambda$, $d$, and $n$. 
\end{Lemma}

\begin{proof}
We only prove the first bound, since \eqref{bdBr1} is established with the same computations, by simply shifting switching the place of $\Psi$ and $\Psi^3$ when we bound $|\I_3| + |\I_5|$ below. By integration by parts (see Proposition \ref{IBBdrdphi}), we have
\[\mathcal T:= \iint_{\Omega}|\partial_r u -\alpha|^2\partial_r(\Psi^3)\frac{dtdx}{|t|^{n-d-1}} = - 2 \iint_{\Omega} (\partial_r u - \alpha)(\partial_r^2 u ) \Psi^3\frac{dtdx}{|t|^{n-d-1}}\]
But now, we can use the equation in cylindrical coordinate, that is \eqref{eqPOPER}, to obtain
\[\partial^2_r u = - \diver_x \A_1 \nabla_x u + \diver_x \A_2 \partial_r u - \frac12 \sum_{i,j=d+1}^n \partial_{\varphi_{ij}}^2 u\]
and then
\begin{multline*}
\mathcal T = 2 \iint_{\Omega} (\partial_r u - \alpha)(\diver_x \A_1 \nabla_x u) \Psi^3\frac{dtdx}{|t|^{n-d-1}} + 2 \iint_{\Omega} (\partial_r u - \alpha)(\diver_x \A_2 \partial_r u) \Psi^3\frac{dtdx}{|t|^{n-d-1}} \\
+ \sum_{i,j=d+1}^n  \iint_{\Omega} (\partial_r u - \alpha)(\partial_{\varphi_{ij}}^2 u) \Psi^3\frac{dtdx}{|t|^{n-d-1}} := \RN{1} + \RN{2} + \RN{3}.
\end{multline*}
We first deal with $\III$, which is easier. Since $\Psi$ and $|t|$ are radial, $\partial_{\varphi_{ij}} (|t|^{d+1-n} \Psi^3) = 0$ and thus, thanks to integration by parts, $\III$ becomes
\[\III = - \sum_{i,j=d+1}^n  \iint_{\Omega} (\partial_{\varphi_{ij}} \partial_r u)  (\partial_{\varphi_{ij}} u) \Psi^3\frac{dtdx}{|t|^{n-d-1}} \lesssim  \|S(\overline{\nabla} u|\Psi^3)\|_2^2\]
by the Cauchy-Schwarz inequality and then \eqref{AphiSphi}. The terms $\I$ and $\II$ are similar. We write $\A_{1,2} \nabla_{x,r} u$ for $\A_{1} \nabla_x u + \A_2 \partial_r u$, and by using the fact that $\partial_r |t| = 1$, we get
\[\I + \II = 2 \iint_{\Omega}   (\partial_r u - \alpha) (\diver_x \A_{1,2} \nabla_{x,r} u) \, \Psi^3 \, \partial_r (|t|) \frac{dtdx}{|t|^{n-d-1}}.\]
So with an integration by parts to move the derivative $\partial_r$ away from $|t|$, we have
\begin{multline*}
\I + \II =  - 2 \iint_{\Omega}   (\partial_r u -\alpha) (\partial_r \diver_x \A_{1,2} \nabla_{x,r} u) \, \Psi^3 \, \frac{dtdx}{|t|^{n-d-2}} \\ 
- 2 \iint_{\Omega}   (\partial^2_r u) (\diver_x \A_{1,2} \nabla_{x,r} u) \, \Psi^3 \, \frac{dtdx}{|t|^{n-d-2}} \\ 
- 2 \iint_{\Omega}   (\partial_r u - \alpha) (\diver_x \A_{1,2} \nabla_{x,r} u) \, \partial_r (\Psi^3) \, \frac{dtdx}{|t|^{n-d-2}} = \I_1 + \I_2 + \I_3.
\end{multline*}
The integrate further by parts  in $\I_1$ to move the $\diver_x$ away from $\partial_r \A_{1,2} \nabla_{x,r} u$ (note beforehand that $\partial_r$ and $\diver_x$ commute), and we obtain
\begin{multline*}
\I_1 = 2 \iint_{\Omega}   (\nabla_x \partial_r u) \cdot  (\partial_r \A_{1,2} \nabla_{x,r} u) \, \Psi^3\frac{dtdx}{|t|^{n-d-2}} \\ + 2 \iint_{\Omega}  (\partial_r u - \alpha) \, (\partial_r \A_{1,2} \nabla_{x,r} u) \cdot \nabla_x (\Psi^3)  \, \frac{dtdx}{|t|^{n-d-2}}
:= \I_4 + \I_5.
\end{multline*}
So it remains to bounds $\I_2$, $\I_3$, $\I_4$, and $\I_5$. The terms $\I_2$ and $\I_4$ are similar, in that 
 \[\begin{split}
|\I_2| + |\I_4| & \lesssim \iint_{\Omega}   |\nabla \overline{\nabla} u|  |\nabla \A_{1,2} \nabla_{x,r} u| \, \Psi^3\frac{dtdx}{|t|^{n-d-2}}  \\
& \lesssim \iint_{\Omega}   |\nabla \overline{\nabla} u|  |\nabla \A_{1,2}| |\nabla_{x,r} u| \, \Psi^3\frac{dtdx}{|t|^{n-d-1}}  + \iint_{\Omega} |\A_{1,2}| |\nabla \overline{\nabla} u|^2 \, \Psi^3\frac{dtdx}{|t|^{n-d-2}}  \\
\end{split}\]
We use the boundedness of $\A_{1,2}$ and \eqref{S=int2} to get that the last term in the right-hand side is bounded by $\|S(\overline{\nabla} u|\Psi)\|_2^2$. As the first term in the right-hand side above, we use the inequality $ab \leq a^2 + b^2/4$, the fact that $\nabla \A_{1,2} \in CM(\kappa)$, and the Carleson inequality \eqref{Carleson} to bound it by $C\kappa \|\wt N(\nabla u|\Psi^3)\|_2^2 + C \|S(\overline \nabla u | \Psi)\|_2^2$. Altogether,
\[|\I_2| + |\I_4| \lesssim \kappa \|\wt N(\nabla u|\Psi^3)\|_2^2 + \|S(\overline \nabla u | \Psi^3)\|_2^2.\]
The terms $\I_3$ and $\I_5$ are also similar, in that they are bounded as follows
 \[\begin{split}
|\I_3| + |\I_5| & \lesssim \iint_{\Omega}   |\partial_r u- \alpha|  |\nabla \A_{1,2} \nabla_{x,r} u| \, |\nabla \Psi^3|\frac{dtdx}{|t|^{n-d-2}}  \\
& \lesssim \iint_{\Omega}   |\partial_r u- \alpha| |\nabla \A_{1,2}| |\nabla_{x,r} u| \,  |\nabla \Psi^3| \frac{dtdx}{|t|^{n-d-2}}  + \iint_{\Omega} |\A_{1,2}|  |\partial_r u- \alpha| |\nabla \overline{\nabla} u| \, |\nabla \Psi^3| \frac{dtdx}{|t|^{n-d-1}}  \\
& \lesssim \delta \|\wt N(\partial_r u- \alpha|\Psi^3)\|_2^2 + \delta^{-1}K^2 \kappa \|\wt N(\nabla u|\Psi)\|_2^2 + \delta^{-1} K^2 \|S(\overline{\nabla} u|\Psi)\|_2^2
\end{split}\]
by using the inequality $ab \leq \delta a^2 + b^2/4\delta$, the Carleson inequality \eqref{Carleson}, the fact that $\Psi$ satisfies $|\nabla \Psi| \leq K/|t|$ and $\1_{\supp \nabla \Psi} \in CM(K)$, and the fact that $|\nabla \A_{1,2}|^2 \leq \kappa$ by \eqref{PpCAL}. The lemma follows.
\end{proof}

In the following, we summarize the results from Lemma \ref{lmGNSE} and Lemma \ref{LLRAD}. Before stating the precise result, we should introduce a notation first. We write $|\nabla u-\vec \alpha|^2$ for a sum of $|\nabla_x u-\vec \alpha_\parallel|^2$, $|\nabla_\varphi u-\vec \alpha_\varphi|^2$, and $|\partial_r u-\vec \alpha_r|^2$, where $\vec \alpha_x$, $\vec \alpha_\varphi$, and $\vec \alpha_r$ are different components of constant vector $\vec \alpha$ corresponding to $\nabla_x$, $\nabla_\varphi$, and $\partial_r$ respectively.

\begin{Lemma}\label{INLSFT}
Let $L$ be an elliptic operator satisfying \HH$_{\lambda,\kappa}$. 
For any weak solution $u\in W^{1,2}_{loc}(\Omega)$ to $Lu=0$, any cut-off function $\Psi\in C^\infty_0(\Omega,[0,1])$ satisfying \COF$_K$, any real constant $\alpha$, and any $\delta \in (0,1)$, we have
\begin{multline}\label{eqINLST01}
\Big |\iint_{\Omega}|\nabla u-\vec \alpha|^2\partial_r\Psi^3\frac{dtdx}{|t|^{n-d-1}}\Big |\leq \delta \|\wt N(\partial_r u-\alpha|\Psi^3)\|^2_2 + C(1+\delta^{-1}K^2) \kappa \|\wt N(\nabla u|\Psi)\|^2_2 \\ + C(1+\delta^{-1}K^2) \|S(\overline{\nabla} u|\Psi)\|^2_2,
\end{multline}
$C$ depends only on $\lambda$, $d$, and $n$.
\end{Lemma}
\begin{proof}
Immediate from Lemma \ref{lmGNSE} and Lemma \ref{LLRAD}.
\end{proof}

\section{\texorpdfstring{$N\leq S$}{TEXT} Local Estimates, Part 2: the Good Lambda Argument} \label{Sgoodlambda}

The main goal of this section is to establish the ``good-lambda'' distributional inequality, that will give the desired $N<S$ estimate. 

\medskip

In this section, a boundary ball (a ball in $\R^d$) with center $x$ and radius $l$ will be written $B_l(x)$. First, we recall several results from \cite{feneuil2018dirichlet}. 
Let $h_{\beta}:\, \mathbb{R}^d\rightarrow \mathbb{R}$ be a function such that for any compactly supported and continuous function $v$,
\begin{align*}
h_{\beta}(v)(x):=\inf\Big\{r>0,\sup_{(y,s)\in \Gamma(x,r)}|v(y,s)|<\beta\Big\},
\end{align*}
where $\Gamma(x,r)\subset \mathbb{R}^{d+1}_+$ is defined as the translation of the cone $\Gamma(0)$ with vertex at $(x,r)$. 

\begin{Lemma}[Lemma 6.1 in \cite{feneuil2018dirichlet}] \label{LEMA21}
For any $v$ such that $h_{\beta}(v)<\infty$, the map $h_{\beta}(v)$ is a $1$-Lipschitz function. 
\end{Lemma} 

\begin{Lemma}[Lemma 6.2 in \cite{feneuil2018dirichlet}] \label{LEMA22}
Let $v\in L^2_{loc}(\Omega)$ and $\Psi$ be a smooth function which satisfies $0\leq \Psi\leq 1$. Set $h_\beta:= h_\beta((v|\Psi^3)_W)$. There exists a small constant $c>0$ depending only on $d$ and $n$ such that for any $\beta>0$ and $\wt{N}(v|\Psi^3)(x)>\beta$, we have:
\begin{align*}
\mathcal{M}\Big [\Big (\dashint_{y\in B_{h_{\beta}(.)/2}(.)}\int_{s\in \mathbb{R}^{n-d}}|v|^2\Psi^3\partial_r[\chi_{\beta}^3]\frac{ds}{|s|^{n-d-1}}dy\Big )^{\frac{1}{2}}\Big ](x)\geq c\beta,
\end{align*}
where $\chi_{\beta}$ is a cut-off function defined as $\chi_{\beta}(y,.)\equiv 0$ if $h_\beta (y)=0$ and
\begin{align*}
\chi_{\beta}(y,s):=\phi\Big (\frac{|s|}{h_{\beta}(y)}\Big ),\ \ \text{ with }\ \ \phi(r)&:=
\begin{cases}
0        & \text{if } 0\leq r<1/5, \\
(25-5r)/24 & \text{ if } 1/5 \leq r \leq 5, \\
1       & \text{if }  r>5
\end{cases}
\end{align*}
otherwise. 
\end{Lemma}

The two above lemmas are analogues to results from \cite{kenig2000new} and \cite{dindovs2019regularity} adapted to our setting and to the use of cut-off functions $\Psi$. Let us first introduce some specific cut-off functions.

\begin{Def} \label{DECUTO}
Let $\phi\in C_0^\infty(\mathbb{R})$ be a non-increasing function such that $\phi\equiv 1$ on $[0,1]$ and $\phi\equiv 0$ on $[2,\infty)$. We define the cut-off functions on $\Omega$ as 
\[\Psi_e(y,t) :=  \phi\Big (\frac{e(x)}{|t|}\Big ) \1_\Omega(x,t) \]
 if $x\to e(x)\geq0$ is a 1-Lipschitz function, in particular,
\[\Psi_\epsilon(y,t) :=  \phi\Big (\frac{\epsilon}{|t|}\Big ) \1_\Omega(x,t) \]
if $\epsilon>0$. Also, let us denote 
\[\Psi_B(y,t) := \phi\Big (\frac{\dist(y,B)}{100|t|}\Big ) \1_\Omega(x,t)\]
if $B$ is a boundary ball.
Moreover, we write $\Psi_{B,l,\epsilon}$ for the product $\Psi_B(1-\Psi_{2l})\Psi_\epsilon$.  \end{Def}

Note that from the fact that $\phi$ is non-increasing, for any (non-negative) $1$-Lipschitz function $e$, we have
\begin{equation} \label{drPsie>0}
\partial_r \Psi_{e} \geq 0.
\end{equation}
The proof of next lemma is easy but can nevertheless be found after Lemma 4.5 in \cite{feneuil2018dirichlet}.

\begin{Lemma}[\cite{feneuil2018dirichlet}] \label{LECUTO}
There exists a uniform $K$ that depends only on $d$ and $n$ such that the functions $\Psi_e$ and their ``complements'' $1-\Psi_e$ satisfy \COF$_K$. 
Since $\Psi_\epsilon$ and $\Psi_B$ are particular cases of $\Psi_e$, then (of course) they also satisfy \COF$_K$ with the same uniform constant $K$. In addition,  the property \COF$_K$ is stable under the product, in the sense that if $\Psi$ satisfies \COF$_{K_1}$ and $\Phi$ satisfies \COF$_{K_2}$, then $\Psi\Phi$ satisfies \COF$_{K_1+K_2}$.
\end{Lemma}

We state the precise statement of the ``good-lambda'' distributional inequality that we will need in the following. 

\begin{Lemma}\label{GLAM}
Let $L$ be an elliptic operator satisfying \HH$_{\lambda,M,\kappa}$. There exists $\eta\in (0,1)$ that depends only on $d$ and $n$ and $C>0$ that depends on $\lambda$, $d$ and $n$  such that the following holds.

For any a weak solution $u\in W^{1,2}_{loc}(\Omega)$ to $Lu=0$, any cut-off function in the form $\Psi:=\Psi_{B,l,\epsilon}$ for some $\epsilon>0$, some $l>100\epsilon$, and some boundary ball $B$ of radius $l$, and for any triplet $\beta>0$, $\gamma>0$, $\delta\in (0,1)$, we have
\begin{align}\label{eqGLMN1}
|\{x\in \mathbb{R}^d, \wt{N}(\nabla u|\Psi^{3})(x)>\beta\}\cap E_{\beta,\gamma,\delta}|
\leq C\gamma^2|\{x\in \mathbb{R}^d,\mathcal{M}[\wt{N}(\nabla u|\Psi^{3})](x) >\eta \beta\}|,
\end{align}
where
\begin{multline*}
E_{\beta,\gamma,\delta}:=\bigg \{x\in \mathbb{R}^d: \mathcal M\Big[\Big (\dashint_{B_{l}(.)}\dashint_{l\leq |s|\leq 2l}|\nabla u|^2 \Psi_B^3 \, ds\, dy\Big )^{1/2}\Big](x) + \delta^{1/2} \mathcal{M}[\wt{N}(\nabla u|\Psi^3)](x) \\
+ \delta^{-1/2} \kappa^{1/2}  \mathcal{M}[\wt{N}(\nabla u|\Psi^3)](x) 
+ \delta^{-1/2} \mathcal{M}[S(\overline\nabla u|\Psi)](x)
\leq \gamma\beta\bigg \}.
\end{multline*}

\end{Lemma}
\begin{proof}
\textbf{Step 1: The Whitney decomposition.} We fix $\beta,\delta>0$ and we take a ball $B\subset \mathbb{R}^d$ with radius $l>0$. Define
\[\mathcal{E}:=\{x\in \mathbb{R}^d,\mathcal{M}[\wt{N}(\nabla u|\Psi^3)](x)>\eta\beta\}.\]
We notice that $(\nabla u|\Psi^3)_{W,a}$ is continuous and $\Psi$ is compactly supported. Hence $\mathcal{E}$ is open and bounded. We pick a ball $B_r(x)$ of radius $r:=\text{dist}(x,\mathcal{E}^c)/10$ centered at $x\in \mathcal{E}$. Under this construction, $\mathcal{E}=\bigcup_{i\in J}B_{r_i}(x_i)$ and $\sup_{i\in J}r_i<\infty$. By Vitali covering lemma, there exists a countable subcollection of balls $\{B_{r_i}(x_i)\}_{i\in I}$, which are disjoint and satisfy that $\mathcal{E}\subseteq \bigcup_{i\in I}B_{5r_i}(x_i)$. For each $i\in I$, we set $B_i:=B_{10r_i}(x_i)$ and thus there exists a 
\begin{align}\label{eqGLM00}
\text{$y_i\in \overline{B_i}\cap \mathcal{E}^c$, in particular $\mathcal{M}[\wt{N}_a(\nabla u|\Psi^{3})](y_i)\leq\eta \beta$.}
\end{align}
We define the set $F_\beta^i$ such that
\begin{align*}
F^i = F^i_{\beta,\gamma,\delta}:=\{x\in \overline{B_i}, \, \wt{N}_a(\nabla u|\Psi^{3})(x)>\beta\}\cap E_{\beta,\gamma,\delta}.
\end{align*}
It suffices to prove that for each $i\in I$, 
\begin{align}\label{eqGLMAIN}
|F^i|\lesssim C\gamma^2|B_i|
\end{align}
because $\sum_{i\in I}|B_i|\leq 10^d\sum_{i\in I}|B_{r_i}(x_i)|\leq 10^d|\mathcal{E}|. $
The inequality (\ref{eqGLMAIN}) is trivial when $F_\beta^i=\emptyset$. Hence we assume that $F_\beta^i \supset \{x_i\}$ is non-empty in the sequel of the proof. 

\medskip

\textbf{Step 2: Localization of $\wt N(\nabla u|\Psi^3)$ in $B_i$.} In this step, we show that if $x\in F^i$, then $\wt{N}(\nabla u|\Psi^{3})(x)$ has to reach its maximum value at a point $(z,r) \in \Gamma(x)$ verifying $r\leq r_i$. 
Indeed, take $x\in F^i$ and then $(z,r) \in \Gamma(x)$ such that $r> r_i$. Notice that $(z,r) \in \bigcup_{y\in B_{r}(z)} \Gamma(y)$, so
\begin{align}\label{step2.eq00}
(\nabla u|\Psi^{3})_{W}(z, r) \leq \wt{N}(\nabla u|\Psi^{3})(y)\ \ \text{for all $y\in B_{r}(z) \subset B_{20r}(y_i)$}. 
\end{align}
Therefore, for a constant $C$ that depends only on $d$,
\[(\nabla u|\Psi^{3})_{W}(z, r) \leq \mathcal M[\wt{N}(\nabla u|\Psi^{3})](y_i) \leq C\eta\beta < \beta\] 
by \eqref{eqGLM00}, if $\eta$ is small enough (depending only on $d$). So it means that for any $x\in F^i$, we have
 \begin{equation} \label{NlocBi}
 \beta < \wt{N}(\nabla u|\Psi^{3})(x) = \sup_{(z,r) \in \Gamma(x)} \1_{r\leq r_i} (\nabla u|\Psi^{3})_{W}(z, r) \qquad \text{ for } x\in F^i.
 \end{equation} 
We construct the cut-off function $\Phi_i(y,s):=(1-\Psi_{K_ir_i}) \Psi_{F^i}$ where $\Psi_{K_ir_i}:= \Psi_{\epsilon_i}$ for $\epsilon_i:= K_ir_i$ and $\Psi_{F^i} := \Psi_{e^i}$ for the $1$-Lipschitz function $e^i(x):= \dist(y,F^i)/M_i$. The constants $10$ and $K_i$ in the construction of $\Phi_i$ are large enough so that $\Phi_i(y,t) = 1$ whenever $(y,t) \in W(z,r)$ for $(z,r) \in \Gamma(x) \cap \{r\leq r_i\}$ and $x\in F_i$. With such a choice and by \eqref{NlocBi}, we have that
 \begin{equation} \label{NlocBi2}
 \wt{N}(\nabla u|\Psi^{3}\Phi_i^3)(x) = \wt{N}(\nabla u|\Psi^{3})(x) > \beta \qquad \text{ for } x\in F^i.
 \end{equation} 
We let a little bit of freedom on the choice of $K_i$ to avoid some future complication. Notice that 
 \begin{equation} \label{defSi}
\supp (\nabla\Psi_{K_ir_i}) \cap \supp (\Psi_{F_i}) \subset S_i:= (K_i+1)B_i \times \{s\in \R^{n-d}, \, K_ir_i/2 < |s| < K_ir_i\}.
\end{equation}
We first try $K_i = 4$, which is large enough for \eqref{NlocBi2} to be satisfied. If $S_i$ intersects $\{\Psi \neq 0\} \cap \{\Psi \neq 1\}$, then we test $K_i = 8$ instead. If $S_i$ still intersects $\{\Psi \neq 0\} \cap \{\Psi \neq 1\}$, we multiply $K_i$ by 2 and we stop at the first time when
 \begin{equation} \label{prSi}
S_i \cap \{\Psi \neq 0\} \cap \{\Psi \neq 1\} = \emptyset, \ \text{ i.e. either } \ S_i \subset \{\Psi \equiv 1\} \text{ or } S_i \subset \{\Psi \equiv 0\}.
\end{equation}
Since $\Psi = \Psi_{B,\epsilon}$ is constructed from the product of three cut-off function $\Psi_e$ where $e$ is either constant or a a slowly growing $1/100$-Lipschitz function,  while $\Psi_{F_i}$ is constructed with a faster growing $1/10$-Lipschitz function, $K_i$ can only take a uniformly finite number of values (i.e. we think that $K_i \leq 2^{7}$ and we say that $K_i \leq 2^{10}$ to have some error margin).

\medskip

\textbf{Step 3: Catching the level sets of $\wt N(\nabla u|\Psi^3)$.} 
Let $h_\beta:=h_\beta((\nabla u|\Psi^3\Phi_i^3)_{W})$. Lemma \ref{LEMA22} and \eqref{NlocBi2} entails that
\begin{align} \label{step3lambda1}
c\beta\leq \mathcal{M}\bigg [\bigg (\dashint_{y\in B_{h_{\beta}(.)/2}(.)}\int_{s\in \mathbb{R}^{n-d}}|\nabla  u|^2\Psi^3\Phi^3_i\partial_r[\chi_{\beta}^3]\frac{dsdy}{|s|^{n-d-1}}\bigg )^{\frac{1}{2}}\bigg ](x) \qquad \text{ for } x\in F^i.
\end{align}
We know from \eqref{prSi} that either $S_i \subset \{\Psi \equiv 1\}$ or $S_i \subset \{\Psi \equiv 0\}$. We set
\begin{align}\label{eqGL04}
\vec \alpha_i:= \fiint_{S_i} (\nabla u) \Psi^{3/2} \, dy\, ds = \left\{\begin{array}{l} \fiint_{S_i} (\nabla u) \, dy\, ds \quad \text{ if } S_i \subset \{\Psi \equiv 1\} \\ 0 \quad \text{ otherwise}  \end{array} \right.
\end{align}
and we want to show that $|\vec \alpha_i|$ is smaller than $c\beta/2$, where $c$ is the constant in \eqref{step3lambda1}. We select $N$ points $\{z_j\}_{j=1}^{N} \in 2K_iB_i$ such that $S_i\subset \bigcup_{j=1}^{N} W(z_j, K_ir_i)$. We can always to so with a uniformly bounded number $N$ of points, because $K_i$ is itself uniformly bounded (between 4 and $2^{10}$). 
So we easily have by simply using the definition of $\vec \alpha_i$, $(\nabla u|\Psi^3)_W$, $\wt N(\nabla u|\Psi^3)$ and then \eqref{eqGLM00} that
\begin{multline}\label{eqGL06}
|\vec \alpha_i| \leq C \sum_{j=1}^N (\nabla u|\Psi^3)_{W}(z_j, K_ir_i)\leq C \sum_{j=1}^N \dashint_{B_{K_ir_i}(z_j)}\wt{N}(\nabla u|\Psi^3) dx \\
\leq C' \dashint_{30K_iB_i}\wt{N}_a(\nabla u|\Psi^3)dx \leq C' \mathcal{M}[\wt{N}(\nabla u|\Psi^3)](y_i)\leq C' \eta\beta \leq c\beta/2,
\end{multline}
if $\eta$ is small enough (depending only on $d$). The combination of \eqref{step3lambda1} and \eqref{eqGL06} infers that
\begin{align} \label{step3lambda}
c\beta/2 \leq \mathcal{M}\bigg [\bigg (\dashint_{y\in B_{h_{\beta}(.)/2}(.)}\int_{s\in \mathbb{R}^{n-d}}|\nabla  u - \vec \alpha_i|^2\Psi^3\Phi^3_i\partial_r[\chi_{\beta}^3]\frac{dsdy}{|s|^{n-d-1}}\bigg )^{\frac{1}{2}}\bigg ](x) \ \text{ for } x\in F^i.
\end{align}
\medskip

\textbf{Step 4: From a pointwise estimate to integral estimates.}
The result \eqref{step3lambda} from the previous step implies that
\begin{multline*}
|F^i|\lesssim 
\frac{1}{\beta^2}\bigg \|\mathcal{M}\bigg [\bigg (\dashint_{y\in B_{h_{\beta}(.)/2}(.)}\int_{s\in \mathbb{R}^{n-d}}|\nabla u-\vec \alpha_i|^2\Psi^3\Phi^3_i\partial_r[\chi_{\beta}^3]\frac{dsdy}{|s|^{n-d-1}}\bigg )^{\frac{1}{2}}\bigg ]\bigg \|^2_{2}\\
\lesssim \frac{1}{\beta^2}\int_{\mathbb{R}^d}\dashint_{y\in B_{h_{\beta}(x)/2}(x)}\int_{s\in \mathbb{R}^{n-d}}|\nabla u-\vec \alpha_i|^2\Psi^3\Phi^3_i\partial_r[\chi_{\beta}^3]\frac{dsdy}{|s|^{n-d-1}}dx
\end{multline*}
thanks to $L^2$ boundedness of the Hardy-Littlewood maximal operator $\mathcal M$. According to Lemma \ref{LEMA21}, the function $h_\beta$ is 1-Lipschitz , that is, $|h_\beta(x)-h_{\beta}(y)|\leq |x-y|$. If $y\in B_{h_{\beta}(x)/2}(x)$, then the Lipschitz condition implies that $|h_{\beta}(x)-h_{\beta}(y)|\leq |x-y|\leq h_{\beta}(x)/2$ and thus $h_{\beta}(x)/2\leq h_{\beta}(y)\leq 3h_\beta(x)/2$. Consequently, by Fubini's theorem,
\begin{multline*}
|F^i|\lesssim \frac{1}{\beta^2}\iint_{\Omega}|\nabla u  -\vec \alpha_i|^2\Psi^3\Phi^3_i\partial_r[\chi_{\beta}^3]\frac{dsdy}{|s|^{n-d-1}}\bigg (\int_{x\in B_{h_{\beta}(y)}(y)} h_\beta(x)^{-d} dx\bigg )\\
\lesssim \frac{1}{\beta^2}\iint_{\Omega}|\nabla u-\vec \alpha_i|^2\Psi^3\Phi^3_i\partial_r[\chi_{\beta}^3]\frac{dsdy}{|s|^{n-d-1}}.
\end{multline*}
Recall that
\[\Psi\Phi_i = \Psi_B \Psi_{F^i} \Psi_\epsilon (1-\Psi_{K_ir_i}) (1-\Psi_{2l}).\]
By \eqref{drPsie>0}, $\partial_r [\Psi_e \Psi_{F^i} \Psi_\epsilon] \geq 0$ and thus the product rule implies 
\begin{align*}
\Psi^3\Phi^3_i\partial_r[\chi_{\beta}^3]\leq \partial_r[\Psi^3\Phi^3_i\chi_{\beta}^3]
+\partial_r[\Psi^3_{K_ir_i}] \Psi^3_{F^i}  \Psi^3 \chi^3_\beta
+ \partial_r[\Psi^3_{2l}] \Psi_B^3 \Psi^3_\epsilon \Phi_i^3 \chi^3_\beta.
\end{align*}
It follows that:
\begin{multline*}
|F^i|\lesssim \frac{1}{\beta^2}\iint_{\Omega} |\nabla u-\vec \alpha_i|^2\partial_r[\Psi^3_{K_ir_i}] \Psi^3_{F^i}  \Psi^3 \chi^3_\beta \frac{dsdy}{|s|^{n-d-1}}\\
+ \frac{1}{\beta^2}\iint_{\Omega} |\nabla u-\vec \alpha_i|^2\partial_r[\Psi^3_{2l}] \Psi_B^3 \Psi^3_\epsilon \Phi_i^3 \chi^3_\beta\frac{dsdy}{|s|^{n-d-1}}\\
+\frac{1}{\beta^2}\iint_{\Omega} |\nabla u-\vec \alpha_i|^3\partial_r[\Psi^3\Phi^3_i\chi_{\beta}^3]\frac{dsdy}{|s|^{n-d-1}}:=\RN{1}+\RN{2}+\RN{3}.
\end{multline*}
In order to prove the claim (\ref{eqGLMAIN}), and hence the lemma, it suffices to show $\RN{1}+\RN{2}+\RN{3}\leq C \gamma^2|B_i|$ with a constant $C$ that depends only on $\lambda$, $d$, and $n$.

\medskip

\textbf{Step 5: We treat $\I$.} 
We recall that $S_i \supset \supp(\partial_r \Psi_{K_ir_i}) \cap \supp(\Psi_{F_i})$, see \eqref{defSi}, and $|\nabla \Psi_{K_ir_i}| \lesssim |t|$ since $\Psi$ satisfies \COF. Therefore, 
\[\RN{1}\lesssim \frac{|B_{i}|}{\beta^2}\fiint_{S_i}|\nabla u-\vec \alpha_i|^2 \Psi^3 \, ds\, dy =  \frac{|B_{i}|}{\beta^2} \1_{S_i \cap \supp \Psi} \fiint_{S_i}|\overline \nabla u-\vec \alpha_i|^2 ds\, dy\]
since we chose $K_i$ so that $\Psi$ is either constant equal to 0 or constant equal to 1 in $S_i$, see \eqref{prSi}, and since changing $\nabla$ to $\overline \nabla$ is just rewriting a vector with a different system of coordinates (and of course we rewrite $\vec \alpha_i$ in this system of coordinates too). 
If $\Psi\equiv 0$ on $S_i$, the bound $I=0 \leq C\gamma^2|B_i|$ is trivial. So we assume for the rest of the step that $\Psi \equiv 1$ on $S_i$. In this case,  since $\vec \alpha_i$ is the average of $\nabla u$ on $S_i$, the Poincar\'e inequality yields that:
\begin{equation}\label{eqGL05}
\RN{1} \lesssim \frac{r_i^2|B_{i}|}{\beta^2} \fiint_{S_i}|\nabla \overline \nabla u |^2 dsdy \lesssim \frac{|B_i|}{\beta^2} \iint_{S_i}|\nabla \overline \nabla u|^2\Psi^3 \frac{ds\, dy}{|s|^{n-2}}
\end{equation}
because $\Psi \equiv 1$ on $S_i$. We adapt the argument that we used to establish \eqref{eqGL06}. We pick a collection of points $\{z_j\}_{j=1}^N \in 2K_iB_i$ such that 
\[S_i \subset \bigcup_{j=1}^N B_{K_ir_i/4}(z_j)\times \{K_ir_i/2< |s| <K_ir_i\}.\] We can choose the collection so that $N$ is uniformly bounded. Since 
\[B_{K_ir_i/4}(z_j)\times \{K_ir_i/2< |s| <K_ir_i\} \subset \widehat \Gamma(x) \ \text{ for } x\in B_{K_ir_i/4}(z_j),\]
we have
\begin{multline} \label{eqGL06z}
\I \lesssim \frac{|B_i|}{\beta^2} \sum_{j=1}^N \iint_{B_{K_ir_i/4}(z_j) \times \{K_ir_i/2< |s| <K_ir_i\}} |\nabla \overline \nabla u|^2\Psi^3 \frac{ds\, dy}{|s|^{n-d-2}}  \\  \lesssim \frac{|B_i|}{\beta^2} \sum_{j=1}^N \left( \fint_{B_{K_ir_i/4}(z_j)} S(\overline \nabla u|\Psi^3)(x) \, dx \right)^2 
\lesssim \frac{|B_i|}{\beta^2} \left( \dashint_{30K_iB_i} S(\overline \nabla u|\Psi^3)(x) \, dx \right)^2 \\ \leq \frac{|B_i|}{\beta^2} \left(\mathcal{M}[S(\overline \nabla u|\Psi)](x_i) \right)^2 \leq \gamma^2 |B_i|,
\end{multline}
where $x_i$ is any point of the non-emptyset $F^i \subset E_{\beta,\gamma,\delta}$ and the last inequality comes from the fact that $x_i \in E_{\beta,\gamma,\delta}$ (we could even have $\I \lesssim \gamma^2 \delta^2 |B_i|$).

\medskip

\textbf{Step 6: We deal with $\II$.} 
Observe that
\[ \supp \{\Psi_B \partial_r \Psi_{2l}\}\subset 500B \times \{ \, l \leq |s|\leq 2l\} \ \text{and } \ \supp \{\Phi_i\} \subset\{\dist(y,B_i)/20 \leq |s|\leq K_ir_i\}\]
since we know that $K_i \leq 2^{10}$. The integral $\II$ is non-zero only if the (interior of the)  two above supports intersect, and in this case, we necessarily have
\begin{equation} \label{prxil}
l < K_ir_i \leq 2^{10}r_i \ \text{ and } \ \dist(500B,F_i) < 40l
\end{equation}
which we know assume. So $500B \subset 2^{20}B_i$ and we can find a boundary point $x_i\in \R^d$ such that 
\begin{equation} \label{defxiz}
x_i \in F_i \cap 550B.
\end{equation}
 By the triangle inequality and the fact that $|\nabla \Psi_{2l}| \lesssim l$ on $\supp (\nabla \Psi_{2l})$, we have
\begin{align*}
\RN{2}\lesssim \frac{|2^{20}B_i|}{\beta^2}\fint_{500B}\fint_{l\leq |s|\leq 2l}|\nabla  u|^2 \Psi_B^3 \, ds\, dy
+\frac{|2^{20}B_i| |\vec \alpha_i|^2}{\beta^2} \Psi_B^3 =\RN{2}_1+\RN{2}_2.
\end{align*}
We want to bound $\II_1$ with the help of the Hardy Littlewood maximal function of $x\to \Big( \fint_{B_l(x)} \fint_{l<|s|<2l} |\nabla y|^2 dsdy \Big)^{1/2}$. So we proceed like we already several times, see around \eqref{eqGL06} and \eqref{eqGL06z}. We take a uniformly finite collection of points $\{z_j\}_{j=1}^N \in 501B$ such that $500B \subset \bigcup B_{l/2}(z_j)$, and since $\fint_{B_{l/2}(z_j)} |\nabla u|^2 \lesssim \fint_{B_l(x)} |\nabla u|^2$ for any $x\in B_{l/2}(z_j) |\nabla u|^2$, we have 
\begin{multline} \label{eqGL06y}
\II_1 \lesssim \frac{|B_i|}{\beta^2} \sum_{j=1}^N \fint_{B_{l/2}(z_j)} \fint_{l<|s|<2l} |\nabla u|^2 \Psi_B^3 ds\, dy  \\ \lesssim \frac{|B_l|}{\beta^2} \sum_{j=1}^N \left( \fint_{B_{l/2}(z_j)} \Big( \fint_{B_{l}(x)} \fint_{l<|s|<2l} |\nabla u|^2 \Psi_B^3 \, ds\, dy \Big)^\frac12 dx \right)^2 \\
\lesssim \frac{|B_i|}{\beta^2} \left( \fint_{B_{2000l}(x_i)} \Big( \fint_{B_{l}(x)} \fint_{l<|s|<2l} |\nabla u|^2 \Psi_B^3 \, ds\, dy \Big)^\frac12 dx \right)^2 \\
 \leq \frac{|B_i|}{\beta^2} \left(\mathcal{M}\Big[\fint_{B_{l}(.)} \fint_{l<|s|<2l} |\nabla u|^2 \Psi_B^3 \, ds\, dy \Big)^\frac12\Big](x_i) \right)^2 \leq \gamma^2 |B_i|,
\end{multline}
because $x_i \in E_{\beta,\gamma,\delta}$. It remains to bound $\II_2$, but that one will be easy. Without loss of generality, we can assume the support of the function $\phi$ used to construct $\Psi_{2l}$ to be exactly $[1,2]$ and hence the support of $\partial_r \Psi_{2l}$ to be exactly $\{l \leq |s| \leq 2l\}$. But the set $S_i$ defined in \eqref{defSi} and used to build $\vec \alpha_i$ has to be included by construction in either $\{\Psi \equiv 1\}$ or $\{\Psi \equiv 0\}$. Combined with \eqref{prxil}, it forces $S_i \subset \{\Psi \equiv 0\}$, and thus $\II_2 = |\vec \alpha_i| = 0$.

\medskip

\textbf{Step 7: We bound $\III$ to conclude.} As discussed at the end of Step 4, we needed to bound $\I$, $\II$, and $\III$ by $C\gamma^2|B_i|$ to finish the proof of the lemma. We already proved the desired estimates of $\I$ and $\II$ in Steps 5 and 6, so it remains to show $\III \lesssim \gamma^2 |B_i|$.

We did not use Section \ref{SN<S} at this point, so as one could expect, it will appear in this last Step. We easily have that
\begin{align}\label{eqGL59b}
    \|\wt{N}(\nabla u-\vec \alpha_i |\Psi^{3}\Phi_i^{3})\|^2_{2}
    \leq \|\wt{N}(\nabla u|\Psi^{3}\Phi^{3}_i)\|^2_{2}+|\vec \alpha_i|^2 \|\wt{N}(1|\Phi_i)\|^2_{2}.
\end{align}
Lemma \ref{LECUTO} shows that $\Psi^3\Phi_i^3\chi_\beta^3$ satisfies \COF with a constant that depends only on $d$ and $n$. Thus we apply Lemma \ref{INLSFT} to the term $\RN{3}$. Together with (\ref{eqGL59b}), we deduce that
\begin{multline*}
\RN{3}\leq \frac{1}{\beta^2}\Big \{\delta \|\wt{N}(\nabla u|\Psi^{3}\Phi_i^{3})\|^2_{2} + \delta |\vec \alpha_i|^2\|\wt{N}(1|\Phi_i)\|^2_{2}  \\
 + C (1+\delta^{-1}K^2) \kappa \|\wt{N}(\nabla u|\Psi\Phi_i)\|^2_{2} +  C(1+\delta^{-1}K^2) \|S(\overline\nabla u|\Psi\Phi_i)\|^2_{2} \Big \}
\end{multline*}
Let $v$ be any function for which $\wt N(v|\Phi_i)$ or $S(v|\Phi_i)$ makes sense, and in this situation, the non-tangential maximal function $\wt N(v|\Phi_i)$ and the square function $S(v|\Phi_i)$ are supported in a ball $CB_i$, where $C$ is universal. Why? Because $\Phi_i$ is supported in a saw-tooth region on top of $F^i\subset B_i$, which is truncated above by $K_ir_i$. Hence the Whitney box $W(z,r)$ for which $(v|\Phi_i)_{W(z,r)} \neq 0$ are such that $r \leq 2K_ir_i$ and then $z\in 10K_iB_i$, which means $\supp N(v|\Phi_i) \subset 10K_iB_i$.
Similarly, a point $(y,t)$ for which $\nabla v(y,t) \neq 0$ are such that $|t|\leq K_ir_i$ and then $y\in 3K_iB_i$, which implies that $\supp S(v|\Phi_i) \subset 3K_iB_i$. Altogether
\begin{equation}\label{eqGL59c}
\supp S(v|\Phi_i) \cup \supp N(v|\Phi_i) \subset 10K_iB_i \subset B^*_i:= 2^{14}B_i.
\end{equation}
With this observation, we have
\begin{multline*}
\RN{3}\lesssim \frac{|B_i|}{\beta^2} \Big \{\delta |\vec \alpha_i|^2 + \delta \big[ \sup_{B^*_i} \wt{N}(\nabla u|\Psi^3\Phi_i^3)\big]^2 + \delta^{-1} \kappa \big[ \sup_{B^*_i} \wt{N}(\nabla u|\Psi\Phi_i)\big]^2 
+ \delta^{-1} \big[\sup_{B^*_i} S(\overline\nabla u|\Psi^{3}\Phi_i^{3})\big] ^2 \Big \} \\
:= \III_1 + \III_2 + \III_3 + \III_4.
\end{multline*}
The three terms above are handled in a similar manner. 
Recall that $\Phi_i$ is supported in a saw-tooth region over $F_i$ truncated at $K_ir_i$. If $(\nabla u|\Psi^3\Phi_i^3)_{W}(z',r')\neq 0$, then $W(z', r')\cap {\rm supp}\{\Phi_i\}\neq \emptyset$ and there exists a $x_i\in F^i \subset B_{i}$ such that $|x_i-z'|\leq 1000 r' \leq 2^{20}r_i$. It follows that for all $(z', r')\in \mathbb{R}^{d+1}_+$,
\begin{multline}\label{eqGLAM1}
(\nabla u|\Psi^3\Phi_i^3)_{W}(z', r')\leq \dashint_{B_{r'}(z')}\wt{N}(\nabla u|\Psi^3\Phi_i^3)(z)dz\\
\lesssim \dashint_{B_{1000r'}(x_i)} \wt{N}(\nabla u|\Psi^3\Phi_i^3)(z)dz
\leq \mathcal{M}[\wt{N}_a(\nabla u|\Psi^3)](x_i).
\end{multline}
Consequently, for each $z\in \mathbb{R}^d$, there exists a $x_i\in F_\beta^i$ such that
\begin{align}\label{eqGLAM22}
\delta \wt{N}(\nabla u|\Psi^3\Phi_i^3)(z)
\lesssim \delta \mathcal{M}[\wt{N}(\nabla u|\Psi^3)](x_i) \leq \gamma \beta
\end{align}
where the last inequality follows from the fact that $x_i \in E_{\beta,\gamma,\delta}$. We easily deduce
\[\III_2:= \frac{\delta|B_i|}{\beta^2} \big[ \sup_{B^*_i} \wt{N}(\nabla u|\Psi^3\Phi_i^3)\big]^2 \lesssim \gamma^2 |B_i|.\]
Similarly, we have
\[\III_3:= \frac{\delta^{-1}\kappa|B_i|}{\beta^2} \big[ \sup_{B^*_i} \wt{N}(\nabla u|\Psi\Phi_i)\big]^2 \lesssim \gamma^2 |B_i|.\]
The term $\III_4$ follows the same lines. If $y\in F^i$, then $S(\overline \nabla u|\Psi\Phi_i)(y) \leq S(\overline \nabla u|\Psi\Phi_i)(y) \leq \gamma \beta$. If $y\notin F_i$, we take $x_i\in F^i$ such that $r_y:= \dist(y,F_i) = |y-x_i|$.
We know from the construction of $\Psi_{F^i}$ that 
\begin{equation} \label{factZZ1}
\Phi_i(z,s) \equiv 0 \ \text{ for $|z-y|< r_i/10$ and $|s| < r_i/400$.}
\end{equation}
We cover $B_{r_i/20}(y)$ by a uniformly finite collection of balls $\{B_{r_i/800}(z_j)\}_{j=1}^N$, and we notice that for any collection $\{w_j\}_{j=1}^N$ of points satisfying $w_j \in B_{r_i/800}(z_j)$, we have
\[\widehat \Gamma(y) \cap \supp \Phi_i \subset  \bigcup_{j=1}^N \widehat \Gamma(w_j).\]
We conclude that
\begin{multline*}
S(\overline \nabla u|\Psi\Phi_i)(y) \leq \sum_{j=1}^N \dashint_{B_{r_i/800}(z_j)} S(\overline \nabla u|\Psi\Phi_i)(w) \, dw \\
\lesssim \dashint_{B_{2r_i}(x_i)} S(\overline \nabla u|\Psi)(w) \, dw \leq \mathcal M \Big[ S(\overline \nabla u|\Psi) \Big](x_i)  \leq \frac{\delta^{1/2}} \gamma\beta.
\end{multline*}
and then $\III_4 \lesssim \gamma^2 |B_i|$ as desired.

It remains to bound $\III_1$, We apply the same argument as of $(\ref{eqGL06})$ using $x_i\in F^i$ instead of $y_i$. So we have
\begin{align}\label{eqGLAM13}
\delta |\vec \alpha_i|^2 \lesssim
\delta \big |\mathcal{M}[\wt{N}(\nabla u|\Psi^{3})](x_i)\big |^2
\lesssim \gamma^2 \beta^2,
\end{align}
because $x_i \in E_{\beta,\gamma,\delta}$, from which we easily deduce $\III_1 \lesssim \gamma^2 |B_i|$. The lemma follows.
\end{proof}

The ``good-lambda'' distributional inequality (\ref{eqGLMN1}) can be used to derive the $L^p-L^p$ boundedness result.

\begin{Lemma}\label{LENSTP}
Let $p>1$ and $L$ be an elliptic operator satisfying \HH$_{\lambda,M,\kappa}$.  For any a weak solution $u\in W^{1,2}_{loc}(\Omega)$ to $Lu=0$, any cut-off function in the form $\Psi:=\Psi_{B,l,\epsilon}$ for some $\epsilon>0$, some $l>100\epsilon$, and some boundary ball $B$ of radius $l$, we have
\begin{multline*}
\|\wt{N}(\nabla u|\Psi^{3})\|^p_{p}\leq C_p \bigg \|\bigg (\dashint_{B_{l}(.)}\dashint_{l\leq |s|\leq 2l}|\nabla u|^2 \Psi_B^3 ds\, dy\bigg )^{1/2}\bigg \|^p_{p}+ C_p \kappa^{p/2} \|\wt{N}(\nabla u|\Psi)\|^p_{p}  \\ + C_p \|S(\overline\nabla  u|\Psi)\|^p_{p}
\end{multline*}
where $C_p>0$ depends only on $\lambda$, $d$, $n$, and $p$. 
\end{Lemma}

\begin{Rem}
The limitation $p>1$ comes from the fact that we used the maximal function $\mathcal M$ in Lemma \ref{GLAM}. However, with the same arguments, we could prove an analogue of \eqref{eqGLMN1} where we replace $\mathcal M$ by $\mathcal M_{q}$ defined as $\mathcal M_q[f] := \big( \mathcal M[f^q]\big)^{1/q}$ for any $q>0$ (with a constant $C$ depending now also on $q$). Then we could establish Lemma \ref{LENSTP} for any $p>0$ by invoking \eqref{eqGLMN1} that used $\mathcal M_{q}$ with $0<q<p$.
\end{Rem}

\begin{proof}
We apply the distribution inequality (\ref{eqGLMN1}) to obtain that there exists a $\eta>0$ such that for any $\gamma,\delta\in (0,1)$, we have
\[\begin{split}
\|\wt{N}(\nabla u|\Psi^{3})\|^p_{p}& =c_p\int_0^\infty \beta^{p-1}|\{\wt{N}(\nabla u|\Psi^3)>\beta\}|d\beta\\
& \leq c_p\int_0^\infty \beta^{p-1}|\{\wt{N}(\nabla u|\Psi^{3})>\beta\}\cap E_{\beta, \gamma,\delta}|d\beta
+c_p\int_0^\infty \beta^{p-1}|E^c_{\beta,\gamma,\delta}|d\beta\\
& \lesssim c_p \gamma^2 \int_0^\infty \beta^{p-1}|\{\mathcal{M}[\wt{N}(\nabla u|\Psi^{3})]>\eta \beta\}|d\beta
+ c_p\int_0^\infty \beta^{p-1}|E^c_{\beta,\gamma,\delta}|d\beta \\
:= \I + \II
\end{split}\]
where the implicit constant depends only on $p$. But in one had, we have
\begin{equation*}
\I = \frac{\gamma^2}{\eta^{p}}  \|\mathcal M[\wt{N}(\nabla u|\Psi^{3})]\|^p_{p} \lesssim \frac{\gamma^2}{\eta^{p}}  \|\wt{N}(\nabla u|\Psi^{3})\|^p_{p}
\end{equation*}
by the $L^p$-boundedness of the Hardy-Littlewood maximal operator. On the other hand, 
\[\begin{split}
\II& = \gamma^{-p} \Big\| \mathcal M\Big[\Big (\dashint_{B_{l}(.)}\dashint_{l\leq |s|\leq 2l}|\nabla u|^2 \Psi_B^3 \, ds\, dy\Big )^{1/2}\Big] +\delta^{1/2} \mathcal{M}[\wt{N}(\nabla u|\Psi^3)] \\ & \qquad\qquad\qquad\qquad\qquad\qquad + \delta^{-1/2}\kappa^{1/2} \mathcal{M}[\wt{N}(\nabla u|\Psi)]   + \delta^{-1/2} \mathcal{M}[S(\overline\nabla u|\Psi)] \Big\|_p^p \\
& \lesssim \gamma^{-p} \Big\| \Big (\dashint_{B_{l}(.)}\dashint_{l\leq |s|\leq 2l}|\nabla u|^2 \Psi_B^3 \, ds\, dy\Big )^{1/2} +\delta^{1/2} \wt{N}(\nabla u|\Psi^3) + \delta^{-1/2}\kappa^{1/2} \wt{N}(\nabla u|\Psi) \\ & \qquad\qquad\qquad\qquad\qquad\qquad\qquad\qquad\qquad\qquad\qquad\qquad\qquad\qquad + \delta^{-1/2} S(\overline\nabla u|\Psi) \Big\|_p^p
\end{split}\]
again using the $L^p$-boundedness of the Hardy-Littlewood maximal operator. Altogether, we have
\begin{multline*}
\|\wt{N}(\nabla u|\Psi^{3})\|^p_{p} \lesssim \Big(\frac{\gamma^2}{\eta^p} + \frac{\delta^{p/2}}{\gamma^p}\Big) \|\wt{N}(\nabla u|\Psi^{3})\|^p_{p}
 + \gamma^{-p} \Big\| \Big (\dashint_{B_{l}(.)}\dashint_{l\leq |s|\leq 2l}|\nabla u|^2 \Psi_B^3 \, ds\, dy\Big )^{1/2}\Big\|_p^p \\
 + \delta^{-p/2}\kappa^{p/2} \gamma^{-p} \|\wt{N}(\nabla u|\Psi)\|^p_{p}  + \delta^{-p/2} \gamma^{-p} \Big\| S(\overline\nabla u|\Psi)\Big\|_p^p.
\end{multline*}
The lemma follows by taking $\gamma$ and then $\delta$ (depending only on $\lambda$, $d$, $n$, and $p$) such that $(\gamma^2\eta^{-p} + \delta^{p/2}\gamma^{-p})$ is small enough, so that the first term on the right-hand side above can be hidden in the left-hand side (which is allowed because all the terms are finite, due to the use of the compactly supported cut-off function $\Psi$). 
\end{proof}

\section{\texorpdfstring{$S\leq N$}{TEXT} Local Estimates} \label{SS<N}

In this section, we aim to establish that the square function is locally bounded by the non-tangential maximal function, result that is eventually given in Lemma \ref{LSLNF} below.

\medskip

Remember that we have three different directional derivatives to deal with, which are the tangential derivatives $\nabla_x$, angular derivatives $\nabla_\varphi$, and radial derivative $\partial_r$. To prove these estimates, we first bound the square function of the radial derivative by the square functions of the tangential and angular derivatives, and we shall rely on Proposition \ref{PpOPER}, i.e. the expression of the equation in cylindrical coordinates. Then, we treat the tangential and angular directional derivatives, and a key point is the fact that those derivatives verify $\partial_r |t| = \partial_\varphi |t| = 0$. 

\begin{Lemma}\label{LMRRR}
Let $L$ be an elliptic operator satisfying \HH$_{\lambda,\kappa}$.
For any weak solution $u\in W^{1,2}_{loc}(\Omega)$ to $Lu=0$ and any radial cut-off function $\Psi\in C^\infty_0(\Omega,[0,1])$, we have
\begin{equation}\label{eqRRR00}
\|S(\partial_r u|\Psi^3)\|^2_{2}
\leq C \Big ( \kappa \|\wt{N}(\nabla u|\Psi^3)\|^2_{2} +\|S(\nabla_x u|\Psi^3)\|^2_{2}
+\|S(\nabla_{\varphi}u|\Psi^3)\|^2_{2}\Big ),
\end{equation}
where the constant $C>0$ depends only on $\lambda$ and the dimensions $d$ and $n$.
\end{Lemma}

\begin{proof} This is basically an outcome of the equation: some derivatives can be represented in terms of others.
Observe that
\[|\nabla \partial_r u| \leq |\nabla_x \partial_r u| +  |\nabla_\varphi \partial_r u| + |\partial_r^2 u| \leq  |\nabla \nabla_x u| + |\nabla \nabla_\varphi u| + \frac{1}{|t|}|\nabla_\varphi u| + |\partial_r^2 u|\]
because $\nabla_x$ and $\partial_r$ commute and the commutator $[\partial_r,\partial_\varphi]$ is $- \frac1{|t|} \partial_\varphi$ (see Proposition \ref{PPLL1}). But since $u$ is a weak solution to $Lu = 0$, \ref{eqPOPER} implies that
\[\begin{split}
|\partial_r^2 u|  & = \Big|- \diver_x(\mathcal{A}_1\nabla_x u ) - \diver_x(\mathcal{A}_2\partial_r u ) - \frac{1}{2}\sum_{i,j=d+1}^n\partial_{\varphi_{ij}}^2 u \Big|\\
& \lesssim |\nabla_x \A| |\nabla u| + \lambda^{-1} |\nabla_x \partial_r u| + \lambda^{-1} |\nabla_x \nabla_x u| + |\nabla_\varphi^2 u| \leq |\nabla_x \A| |\nabla u| + \lambda^{-1} |\nabla \nabla_x  u| + |\nabla \nabla_\varphi u|
\end{split}\]
by using again the fact that $\nabla_x$ and $\partial_r$ commute. 
By combining the two inequalities above, we obtain
\[|\nabla \partial_r u| \lesssim  |\nabla \nabla_x u| + |\nabla \nabla_\varphi u| + \frac{1}{|t|}|\nabla_\varphi u| + |\nabla_x \A| |\nabla u|\]
Now, \eqref{S=int2} entails that
\begin{multline*}
\|S(\partial_r u|\Psi^3)\|_{L^2(\R^d)}^2 \lesssim \|S(\nabla_x u|\Psi^3)\|_{L^2(\R^d)}^2 + \|S(\nabla_\varphi u|\Psi^3)\|_{L^2(\R^d)}^2 \\
+ \iint_\Omega |\nabla_\varphi u|^2 \Psi^3 \frac{dtdx}{|t|^{n-d}} + \iint_\Omega |t|^2|\nabla_x \A|^2 |\nabla u|^2 \Psi^3 \frac{dtdx}{|t|^{n-d}}.
\end{multline*}
However, since $|t||\nabla_x \A| \in CM(\kappa)$, the Carleson inequality \eqref{Carleson} implies that
\[\iint_\Omega |t|^2|\nabla_x \A|^2 |\nabla u|^2 \Psi^3 \frac{dtdx}{|t|^{n-d}} \lesssim \kappa \|\wt N(u|\Psi^3)\|^2_{L^2}.\]
In addition, Proposition \ref{PDECC} applied with $\Phi = |t|^{d-n} \Psi^3 $ infers that 
\[\iint_\Omega |\nabla_\varphi u|^2 \Psi^3\frac{dtdx}{|t|^{n-d}} \lesssim \iint_\Omega |\nabla \nabla_\varphi u|^2 \Psi^3\frac{dtdx}{|t|^{n-d-2}} \lesssim \|S(\nabla_\varphi u|\Psi^3)\|^2_{L^2(\R^d)} \]
by \eqref{S=int2}. The lemma follows.
\end{proof}

In order to deal with the tangential and angular directional derivatives, we will first prove a generalized result that works for both of them. 
Let us write $\partial_v$ for either a tangential derivative $\partial_{x_i}$, an angular derivative $\partial_{\varphi_{ij}}$, or the radial derivative $\partial_r$. The key step is to use the equation $Lu=0$ properly. Since we want to estimate the gradient of solutions, we should study the commutators $[L,\partial_v]$ and try to bound them in a clever way. In the next lemma, we will estimate the square function of $\partial_v u$ and we are able to see how the commutator $[L,\partial_v]$ plays an important role in the estimates.

It will be convenient to introduce the bilinear form $\mathcal{B}(\cdot,\cdot)$ defined for $ f \in L^1_{loc}(\Omega)$ and $\Psi\in C_0^\infty(\Omega)$, 
\begin{align}\label{DEFBL}
    \mathcal{B}(f, \Psi):=-\frac{1}{2}\iint_\Omega \partial_r[|t| f]\, \partial_r\Psi \frac{dtdx}{|t|^{n-d-1}}.
\end{align}
Beware that $\mathcal{B}(f, \Psi)$ {\bf may be negative} even when the function $f$ is positive. We are now ready for our next lemma.

\begin{Lemma}\label{lm.GSN}
Let $L:=-\diver(|t|^{d+1-n}\A \nabla)$ be an elliptic operator satisfying \eqref{defellip} and \eqref{coe.afor}. For any weak solution $u\in W^{1,2}_{loc}(\Omega)$ to $Lu=0$ and any radial cut-off function $\Psi\in C^\infty_0(\Omega,[0,1])$, we have
\begin{multline}\label{eqGSN00}
    \frac78 \lambda \|S(\partial_v u|\Psi^3)\|_{2}^2 \leq \iint_\Omega |\partial_v u|^2\partial_r\Psi^3 \frac{dtdx}{|t|^{n-d-1}}  + C \iint_\Omega |\partial_v u|^2|\nabla_x\Psi|^2\Psi \frac{dtdx}{|t|^{n-d-2}}\\
 + \mathcal{B}(|\partial_v u|^2, \Psi^3) + \iint_\Omega \Big ([L,\partial_v]u\Big )\Big (\Psi^3 |t|\partial_v u \Big )dtdx,   
\end{multline}
where $C>0$ depends only on the ellipticity constant $\lambda$ and the dimensions $d$ and $n$.
\end{Lemma}

The bound \eqref{eqGSN00} may look a bit cryptic. The last term of (\ref{eqGSN00}) is the one that contains the commutator $[L,\partial_v]$, and will be removed in the next lemmas. The first term in the right-hand side is the ``trace'' term, that is the term that will become $\Tr(\partial_v u)$ when we take $\Psi \uparrow 1$. The two other quantities are ``error'' terms that contain derivatives of the cut-off function $\Psi$, and that will eventually disappear when we take $\Psi \uparrow 1$.

\begin{proof}
To lighten the notation, we write $V$ for $\partial_v u$. First of all, since $\mathcal{A}$ satisfies the uniform ellipticity condition \eqref{defellip}, we have
\begin{align*}
    \lambda\|S(V|\Psi^3)\|^2_{2}=\lambda\iint_\Omega |\nabla V|^2\Psi^3\frac{dtdx}{|t|^{n-d-2}}\leq \iint_\Omega \mathcal{A}\nabla V \cdot \nabla V \Psi^3\frac{dtdx}{|t|^{n-d-2}}
\end{align*}
By product rule, 
\begin{multline*}
    \iint_\Omega \mathcal{A}\nabla V \cdot \nabla V \, \Psi^3\frac{dtdx}{|t|^{n-d-2}}=  \iint_\Omega \mathcal{A}\nabla V \cdot \nabla \Big (V \Psi^3|t|\Big )\frac{dtdx}{|t|^{n-d-1}}\\
    -\iint_\Omega \mathcal{A}\nabla V\cdot \nabla \Psi^3 \, V \frac{dtdx}{|t|^{n-d-2}}
    -\iint_\Omega \mathcal{A}\nabla V\cdot \nabla (|t|) \, V \Psi^3\, \frac{dtdx}{|t|^{n-d-1}}=\RN{1}+\RN{2}+\RN{3}.
\end{multline*}

We start from the term $\RN{1}$. Recall that $u$ is a weak solution to the equation $Lu=0$. It follows that:
\begin{align*}
    LV=L(\partial_v u)=\partial_v (Lu)+[L,\partial_v]=[L,\partial_v]\ \text{a.e. in $\Omega$.}
\end{align*}
Consequently, 
\begin{align*}
    \RN{1}=\iint_\Omega \Big ([L,\partial_v]u\Big )\Big (V\Psi^3|t|\Big )\, dtdx,
\end{align*}
which is one of the term from the right-hand side of (\ref{eqGSN00}). 
For the term $\RN{2}$, since matrix is in the form of (\ref{coe.afor}),
\begin{multline*}
    \RN{2}=-\iint_\Omega \mathcal{A}_1\nabla_x V\cdot \nabla_x \Psi^3 \, V\frac{dtdx}{|t|^{n-d-2}}-\iint_\Omega \mathcal{A}_2\frac{t}{|t|}\nabla_t V\cdot \nabla_x \Psi^3 \, V\frac{dtdx}{|t|^{n-d-2}}\\
   -\iint_\Omega \nabla_t V\cdot \nabla_t\Psi^3 \, V\frac{dtdx}{|t|^{n-d-2}}=:\RN{2}_1+\RN{2}_2+\RN{2}_3.
\end{multline*}

The terms $\RN{2}_1$ and $\RN{2}_2$ are estimated together by
\begin{align*}
    \RN{2}_1+\RN{2}_2\leq \frac18 \lambda \|S(V|\Psi^3)\|^2_2+ C_\lambda \iint_\Omega |V|^2|\nabla_x \Psi|^2\Psi \frac{dtdx}{|t|^{n-d-2}}.
\end{align*}
We hide the term $\frac18 \lambda \|S(V|\Psi^3)\|^2_2$ in the left-hand side of (\ref{eqGSN00}), and the second term in the right-hand side above stays on the right-hand side of (\ref{eqGSN00}).

As for $\RN{2}_3$, since $\Psi$ is radial, we have that 
\[\nabla_t V \cdot \nabla_t \Psi^3 = (\nabla_t V \cdot \nabla_t |t|) \, \partial_r \Psi^3 = (\partial_r V)  (\partial_r\Psi^3)\]
 and thus, 
\[\RN{2}_3=-\frac{1}{2} \iint_\Omega (\partial_r V^2) (\partial_r\Psi^3) \, \frac{dtdx}{|t|^{n-d-2}} = \mathcal{B}(V^2,\Psi^k) + \frac{1}{2}\iint_\Omega V^2 \partial_r\Psi^3\, \frac{dtdx}{|t|^{n-d-1}}\]
by definition of $\mathcal{B}(.,.)$, see (\ref{DEFBL}). The last term $\RN{3}$ is similar, because we have 
\[\RN{3}=-\iint_\Omega \partial_r V \, V\Psi^3\, \frac{dtdx}{|t|^{n-d-1}} = - \frac12 \iint_\Omega (\partial_r V^2) \, \Psi^3\, \frac{dtdx}{|t|^{n-d-1}} = \frac12 \iint_\Omega V^2 \, \partial_r \Psi^3\, \frac{dtdx}{|t|^{n-d-1}}\]
thanks to the integration by parts given in Proposition \ref{IBBdrdphi}.
The lemma follows. 
\end{proof}

Now we bound the square function of the tangential derivatives by applying Lemma \ref{lm.GSN} with $\partial_v = \partial_{x}$. Recall that we write $\partial_x$ for any tangential directional derivative $\partial_i:=\vec e_i\cdot \nabla$ where $i\leq d$. As we have discussed in the previous paragraphs, the commutator $[L,\partial_x]$ plays an important role in computing the square function of $\partial_x$. In our particular case, an easy computation shows that
\begin{equation} \label{defCLdx}
    [L, \partial_x]=\diver_x (|t|^{d+1-n} \partial_x \A) \nabla 
\end{equation}
because $\partial_x |t| =0$.  

\begin{Cor}\label{LEMMA1}
Let $L$ be an elliptic operator satisfying \HH$_{\lambda,\kappa}$. 
For any weak solution $u\in W^{1,2}_{loc}(\Omega)$ to $Lu=0$ and any radial cut-off function $\Psi\in C^\infty_0(\Omega,[0,1])$, we have
\begin{multline}\label{SNQP}
\frac34 \lambda \|S(\nabla_x u|\Psi^3)\|_{2}^2\leq \iint_\Omega |\nabla_x u|^2\partial_r\Psi^3\frac{dtdx}{|t|^{n-d-1}} + C\kappa \|\wt N(\nabla u|\Psi^3)\|^2_2 \\
    + C \iint_\Omega |\nabla_x u|^2|\nabla_x\Psi|^2\Psi \frac{dtdx}{|t|^{n-d-2}}
    +\mathcal{B}(|\nabla_x u|^2, \Psi^3),
\end{multline}
where $C>0$ depends only on $\lambda$, $d$, and $n$.
\end{Cor}

\begin{proof}
The bound \eqref{SNQP} is a consequence of the same bound on each of the tangential derivative $\partial_x$, and then summing up. For a given tangential derivative, \eqref{SNQP} is an immediate consequence of Lemma \ref{lm.GSN} and the bound
\begin{multline} \label{claimLE1}
\left| \iint_\Omega \Big ([L,\partial_x]u\Big ) \big(\Psi^3 |t|\partial_x u \big )dtdx \right| \leq \frac18\lambda \|S(\nabla_x u|\Psi^3)\|^2_2 + C \kappa  \|\wt N(\nabla u|\Psi^3)\|^2_2 \\ + C \iint_\Omega |\nabla_\varphi u|^2|\nabla_x\Psi|^2\Psi \frac{dtdx}{|t|^{n-d-2}}.
\end{multline}
for any tangential derivative $\partial_x$. 
So we fix a tangential directional derivative $\partial_x$, and by \eqref{defCLdx} and then integration by parts, we have
\begin{multline}\label{TSLN.eq01}
\iint_\Omega \Big ([L,\partial_x]u\Big )\Big ((\partial_x u)\Psi^3 |t|\Big )dtdx = \iint_\Omega \Big( \diver (|t|^{d+1-n} \partial_x \A) \nabla u \Big) \Big (\Psi^3 |t| \partial_x u \Big )dtdx \\
= - \iint_\Omega (\partial_x \mathcal{A}) \nabla u\cdot \nabla (\partial_x u)\Psi^3\frac{dtdx}{|t|^{n-d-2}}-\iint_\Omega (\partial_x\mathcal{A}) \nabla u\cdot \nabla \big (|t|\Psi^3\big )\partial_x u\frac{dtdx}{|t|^{n-d-1}} \\
    =\RN{1}+\RN{2}.
\end{multline}
Since $|t||\nabla_x \mathcal{A}| \in CM(\kappa)$, the term $\RN{1}$ is bounded as follows
\begin{multline} \label{LEeq1}
    |\RN{1}| \leq \frac18 \lambda \|S(\partial_x u|\Psi^3)\|^2_2 + C_\lambda \int_\Omega |t|^2 |\nabla_x \A|^2 |\nabla u|^2 \Psi^3 \frac{dt\, dx}{|t|^{n-d-1}} \\
    \leq \frac18 \lambda \|S(\partial_x u|\Psi^3)\|^2_2 + C_\lambda \kappa \|\wt N_a(\nabla u|\Psi^3)\|^2_2. 
\end{multline}
For $\RN{2}$, remark that the special structure of $\A$ given in \eqref{coe.afor} implies that the only derivatives that hit $|t|\Psi^3$ are tangential derivative, for which $\nabla_x|t| = 0$. Therefore,
\begin{multline} \label{LEeq2}
   | \RN{2}| = \left| \iint_\Omega (\partial_x\mathcal{A}) \nabla u\cdot (\nabla_x \Psi^3) (\partial_x u) \frac{dtdx}{|t|^{n-d-2}} \right|\\
\leq \iint_\Omega |t|^2 |\partial_x\mathcal{A}|^2 |\nabla u|^2 \Psi^3 \frac{dtdx}{|t|^{n-d}} + \iint_\Omega |\nabla_x u|^2 |\nabla_x\Psi|^2 \Psi \frac{dtdx}{|t|^{n-d-2}} \\
 \leq C\kappa \|\wt N(\nabla u|\Psi^3)\|^2_2 + \iint_\Omega |\nabla_x u|^2 |\nabla_x\Psi|^2 \Psi \frac{dtdx}{|t|^{n-d-2}}.
    \end{multline}
The lemma follows.
\end{proof}

It remains to estimate the square function of the angular directional derivatives.

\begin{Cor}\label{llSVAR}
Let $L$ be an elliptic operator satisfying \HH$_{\lambda,\kappa}$. 
For any weak solution $u\in W^{1,2}_{loc}(\Omega)$ to $Lu=0$ and any radial cut-off function $\Psi\in C^\infty_0(\Omega,[0,1])$, we have
\begin{multline}\label{eqVar25}
    \frac34 \lambda \|S(\nabla_{\varphi} u|\Psi^3)\|_{2}^2\leq C \kappa \|\wt N(\nabla u|\Psi^3)\|^2_2 + C \|S(\nabla_{x} u|\Psi^3)\|^2_{2}\\
    +C \iint_\Omega |\nabla_{\varphi} u|^2|\nabla_x\Psi|^2\Psi\frac{dtdx}{|t|^{n-d-2}}
    +\mathcal{B}(|\nabla_\varphi u|^2,\Psi^3),
\end{multline}
where $C>0$ depends only on $\lambda$, $d$, and $n$.
\end{Cor}

\begin{proof}
Fix an angular directional derivative $\partial_\varphi$. Thanks to Lemma \ref{lm.GSN}, it suffices to show that 
\begin{multline} \label{claimLE2}
\left| \iint_\Omega \Big ([L,\partial_\varphi]u\Big ) \big(\Psi^3 |t|\partial_\varphi u \big )dtdx + \iint_\Omega |\partial_\varphi u|^2 \partial_r \Psi^3 \, \frac{dt\, dx}{|t|^{n-d-1}} \right| \\
\leq \frac18\lambda \|S(\nabla_\varphi u|\Psi^3)\|^2_2 + C \kappa \|\wt N(\nabla u|\Psi^3)\|^2_2 + C  \|S(\nabla_xu|\Psi^3)\|^2_2 \\ 
+ C \iint_\Omega |\nabla_x u|^2|\nabla_x\Psi|^2\Psi \frac{dtdx}{|t|^{n-d-2}}.
\end{multline}
It will be important to estimate the two terms in the left-hand side of \eqref{claimLE2} together, because there will be some cancellation.

We invoke Proposition \ref{CorDAR} to say that 
\begin{multline} \label{LEeq3}
\iint_\Omega \Big ([L,\partial_\varphi]u\Big ) \big(\Psi^3 |t|\partial_\varphi u \big )dtdx = \iint_\Omega \Big( \diver_x (|t|^{d+1-n} \partial_\varphi \A) \nabla u \Big) \big(\Psi^3 |t| \partial_x u \big ) dtdx \\
+ 2\iint_\Omega (\partial_r \partial_\varphi u) \Psi^3 ( \partial_\varphi u) \frac{dt\, dx}{|t|^{n-d-1}} + \iint_\Omega \diver_x(\A_2 \partial_\varphi u) \big(\Psi^3 \partial_\varphi u \big ) \frac{dt\, dx}{|t|^{n-d-2}}\\
=:\RN{1}+\RN{2}+\RN{3}.
\end{multline}
 By the product rule,
\[\RN{3} \lesssim \iint_\Omega |\nabla_x \mathcal{A}_2| |\nabla_\varphi u| |\partial_\varphi u| \Psi^3 \, \frac{dtdx}{|t|^{n-d-1}}+  \iint_\Omega  |\mathcal{A}_2| \nabla_\varphi \nabla_x u|  |\partial_\varphi u| \Psi^3  \, \frac{dtdx}{|t|^{n-d-1}}:=\RN{3}_1+\RN{3}_2.\]
Since $|t||\nabla_x \mathcal{A}_2| \in CM(\kappa)$, the term $\RN{3}_1$ can be estimated as follows
\begin{align*}
    \RN{3}_1\leq \epsilon \iint_\Omega |\nabla_\varphi u|^2\Psi^3 \frac{dtdx}{|t|^{n-d}} + C \epsilon^{-1} \kappa \|\wt N(\nabla u|\Psi^3)\|^2_2 
    \leq \frac1{24} \lambda \|S(\partial_\varphi u|\Psi^3)\|_2^2 + C_\lambda \kappa \|\wt N(\nabla u|\Psi^3)\|^2_2,
\end{align*}
by using Proposition \ref{PDECC}, \eqref{S=int2}, and by taking $\epsilon$ small enough (depending only on $\lambda$, $d$ and $n$).  

Based on the same arguments, the term $\RN{3}_2$ is bounded by
\begin{align*}
    \RN{3}_2\leq \epsilon \iint_\Omega |\nabla_\varphi u|^2\Psi^3 \frac{dtdx}{|t|^{n-d}} + C\epsilon^{-1} \|S(\nabla_x u|\Psi^3)\|^2_2 
    \leq \frac1{24} \lambda \|S(\partial_\varphi u|\Psi^3)\|_2^2 + C_\lambda \|S(\nabla_x u|\Psi^3)\|^2_2.
\end{align*} 

The term $\RN{1}$ is analogous to the one obtained from the commutator in the proof of Corollary \ref{LEMMA1}. We repeat quickly the argument. By integration by parts, 
\[\RN{1} = - \iint_\Omega (\partial_\varphi \mathcal{A}) \nabla u \cdot \nabla_x  (\partial_\varphi u) \Psi^3\frac{dtdx}{|t|^{n-d-2}} -\iint_\Omega (\partial_\varphi \mathcal{A}) \nabla u \cdot \nabla_x \Psi^3 \, (\partial_\varphi u) \, \frac{dtdx}{|t|^{n-d-1}}
    :=\RN{1}_1+\RN{1}_2.\] 
The first integral is bounded by using the inequlity $2ab \leq \epsilon a^2 + \epsilon^{-1} b^2$, and the fact that $|t||\nabla_\varphi A| \in CM(\kappa)$ we get similarly to \eqref{LEeq1} that
\[|\RN{1}_1| \leq \frac1{24} \lambda \|S(\partial_\varphi u|\Psi^3)\|_2^2 + C_\lambda \kappa \|\wt N(\nabla u|\Psi^3)\|_2^2.\]
As for the second integral, we proceed as in \eqref{LEeq2} and we obtain
\[|\RN{1}_2| \leq C\kappa \|\wt N(\nabla u|\Psi^3)\|^2_2 + \iint_\Omega |\nabla_\varphi u|^2 |\nabla_x\Psi|^2 \Psi \frac{dtdx}{|t|^{n-d-2}}.\]

The term $\RN{2}$ cancels out with the ``trace'' term. Indeed, we have
\[\RN{2} = \iint_\Omega \partial_r |\partial_\varphi u|^2 \Psi^3 \frac{dt\, dx}{|t|^{n-d-1}} = - \iint_\Omega |\partial_\varphi u|^2 \partial_r \Psi^3 \frac{dt\, dx}{|t|^{n-d-1}}\] 
by the integration by parts (Proposition \ref{IBBdrdphi}). Observe that all our computations proved the claim \eqref{claimLE2}, thus the lemma follows.
\end{proof}

In the following, we combine all the previous results of this section together. We recall that $\overline{\nabla}$ stands for the gradient in cylindrical coordinates. Remember that we write respectively $\|S(\nabla_x u|\Psi^3)\|^2_2$ and $\|S(\nabla_\varphi u|\Psi^3)\|^2_2$ for the sums of the square functions over all tangential derivatives and angular derivatives in $L^2$ norm.

\begin{Lemma}\label{LSLNF}
Let $L$ be an elliptic operator satisfying \HH$_{\lambda,\kappa}$. 
There exists three constants $C_1,C_2,C_3>0$ depending only on $\lambda$, $d$, and $n$ such that for any weak solution $u\in W^{1,2}_{loc}(\Omega)$ to $Lu=0$ and any cut-off radial function $\Psi\in C^\infty_0(\Omega,[0,1])$, we have
\begin{multline}\label{eqSLNFZ}
    \|S(\overline\nabla u|\Psi^3)\|^2_{2}
    \leq C_1\bigg (\iint_\Omega |\nabla_x u|^2\partial_r\Psi^3\frac{dtdx}{|t|^{n-d-1}} + \mathcal{B}(|\nabla_x u|^2, \Psi^3) \bigg ) + C_2\, \mathcal{B}(|\nabla_\varphi u|^2, \Psi^3) \\ 
    + C_3 \bigg (\kappa \|\wt{N}_a(\nabla u|\Psi^{3})\|^2_{2}
    +\iint_\Omega |\nabla u|^2|\nabla_x\Psi|^2\Psi\frac{dtdx}{|t|^{n-d-2}}\bigg )
\end{multline}
In addition, if $\Psi$ satisfies \COF$_K$, we have
\begin{align}\label{eqSLNFF}
    \|S(\overline\nabla u|\Psi^3)\|^2_{2}\leq C_K \|\wt N(\nabla u|\Psi)\|^2_{2},
\end{align}
where $C_K$ depends on $\lambda$, $n$, and $K$.
\end{Lemma}

\begin{Rem}
Remember that the first term in the right-hand side of \eqref{eqSLNFZ} is the ``trace'' term, and all the other terms are meant to disappear when $\Psi \uparrow 1$.

Moreover, we have different constants because the terms that are multiplied by $C_1$ and $C_2$ may be negative. We can say that $1 \leq C_2 \leq C_1 \leq C_3$, but nothing more, in particular taking $C_1=C_2=C_3$ would probably render the inequality false.
\end{Rem}

\begin{Rem}
The result (\ref{eqSLNFF}) tells us that the sum of the square functions of all tangential directional derivatives, angular directional derivatives, and radial direction derivatives can be estimated locally by the non-tangential maximal function of the full gradient. 
\end{Rem}

\begin{proof}
The inequality (\ref{eqSLNFZ}) is an immediate consequence of Lemma \ref{LMRRR}, Corollary \ref{LEMMA1}, and Corollary \ref{llSVAR}. 

We turn to the proof of (\ref{eqSLNFF}). Since $\Psi$ satisfies \COF$_K$, we have
\begin{equation} \label{PsiCMZ}
|t| |\nabla \Psi| \leq K \quad \text{ and } \quad \1_{\supp \nabla \Psi} \in CM(K),\end{equation}
in particular $|t||\nabla \Psi| \in CM(K^3)$. We deduce that
\begin{multline} \label{DPsibyN}
\iint_\Omega |\nabla u|^2 |\nabla_x\Psi|^2\Psi \frac{dtdx}{|t|^{n-d-2}}+\iint_\Omega |\nabla_x u|^2|\partial_r\Psi^3|\frac{dtdx}{|t|^{n-d-1}} \\
\leq C \iint_\Omega |\nabla u|^2 \big[ |t|^2|\nabla \Psi|^2 + |t||\nabla \Psi|\big] \Psi \frac{dtdx}{|t|^{n-d}} \leq CK^3 \|\wt N(\nabla u|\Psi)\|^2_2.
\end{multline}
by \eqref{PsiCMZ} and the Carleson inequality \eqref{Carleson}.

Consequently, it suffices to show for any $V\in L^2_{loc}(\Omega)$ and $\delta\in (0,\infty)$,
\begin{align}\label{SLNF.eq00}
    |\mathcal{B}(|V|^2,\Psi^3)| \leq \delta \|S(V|\Psi^3)\|^2_2+ C\delta^{-1} \|\wt N(V|\Psi)\|^2_2
\end{align}
because then (\ref{eqSLNFF}) follows easily by choosing $\epsilon$ small enough. From the definition of $\mathcal{B}(.,.)$, see (\ref{DEFBL}), and the product rule, we have
\begin{multline}\label{SLNF.eq01}
    |\mathcal{B}(|V|^2,\Psi^3)| \leq \iint_\Omega |V| |\partial_r V| |\partial_r\Psi| \Psi ^2 \frac{dtdx}{|t|^{n-d-2}}  + \frac{1}{2}\iint_\Omega |V|^2 |\partial_r\Psi| \Psi^2 \frac{dtdx}{|t|^{n-d-1}} \\
    \leq \delta \|S(V|\Psi^3)\|^2_2+C\delta^{-1} \iint_\Omega |V|^2 |\partial_r \Psi|^2\Psi \frac{dtdx}{|t|^{n-d-2}} + \frac{1}{2} \iint_\Omega |V|^2 |\partial_r\Psi| \Psi^2\frac{dtdx}{|t|^{n-d-1}}.
\end{multline}
By using again \eqref{PsiCMZ} and the Carleson inequality, the last two terms of (\ref{SLNF.eq01}) are bounded by $K^3 \|\wt N(V|\Psi)\|^2_2$. Hence (\ref{SLNF.eq00}) follows. 
\end{proof}

Let us get a little bit further, since it will help us when we pass from local to global estimates.

\begin{Lemma}\label{LB<SN}
Let $L$ be an elliptic operator satisfying \HH$_{\lambda,\kappa}$. 
For any weak solution $u\in W^{1,2}_{loc}(\Omega)$ to $Lu=0$, any cut-off function $\Psi\in C^\infty_0(\Omega,[0,1])$ satisfying \COF$_K$, any $\delta \in (0,1)$, we have
\begin{multline*}
|\mathcal B(|\partial_r u|^2,\Psi^3)| = \frac12 \left| \iint_\Omega \Big(|\partial_r u|^2 + |t| \partial_r(|\partial_r u|^2) \Big) \partial_r \Psi^3 \frac{dt\, dx}{|t|^{n-d-1}} \right|  \\ 
\leq (\delta+ \delta^{-1}K^2\kappa) \|\wt N(\nabla u|\Psi)\|_2^2  + C\delta^{-1}K^2 \|S(\overline\nabla u|\Psi^3)\|_2^2,
\end{multline*}
where $C$ depends on $\lambda$, $n$, $\delta$ and $K+M+\kappa$.
\end{Lemma}

\begin{proof}
The equality is just the product rule and the definition of $\mathcal B(|\partial_r u|^2,\Psi^3)$, see \eqref{DEFBL}. The bound
\[\left| \iint_\Omega |\partial_r u|^2 \partial_r \Psi^3 \frac{dt\, dx}{|t|^{n-d-1}} \right| \leq  (\delta+ \delta^{-1}K^2\kappa) \|\wt N(\nabla u|\Psi)\|_2^2  + C\delta^{-1}K^2 \|S(\overline\nabla u|\Psi^3)\|_2^2 \]
was established in \eqref{bdBr1}. It remains to show a bound on
\[\I := \left| \iint_\Omega \partial_r(|\partial_r u|^2) \partial_r \Psi^3 \, \frac{dt\, dx}{|t|^{n-d-2}} \right| ,\]
but that is similar to \eqref{SLNF.eq01}. Indeed, we use $|t||\nabla \Psi| \in CM(K)$, \eqref{Carleson}, and the inequality $ab \leq \delta a^2 + b^2/4\delta$ to obtain
\[I = 2 \left| \iint_\Omega (\partial_r^2 u) (\partial_r u) \partial_r \Psi^3 \frac{dt\, dx}{|t|^{n-d-2}} \right| \leq \delta \|\wt N(\partial_r u|\Psi)\|_2^2 + CK\delta^{-1} \|S(\partial_r u|\Psi^3)\|_2^2.\]
The lemma follows.
\end{proof}

\begin{Lemma}\label{LlastS<N}
Let $\kappa \in (0,1)$ and let $L$ be an elliptic operator satisfying \HH$_{\lambda,\kappa}$. Take $\Psi \in C^\infty_0(\Omega)$ satisfying \COF$_K$. 
There exist four constants $c_0, C_1,C_2,C_3>0$ depending on $\lambda$ and $n$ - the first one being small and the last three being large - such that, after defining
\[|\nabla u|^2_\kappa := C_1|\nabla_x u|^2 + C_2|\nabla_\varphi u|^2 + \kappa^{1/2} c_0 K^{-2} |\partial_r u|^2,\]
we have, for any weak solution $u\in W^{1,2}_{loc}(\Omega)$ to $Lu=0$, that
\begin{multline}\label{eqSLNF}
    \|S(\overline\nabla u|\Psi^3)\|^2_{2}
    \leq C_1\iint_\Omega |\nabla_x u|^2\partial_r\Psi^3\frac{dtdx}{|t|^{n-d-1}} + \mathcal B(|\nabla u|^2_{\kappa},\Psi^3) \\
    + C_3 \bigg (\kappa \|\wt{N}(\nabla u|\Psi^{3})\|^2_{2}
    +\iint_\Omega |\nabla u|^2|\nabla_x\Psi|^2\Psi\frac{dtdx}{|t|^{n-d-2}}\bigg ).
\end{multline}
\end{Lemma}

\begin{proof}
By Lemma \ref{LSLNF}, we have three constants $C'_1,C'_2,C'_3$ depending only on $\lambda$, $d$, and $n$ such that
\begin{multline}\label{eqSLNFz}
    \|S(\overline\nabla u|\Psi^3)\|^2_{2}
    \leq C'_1\bigg (\iint_\Omega |\nabla_x u|^2\partial_r\Psi^3\frac{dtdx}{|t|^{n-d-1}} + \mathcal{B}(|\nabla_x u|^2, \Psi^3) \bigg ) + C'_2\, \mathcal{B}(|\nabla_\varphi u|^2, \Psi^3) \\ 
    + C'_3 \bigg (\kappa \|\wt{N}(\nabla u|\Psi^{3})\|^2_{2}
    +\iint_\Omega |\nabla u|^2|\nabla_x\Psi|^2\Psi\frac{dtdx}{|t|^{n-d-2}}\bigg ).
\end{multline}
We take $\delta = C'_3\kappa^{1/2}$ in Lemma \ref{LB<SN} and we obtain that
\[- \mathcal B(|\partial_r u|^2,\Psi^3) \leq C'_3 \kappa^{1/2} \|\wt N(\nabla u|\Psi)\|_2^2  + C_0 \kappa^{-1/2} K^2 \|S(\overline\nabla u|\Psi^3)\|_2^2,\]
for some $C_0$ depending on $\lambda$, $n$, that is
\begin{equation} \label{eqSLNFy} 
-  \frac12 \|S(\overline\nabla u|\Psi^3)\|_2^2  \leq  C'_3 \kappa \|\wt N(\nabla u|\Psi)\|_2^2 + \frac{\kappa^{1/2}}{2C_0K^2}  \mathcal B(|\partial_r u|^2,\Psi^3) ,
\end{equation}
Add \eqref{eqSLNFy} to \eqref{eqSLNFz} implies
\begin{multline}\label{eqSLNFzz}
    \frac12 \|S(\overline\nabla u|\Psi^3)\|^2_{2}
    \leq C'_1\iint_\Omega |\nabla_x u|^2\partial_r\Psi^3\frac{dtdx}{|t|^{n-d-1}} \\ +  \mathcal{B}(C'_1|\nabla_x u|^2 + C'_2|\nabla_\varphi u|^2 + (2C_0K^2)^{-1}\kappa^{1/2} |\partial_r u|^2, \Psi^3)  \\ 
    + 2C'_3 \bigg (\kappa \|\wt{N}(\nabla u|\Psi^{3})\|^2_{2}
    +\iint_\Omega |\nabla u|^2|\nabla_x\Psi|^2\Psi\frac{dtdx}{|t|^{n-d-2}}\bigg ).
\end{multline}
The lemma follows by taking $c_0 = (2C_0)^{-1}$, $C_1=2C'_1$, $C_2=2C'_2$, and $C_3 = 4C'_3$. 
\end{proof}

\section{Global Estimates for Energy Solutions} \label{SLtoG}

We define the weighted homogeneous Sobolev space $W$ as
\[W:= \{u\in L^1_{loc}(\Omega), \, \nabla u \in L^2(\Omega,|t|^{d+1-n}\}\]
which is equipped with the semi-norm 
\[\|u\|_W := \left(\iint_\Omega |\nabla u|^2 \frac{dt\, dx}{|t|^{d+1-n}} \right)^\frac12.\]
The space $H:= \dot{W}^{1/2,2}(\R^d)$ is the usual homogeneous space of traces on the boundary $\partial \Omega = \R^d$, equipped with the usual semi-norm
\[\|g\|_H := \left( \iint_{\R^d \times \R^d} \frac{|f(x) - f(y)|^2}{|x-y|^{d+1}} \right)^\frac12.\]
The trace of a function $u\in W$ is defined for any $x\in \R^d$ as 
\[\Tr(u)(x) := \lim_{\epsilon \to 0} \fiint_{B(x,\epsilon) \cap \Omega} u \, dy\, dt\]
if the limit exists and $+\infty$ otherwise. Equivalent definitions of trace exist, and we choose the one constructed in \cite{david2017elliptic} just because we shall refer to this manuscript for basic results.

We know from \cite{david2017elliptic}\footnote{when $d<n-1$, the case $d=n-1$ being general knowledge.} that $\Tr$ is a bounded linear operator from $W$ to $H$. Moreover, $\|.\|_W$ is a norm for the subspace $W_0:= \{u\in W, \, \Tr(u) = 0\}$, and $(W_0,\|.\|_W)$ is complete. We can apply Lax-Milgram's theorem to obtain weak solution to $Lu=0$ with prescribed data $g\in H$.

\begin{Lemma}[Lemma 9.1 and Lemma 9.4 (v)  in \cite{david2017elliptic}] \label{LaxMilgram}
Let $L=-\diver [ |t|^{d+1-n} \A \nabla]$ be an operator satisfying \eqref{ELLIP}. For each $g\in H$, there exists a unique function $u_g\in W$ such that 
\[\iint_\Omega \A \nabla u_g \cdot \nabla v \, \frac{dt\, dx}{|t|^{n-d-1}} = 0 \qquad \text{ for any } v \in W_0\]
and $\Tr(u_g) = g$. Moreover, we have the bound
\[\|u_g\|_W \leq C\|g\|_H\] 
with a constant $C$ that depends only on the elliptic constant $\lambda$ in \eqref{ELLIP}.

When $g\in H$ is continuous and compactly supported, the solution given by this lemma and the solution given by \eqref{defug} are the same. 
\end{Lemma}

In this section, we assume that $u\in W$, and we observe to which extend we can take non-compact cut-off functions in the local estimates given in Lemma \ref{LSLNF}, Lemma \ref{LlastS<N}, and Lemma \ref{LENSTP}, and thus obtain global estimates.

\begin{Lemma}\label{LN<infty}
Let $u\in W$. For any $\epsilon>0$, we have
\[\|\wt N(u|\Psi_\epsilon)\|_{2}^2 \leq C \epsilon^{-1} \|u\|_W^2\]
where $C$ depends only on $d$ and $n$.
\end{Lemma}

\begin{proof}
Take $x\in \R^d$ and $(z,r) \in \Gamma(x) \cap \{r\geq \epsilon/2\}$. Then
\[
|(\nabla u)_{W}(z,r)|^2 := \fiint_{W(z,r)} |\nabla u|^2 \, dt\, dy \lesssim \epsilon^{-1} \iint_{W(z,r)} |\nabla u|^2 \, \frac{dt\, dy}{|t|^{n-1}} \leq \epsilon^{-1} \iint_{\widehat \Gamma^*(x)} |\nabla u|^2 \, \frac{dt\, dy}{|t|^{n-1}}
\]
where $\widehat \Gamma^*(x) := \{(y,t)\in \Omega, \, |y-x| \leq C^* |t|\}$ and $C^*$ is a large constant that depends only on $d$ and $n$ and is such that $W(z,r) \in \widehat \Gamma^*(x)$ for all $(z,r)\in \Gamma(x)$. We deduce that 
\[\wt N(u|\Psi_\epsilon)(x) \lesssim \epsilon^{-1} \iint_{\widehat \Gamma^*(x)} |\nabla u|^2 \, \frac{dt\, dy}{|t|^{n-1}}\]
and then 
\[\|\wt N(u|\Psi_\epsilon)\|_2^2 \lesssim \epsilon^{-1} \iint_{\Omega} |\nabla u|^2 \, \frac{dt\, dy}{|t|^{n-d-1}} = \epsilon^{-1} \|u\|_W^2\]
by a simple variant of \eqref{S=int1}.
\end{proof}

When $u\in W$, we can see that $\|\wt N(u)\|_2^2$ will explode only when we get close to the boundary. We can use the monotone convergence of $\Psi_{B,l,\epsilon} \uparrow \Psi_\epsilon$ as $B \uparrow \R^d$ and $l\to \infty$, and then take $\Psi = \Psi_\epsilon$ in \eqref{eqSLNFF}, which shows that $\|S(\overline\nabla u|\Psi_\epsilon)\|_2^2$ is also finite and bounded by $C\epsilon^{-1}\|u\|_W^2$. 

Moreover, we can also take $\Psi = \Psi_\epsilon$ in \eqref{eqSLNF}. Indeed, any term containing $\Psi_\epsilon \nabla \Psi_B$ can be bounded by $C\epsilon^{-1} \|\nabla u\|^2_{L^2(\R^d \setminus B,|t|^{d+1-n}dtdx)}$ and any term containing $\nabla \Psi_l$ can be bounded by $Cl^{-1}\|u\|_W^2$, and both those terms converges to 0 as $B \uparrow \R^d$ and $l\to \infty$.  It means that the terms in \eqref{eqSLNF} that contain either $\nabla \Psi_B$ or $\nabla \Psi_l$ with eventually disappear when we take the limit $\Psi_{B,l,\epsilon} \uparrow \Psi_\epsilon$. 
Let us give a bit of details. In the left-hand side of \ref{eqSLNF}, we first have
\begin{multline*}
\iint_\Omega |\nabla_x u|^2\partial_r\Psi_{B,l,\epsilon}^3\frac{dtdx}{|t|^{n-d-1}} \\
 \leq \iint_\Omega |\nabla_x u|^2\partial_r\Psi^3_\epsilon \frac{dtdx}{|t|^{n-d-1}} + \iint_\Omega |\nabla_x u|^2 \Psi_\epsilon |\partial_r\Psi_B| \frac{dtdx}{|t|^{n-d-1}} + \iint_\Omega |\nabla_x u|^2 |\partial_r\Psi_l| \frac{dtdx}{|t|^{n-d-1}} \\
\leq \iint_\Omega |\nabla_x u|^2\partial_r\Psi^3_\epsilon \frac{dtdx}{|t|^{n-d-1}} + C\epsilon^{-1} \|\nabla u\|^2_{L^2(\R^d \setminus B,|t|^{d+1-n}dtdx)} + Cl^{-1}\|u\|_W^2 \\
 \longrightarrow \iint_\Omega |\nabla_x u|^2\partial_r\Psi^3_\epsilon \frac{dtdx}{|t|^{n-d-1}} \quad \text{ as $B \uparrow \R^d$ and $l\to \infty$.} 
\end{multline*}
The terms $\kappa \|\wt N(\nabla u|\Psi_{B,l,\epsilon}^3)\|^2_2$ is easily bounded by $\kappa \|\wt N(\nabla u|\Psi_{\epsilon}^3)\|^2_2$ and 
\[\iint_\Omega |\nabla u|^2|\nabla_x\Psi_{B,l,\epsilon}|^2\Psi_{B,l,\epsilon}\frac{dtdx}{|t|^{n-d-2}} \leq \frac{Cl}{r_B^2}\|\nabla u\|^2_{L^2(\R^d \setminus B,|t|^{d+1-n}dtdx)} \longrightarrow 0 \]
as long as we always take the radius $r_B$ of $B$ bigger than $l$.
The last term in the left-hand side of \ref{eqSLNF} is $\mathcal B(|\nabla u|^2_{\kappa},\Psi_{B,l,\epsilon}^3)$. We have
\begin{multline*}
|\mathcal B(|\nabla u|^2_{\kappa},\Psi_{B,l,\epsilon}^3) - \mathcal B(|\nabla u|^2_{\kappa},\Psi_{\epsilon}^3)| \\
\lesssim \iint_\Omega \big(|\nabla u|^2 + t |\nabla \overline{\nabla} u| |\nabla u|\big) \, \big( \Psi_{\epsilon}^3 |\partial_r \Psi_{B,l}^3| + |1-\Psi_{B,l}^3| |\partial_r \Psi_\epsilon^3|\big) \,  \frac{dtdx}{|t|^{n-d-2}} \\
\lesssim \iint_\Omega \big(|\nabla u|^2 + t^2 |\nabla \overline{\nabla} u|^2\big) \, \big( \Psi_{\epsilon}^3 |\partial_r \Psi_{B,l}^3| + |1-\Psi_{B,l}^3| |\partial_r \Psi_\epsilon^3|\big) \,  \frac{dtdx}{|t|^{n-d-2}} \\
\longrightarrow 0 \quad \text{ as $B \uparrow \R^d$ and $l\to \infty$}
\end{multline*}
since $\|u\|_W^2 + \|S(\overline\nabla u|\Psi_\epsilon)\|_2^2 < +\infty$.
 A similar argument gives that the term 
\[ \bigg \|\bigg (\dashint_{B_{l}(.)}\dashint_{l\leq |s|\leq 2l}|\nabla u|^2 \Psi_B^3 ds\, dy\bigg )^{1/2}\bigg \|^2_{2} \approx \int_{\R^d} \fint_{l\leq |s| \leq 2l} |\nabla u|^2 \Psi_B^3 \, ds \, dy\]
that appears in Lemma \ref{LENSTP} is bounded by $Cl^{-1}\|u\|_W$ and also converges to 0 as $l$ goes to infinity.

From those observations, Lemmas \ref{LlastS<N} and \ref{LENSTP} combined with the fact that $u\in W$ entail the following estimates.

\begin{Lemma} \label{Lglobalest}
Let $\kappa \in (0,1)$ and let $L$ be an elliptic operator satisfying \HH$_{\lambda,\kappa}$. 
There exist four constants $c_0, C_1,C_2,C_3>0$ depending on $\lambda$ and $n$ such that, if 
\[|\nabla u|^2_\kappa := C_1|\nabla_x u|^2 + C_2|\nabla_\varphi u|^2 + c_0 \kappa^{1/2} |\partial_r u|^2,\]
then for any weak solution $u\in W$ to $Lu=0$ and for any $\epsilon>0$, we have that
\begin{equation}\label{eqSLNF2}
    \|S(\overline\nabla u|\Psi_\epsilon^3)\|^2_{2}
    \leq C_1\iint_\Omega |\nabla_x u|^2\partial_r\Psi_\epsilon^3\frac{dtdx}{|t|^{n-d-1}} + \mathcal B(|\nabla u|^2_{\kappa},\Psi_\epsilon^3) \\
    + C_3 \kappa \|\wt{N}(\nabla u|\Psi_\epsilon^{3})\|^2_{2}
\end{equation}
and
\begin{equation}\label{eqSLNF3}
\|\wt N(\nabla u|\Psi_\epsilon^3)\|_2^2 \leq C_3 \|S(\overline\nabla u|\Psi_\epsilon)\|_2^2 + C_3\kappa \|\wt N(\nabla u|\Psi_\epsilon)\|_2^2.
\end{equation}
Moreover, all the quantities that appear in \eqref{eqSLNF2} and \eqref{eqSLNF3} are finite and bounded by $C\epsilon^{-1}\|u\|_W^2$.
\end{Lemma}

Note that the we only assume that $\kappa\in (0,1)$ in Lemma \ref{Lglobalest} for a technical reason (that comes from the fact that $\delta\in (0,1)$ in Lemma \ref{LB<SN}) and that condition can probably be removed. But all this does not really matter because the proof of our next result (and thus the proof of the main result of the article) requires $\kappa$ to be small anyway.
 
\begin{Th}\label{THMAIN1a}
Take $\lambda>0$. There exists $\kappa \in (0,1)$ small enough (depending only on $\lambda$, $d$ and $n$) such that if $L:=-\diver (|t|^{d+1-n}\mathcal{A}\nabla)$ is an elliptic operator satisfying \HH$_{\lambda,M,\kappa}$, then for any weak solution $u\in W$ to $Lu=0$ we have that
\begin{align}\label{eqTHM00}
\|\wt{N}(\nabla u)\|^2_{2}\leq C \limsup_{\epsilon \to 0} \int_{\R^d} \fint_{\epsilon/2 \leq |s| \leq \epsilon} |\nabla_x u|^2 \, ds\, dy,
\end{align}
where $C>0$ depends only on $\lambda$, $d$, and $n$. 

Neither the right-hand side nor the left-hand side of \eqref{eqTHM00} are guaranteed to be finite, but the left-hand side is finite as long as the right-hand side is. More precisely, there exists a sequence $\epsilon_k \to 0$ such that 
\begin{align}\label{eqTHM00b}
\|\wt{N}(\nabla u|\Psi_{\epsilon_k}^9)\|^2_{2}\leq C \int_{\R^d} \fint_{\epsilon_k/2 \leq |s| \leq \epsilon_k} |\nabla_x u|^2 \, ds\, dy.
\end{align}
\end{Th}

\begin{proof} The bound \eqref{eqTHM00} is an immediate consequence of  \eqref{eqTHM00b}, so we just need to establish the later.

Remember that $\Psi_\epsilon$ is constructed from a smooth function $\phi$. All the constant depend on the fixed $\phi$, but we have a bit of freedom (as long as we do not take $\phi$ so that $\Psi_\epsilon$ satisfies \COF  with a controlled constant $K$. Therefore, we can replace $\Psi_\epsilon$ by $\Psi_\epsilon^3$ in Lemma \ref{Lglobalest} and thus \eqref{eqSLNF3} gives that
\[\|\wt N(\nabla u|\Psi_\epsilon^9)\|_2^2 \leq C_3 \|S(\overline\nabla u|\Psi_\epsilon^3)\|_2^2 + C_3 \kappa \|\wt N(\nabla u|\Psi_\epsilon^3)\|_2^2. \] 
Observe also that $\partial_r \Psi_\epsilon^3$ is non-negative, supported in $\{\epsilon/2\leq |s| \leq s\}$, and is bounded by $C\epsilon^{-1}$, so
\[0 \leq \iint_\Omega |\nabla_x u|^2\partial_r\Psi_\epsilon^3\frac{dtdx}{|t|^{n-d-1}} \lesssim \int_{\R^d} \fint_{\epsilon/2 \leq |s| \leq \epsilon} |\nabla_x u|^2 \, ds\, dy.\]
The two last bounds combined with \eqref{eqSLNF2} entail that
\begin{equation} \label{N<Tr+B+kN}
\|\wt N(\nabla u|\Psi_\epsilon^9)\|_2^2 \leq C_4 \bigg( \int_{\R^d} \fint_{\epsilon/2 \leq |s| \leq \epsilon} |\nabla_x u|^2 \, ds\, dy + \mathcal B_+(|\nabla u|^2_{\kappa},\Psi_\epsilon^3)
   + \kappa \|\wt{N}(\nabla u|\Psi_\epsilon^{3})\|^2_{2} \bigg),
\end{equation}
where $\mathcal B_+(v,\Phi)$ is the positive part of $\mathcal B(v,\Phi)$ and $C_4$ depends only $\lambda$, $n$. The proof consists to say that if $\kappa$ is small enough, there exists $\epsilon>0$ as close as $0$ as we want such that 
\begin{equation} \label{conditionBN}
\mathcal B_+(|\nabla u|^2_{\kappa},\Psi_\epsilon^3)= 0 \quad \text{ and
} \quad 2C_4 \kappa \|\wt{N}(\nabla u|\Psi_\epsilon^{3})\|^2_{2} \leq \|\wt N(\nabla u|\Psi_\epsilon^9)\|_2^2.
\end{equation}
For such values of $\epsilon$, the bound \eqref{N<Tr+B+kN} easily self-improves to \eqref{eqTHM00b}, which is exactly our objective.

\medskip

 The rough strategy of the proof consists in studying the quantity
\begin{align}\label{ltglo06}
    \omega(\epsilon):=\iint_\Omega |\nabla u|^2_\kappa\partial_r\Psi^{3}_\epsilon \frac{dtdx}{|s|^{n-d-2}},
\end{align}
which is non-negative since $\Psi_\epsilon$ is increasing in $r$. Since $\partial_r \Psi_\epsilon$ is supported in the strip $\{\epsilon/2 \leq |s| \leq \epsilon\}$ and since $\partial_r \Psi_\epsilon \lesssim |s|^{-1}$, we deduce that 
\[\omega(\epsilon) \lesssim  \int_{\R^d} \int_{\epsilon/2 \leq |s| \leq \epsilon} |\nabla u|^2_\kappa \, \frac{ds\, dy}{|s|^{n-d-1}} \lesssim \int_{\R^d} \int_{\epsilon/2 \leq |s| \leq \epsilon} |\nabla u|^2 \, \frac{ds\, dy}{|s|^{n-d-1}} \]
with constants that depends only on $\lambda$, $d$ and $n$. But the right-hand side above converges to 0 as $\epsilon$ goes to 0 (because it is the tail of $\|u\|_W^2 < +\infty$). So we necessarily have that
\begin{align}\label{ltglo06b}
  \lim_{\epsilon \to 0}  \omega(\epsilon) = 0.
\end{align}
We shall prove that if \eqref{conditionBN} fails for every $\epsilon$ in a small neighborhood $(0,\epsilon_0]$ of zero, then \eqref{ltglo06b} does not hold. So by contraposition, \eqref{ltglo06b} implies the existence of $\epsilon$ arbitrary close to zero such that \eqref{conditionBN} holds.

\medskip

\textbf{Step 1:} For this step, we look at the implications of the fact that 
\begin{equation} \label{condB1}
\mathcal B_+(|\nabla u|^2_\kappa, \Psi_\epsilon^3) >0.
\end{equation} 
We write $I_{\mathcal B}$ for the values $\epsilon \in (0,\infty)$ for which \eqref{condB1} holds. Due to the fact that $\nabla u \in L^2_{loc}(\Omega)$ and $\partial_r(\Psi_\epsilon^3)$ is smooth (and compactly supported in $\Omega$), we have that the domain $I_{\mathcal B}$ is open. 

Recall that $\Psi_\epsilon:=\phi(\frac{\epsilon}{|t|})$, where $\phi:\mathbb{R^+}\rightarrow [0,1]$ such that $\phi\equiv 1$ on $[0,1]$ and $\phi\equiv 0$ on $[2,+\infty)$. We compute 
\[\partial_r \Psi_\epsilon = - \frac{\epsilon }{|t|^2} \phi', \qquad \partial_r\partial_r\Psi_\epsilon= \frac{2\epsilon}{|t|^3}\phi'+\frac{\epsilon^2}{|t|^4}\phi'',\]
and we notice that
\[\frac{\partial }{\partial_\epsilon}(\partial_r\Psi_\epsilon)= -\frac{1}{|t|^2}\phi'-\frac{\epsilon}{|t|^3} \phi''
= - \frac{|t|}{\epsilon} \partial_r^2\Psi_\epsilon - \frac1\epsilon \partial_r \Psi_\epsilon = - \frac{1}{\epsilon} \partial_r( |t|\partial_r \Psi_\epsilon).\]
The same argument shows that $\frac{\partial }{\partial_\epsilon}(\partial_r\Psi_\epsilon^3) = - \epsilon^{-1} \partial_r( |t|\partial_r \Psi_\epsilon^3)$. So we deduce
\[\begin{split}
    \omega'(\epsilon)& =-\frac{1}{\epsilon} \iint_\Omega |\nabla u|_\kappa^2 \partial_r\Big (|t|\partial_r\Psi^{3}_\epsilon \Big ) \frac{dt \, dy}{|t|^{n-d-2}}
    = \frac{1}{\epsilon}\iint_\Omega \partial_r \Big (|t|\nabla u|_\kappa^2\Big )\partial_r\Psi^{3}_\epsilon\frac{dt\, dy}{|t|^{n-d-2}} \\ & = -\frac2\epsilon \mathcal B(\nabla u|_\kappa^2, \Psi_\epsilon^3).
\end{split} \]
By the integration by parts in $r$ (see Proposition \ref{IBBdrdphi}), the function $\epsilon \mapsto \omega(\epsilon)$ is decreasing on $I_{\mathcal B}$, that is 
\begin{equation} \label{bdd35}
\omega(a) > \omega(b) \text{ whenever } (a,b) \subset I_{\mathcal B}.
\end{equation}

\medskip

\textbf{Step 2:} Now, we look at the implications of the fact that $\|\wt{N}(\nabla u|\Psi_\epsilon^{3})\|_{2} \gg \|\wt N(\nabla u|\Psi_\epsilon^9)\|_2
$. 

We write $\alpha$ for the universal constant $5/4$. The exact value of $\alpha$ does not matter, as long as we have $\alpha^{-1} < 1 < \alpha < \alpha^3 < 2$. We have a bit of freedom on the function $\phi$ that is used to construct $\Psi_\epsilon$ (see Definition \ref{DECUTO}). It is always possible to choose $\phi$ such that $\phi(x) = 2-x$ when 
\[1 < \alpha \leq x \leq \max\{2\alpha^{-1},\alpha^3\} <  2.\]
For the same values of $x$, we have $\phi' = 1$. If $\epsilon' \geq \alpha \epsilon$, we have $\Psi_\epsilon \geq 2-2\alpha^{-1}) >0$ on $\supp \Psi_{\epsilon'}$, hence $\epsilon' \partial_r \Psi_{\epsilon'}^3 \lesssim \Psi_\epsilon^9$ and
\[ (\epsilon')^{-1}\omega(\epsilon') \lesssim \epsilon' \iint_\Omega |\nabla u|^2 \partial_r \Psi_{\epsilon'}^3 \,  \frac{ds \, dy}{|s|^{n-d}} \lesssim  \|\wt N(\nabla u|\Psi_\epsilon^9)\|_2^2.\]
So there exists $C'>0$ depending only on $d$, $n$ (and $\alpha$) such that
\begin{equation} \label{bdd36}
\sup_{\alpha \epsilon\leq \epsilon' \leq \alpha^2\epsilon} \omega(\epsilon') \leq C' \epsilon \|\wt N(\nabla u|\Psi_\epsilon^9)\|_2^2.
\end{equation}
By the triangle inequality, we have
\[\|\wt N(\nabla u|\Psi_\epsilon^3)\|_2 - \|\wt N(\nabla u|\Psi_\epsilon^9)\|_2 \leq \|\wt N(\nabla u|(1-\Psi_\epsilon^6)\Psi^3_\epsilon)\|_2 \]
But since $(1-\Psi_\epsilon^6)\Psi^3_\epsilon$ is supported in the strip $\{\epsilon/2 \leq |s| \leq \epsilon\}$, it is fairly easy to see that
\[\|\wt N(\nabla u|(1-\Psi_\epsilon^6)\Psi^3_\epsilon)\|_2^2 \lesssim \int_{\R^d} \fint_{\epsilon/2 \leq |s| \leq \epsilon} |\nabla u|^2 \, ds\, dy \]
Take now $\epsilon' \in [\alpha^{-2} \epsilon, \alpha^{-1}\epsilon]$. With our choice of $\phi$ and $\alpha$, we have that $\Psi_\epsilon + \epsilon' \partial_r \Psi_{\epsilon'}$ is bounded from below by a uniform constant (i.e. depend only on $\alpha$) on $[\epsilon/2,\epsilon]$. 
It implies that
\begin{multline}
\int_{\R^d} \fint_{\epsilon/2 \leq |s| \leq \epsilon} |\nabla u|^2 \, ds\, dy \lesssim \|\wt N(\nabla u|\Psi_\epsilon^9)\|_2^2 + (\epsilon')^{-1} \iint_\Omega |\nabla u|^2 \partial_r \Psi^3_{\epsilon'} \frac{ds\, dy}{|s|^{n-d-2}} \\
 \lesssim \|\wt N(\nabla u|\Psi_\epsilon^9)\|_2^2 + \kappa^{-1/2} (\epsilon')^{-1}\omega(\epsilon'),
\end{multline}
where, in the last line, the coefficient $\kappa^{-1/2}$ appears because $\omega$ is defined using the $|\nabla u|_\kappa^2$. 
By combining the last three computations, we obtain the existence of a constant $C''>0$ depending on $n$ and $\lambda$ (and $\alpha$) such that 
\begin{equation} \label{bdd37}
\epsilon \|\wt N(\nabla u|\Psi_\epsilon^3)\|_2^2 \leq C'' \epsilon \|\wt N(\nabla u|\Psi_\epsilon^9)\|_2^2 + C''\kappa^{-1/2} \inf_{\epsilon/\alpha^2 \leq \epsilon' \leq \epsilon/\alpha} \omega(\epsilon').
\end{equation}
We say that $\epsilon \in I_{N}$ if 
\begin{equation} \label{bdd38}
\|\wt N(\nabla u|\Psi_\epsilon^3)\|_2^2 > C''(1+\kappa^{-1/2}C')  \|\wt N(\nabla u|\Psi_\epsilon^9)\|_2^2.
\end{equation}
If $\epsilon \in I_N$, then we have by \eqref{bdd36}, and then \eqref{bdd37}--\eqref{bdd38} that 
\begin{equation} \label{bdd39}
\sup_{\alpha \epsilon\leq \epsilon' \leq \alpha^2\epsilon} \omega(\epsilon') \leq C'  \epsilon \|\wt N(\nabla u|\Psi_\epsilon^9)\|_2^2 < \inf_{\epsilon/\alpha^2 \leq \epsilon' \leq \epsilon/\alpha} \omega(\epsilon').
\end{equation}

\medskip

\textbf{Step 3:} We want to prove that the point $0$ is in the closure of 
\begin{equation} \label{bdd40}
(0,+\infty) \setminus (I_\B \cup I_N).
\end{equation}
Indeed, we take $\kappa$ (that depends on $\lambda$, $d$, and $n$) such that $2C_M\kappa C''(1+\kappa^{-1/2}C') \leq 1$, and then we have \eqref{conditionBN} for any value of $\epsilon \in (0,+\infty) \setminus (I_\B \cup I_N)$. If the claim \eqref{bdd40} is true, then we can find values of $\epsilon \in (0,+\infty) \setminus (I_\B \cup I_N)$ arbitrary close to 0, and for those values, \eqref{N<Tr+B+kN} self-improves to the desired bound \eqref{eqTHM00b}.

\medskip

The argument is a bit technical because we have to combine \eqref{bdd35} - which shows that $\omega(\epsilon)$ `continuously increases'  as $\epsilon \to 0$ - with \eqref{bdd39} - which implies that $\omega(\epsilon)$ increases by `jumps' when $\epsilon$ get closer to 0. The quantity $\omega(\epsilon)$ may not be increasing as $\epsilon\to 0$, but a subsequence will be increasing, and that is enough for us.

We decided the write the (simpler) arguments that show that zero is in the closure of both $(0,+\infty) \setminus I_\B$ and $(0,+\infty) \setminus I_N$. These simpler arguments are not necessary for the proof, but we hope it will help the reader understand later what we are doing when we look at $(0,+\infty) \setminus (I_\B \cup I_N)$.

\medskip

\textbf{Step 3(a):} We claim that $0$ is in the closure of $(0,+\infty) \setminus I_\B$. Indeed, if it is not the case, then there exists $\epsilon_0$ such that $(0,\epsilon_0] \in I_\B$. The bound \eqref{bdd35} implies then that $\omega(\epsilon) \geq \omega(\epsilon_0/2) >0$ for all $\epsilon\in (0,\epsilon_0/2)$, which contradicts \eqref{ltglo06b}.

\medskip

\textbf{Step 3(b):} We claim that $0$ is in the closure of $(0,+\infty) \setminus I_N$. Indeed, if it is not the case, then there exists $\epsilon_0$ such that $(0,\epsilon_0] \in I_N$. The bound \eqref{bdd39} implies then that $\omega(\epsilon/\alpha^{2k+1}) \geq \omega(\epsilon_0/\alpha) >0$ for all $k\in \mathbb N$, which contradicts \eqref{ltglo06b}.

\medskip

\textbf{Step 4:} The argument is similar to the one done in the proof of Lemma 7.8 from \cite{feneuil2018dirichlet}, but we try to give a clearer presentation.

\medskip

We assume that $(0,\epsilon_0] \subset I_\B \cup I_N$, and we want to prove that \eqref{ltglo06b} does not hold. Because of Step 3(a), we can also assume that zero is in the closure of $(0,\infty) \setminus I_{\B}$, meaning that for any $0<\epsilon<\epsilon_0$, we have
\begin{equation} \label{bdd41}
\inf\{\epsilon' >0,\, (\epsilon',\epsilon) \subset I_\B\} >0.
\end{equation}
We write 
\[\delta(\epsilon):= \inf_{\epsilon\leq \epsilon' \leq \alpha \epsilon} \omega(\epsilon')\]
and we want to construct $\epsilon_k$ such that that $\epsilon_{k+1} \leq \epsilon_k/\alpha^2$ and $\delta(\epsilon_{k+1}) > \delta(\epsilon_{k})$.

\medskip

{\bf Induction step.} We have $(0,\epsilon_k] \subset I_\B \cup I_N$.
\begin{enumerate}
\item If we have $(\epsilon_k/\alpha^2,\epsilon_k) \subset I_\B$, then by \eqref{bdd41}, we have 
\[0 < \epsilon_{k+1}:= \inf\{\epsilon, \, (\epsilon,\epsilon_k) \subset I_B\} \leq  \epsilon_k/\alpha^2.\] 
Now, thanks to \eqref{bdd35}, we also get
\[\delta(\epsilon_{k+1}) := \inf_{\epsilon_{k+1}\leq \epsilon' \leq \alpha \epsilon_{k+1}} \omega(\epsilon') = \omega(\alpha\epsilon_{k+1}) > \omega(\epsilon_k)\geq \delta(\epsilon_k).\]

\medskip

\item If (1) is false, then take again $\tilde \epsilon:= \inf\{\epsilon, \, (\epsilon,\epsilon_0) \subset I_B\} > \epsilon_0/\alpha^2$.
Since $I_\B$ is open, we necessarily have $\tilde \epsilon \notin I_\B$, which forces $\tilde \epsilon \in I_N$. Since $\epsilon_k/\alpha^2< \tilde \epsilon \leq \epsilon_k$, the intersection $[\alpha \tilde \epsilon_,\alpha^2\tilde \epsilon] \cap [\epsilon_k,\alpha\epsilon_k]$ is not empty and contains $\tilde{\tilde\epsilon}$. So if we choose $\epsilon_{k+1} = \tilde \epsilon/\alpha^2$ we have by \eqref{bdd39} that
\[\delta(\epsilon_{k+1}) > \sup_{\alpha \tilde \epsilon \leq \epsilon' \leq \alpha^2 \tilde \epsilon} \omega(\epsilon') \geq \omega(\tilde{\tilde\epsilon}) \geq \delta(\epsilon_k).\]
\end{enumerate}

By construction, the value of $\omega(\epsilon_k)$ will be bigger than $\delta(\epsilon_1)>0$ for any $k\geq 1$, which means that the convergence \eqref{ltglo06b} does not hold. To summarize, we established that if zero is not in the closure of $(0,+\infty) \setminus (I_\B \cup I_N)$, which means that there exists $\epsilon_0>0$ such that $(0,\epsilon_0] \subset I_\B \cup I_N$, then \eqref{ltglo06b} fails. By contraposition, the convergence \eqref{ltglo06b}  - which holds because $u\in W$ - implies that zero is in the closure of $(0,+\infty) \setminus (I_\B \cup I_N)$. The theorem follows.
\end{proof}

\section{Approximation by operators with Lipschitz coefficients.} \label{Sapprox}

With Theorem \ref{THMAIN1a}, we get closer to Theorem \ref{THRE2MA}, which is our objective. We ``just'' need to prove that if $u_g \in W$ is the solution given by Lemma \ref{LaxMilgram} for $g \in H$ satisfying $\|\nabla g\|_{L^2(\R^d)} < \infty$, we have
\begin{equation} \label{Trepstog}
\limsup_{\epsilon \to 0} \int_{\R^d} \fint_{\epsilon/2 \leq |s| \leq \epsilon} |\nabla_x u|^2 \, ds\, dy \leq C \int_{\R^d} |\nabla g|^2 \, dy.
\end{equation}
However, the above convergence is not a simple fact. In some sense, it is a weaker version of Theorem \ref{THRE2MA} that only consider the values of $u$ as close as the boundary as we want. The strategy consists of smoothing the coefficients of $\A$ in a small tube close to the boundary while satisfying \HH{} with uniform constants. For those operators, the convergence \eqref{Trepstog} hence Theorem \ref{THRE2MA} will hold with uniform constants. But since the coefficients are modified only an a small enough set, the solutions to the modified operators will converge to the solution of the initial operator, and we eventually are able to prove Theorem \ref{THRE2MA}.

\medskip

First, we show that the weak solutions are in $W^{2,2}_{loc}$ (see Proposition \ref{LW22loc}), which means that taking second derivative is allowed, and so the square functional $S$ and its local version make sense. We should have given this argument long ago, but this result is already well known, and we decided to write it just as an introduction for Proposition \ref{WT01}. In Proposition \ref{WT01}, we establish a global bound for the tangential derivatives: we show that if the boundary data and the coefficients of $\A$ are smooth enough, any weak solution satisfies $\nabla_x u \in W$. As a corollary (Proposition \ref{APRP03}) we prove a technical lemma stating that the ``approximation of the trace'' of a solution converges to the actual trace.

In the end of the section (Theorem \ref{PrAPA}), we establish that any elliptic operator satisfying \HH{} can be approximated by operators with smooth enough coefficients (so that the global estimates given in Proposition \ref{WT01} apply) that satisfy \HH{} with constants controlled by the ones of the approximated operator.

In the next section, we establish the convergence of solutions of approximating operators and then we combine Theorem \ref{THMAIN1a}, Theorem \ref{PrAPA}, and Proposition \ref{APRP03} to obtain Theorem \ref{THRE2MA}.

\begin{Prop}\label{LW22loc}
Let $L:=-\diver [A\nabla]$ be an elliptic operator defined on $D$ such that $A$ satisfies \eqref{ELLIP} and $\nabla A \in L^\infty_{loc}(D)$. Then any weak solution $u\in W^{1,2}_{loc}(D)$ to $Lu = 0$ in $D$ also lies in $W^{2,2}_{loc}(D)$.
\end{Prop}

\begin{Rem}
Observe that the quantities $\nabla \partial_x$, $\nabla \partial_\varphi$ and $\nabla \partial_r$ are a linear combination of second order derivatives and first order derivatives, so those derivatives are (locally in $L^2$) well defined for any weak solution to $Lu=0$ for which $L$ satisfies \HH.
\end{Rem}

\begin{proof}
The proof is classical, but we could not pinpoint a good reference, so since the proof is quite simple and will be a good introduction to the global analogue, we decided to write it.

Take $E\Subset D$ and then $\Psi \in C^\infty_0(D)$ such that $\Psi \equiv 1$ on $E$ and $0\leq \Psi \leq 1$. Pick a unit vector $e\in \R^n$. Define when $0 < h < \dist(\supp\Psi,D^c)/2$ the operator $\Delta_e^h$ as
\[\Delta_e^h u := \frac{u(X+he) - u(X)}{h} \in W^{1,2}_{loc}(D)\]
We want to prove that 
\begin{equation} \label{claimW22loc}
I_h = \iint_D |\nabla (\Delta^h_e u)|^2 \Psi^2 \, dX \leq C_{E,u},
\end{equation} 
with a bound $C_{E,u}$ independent of $h$. Indeed, once the claim \eqref{claimW22loc} is established, by the weak compactness of the unit ball in $L^2$ we can extract a sequence $h_m \in (0,1)$ such that  $\nabla (\Delta_e^{h_m} u)$ converges weakly in $L^2(E)$ (and thus in the sense of distribution). But we know that $\nabla (\Delta_e^{h_m} u)$ has to converge to $\nabla \partial_e u$ in the sense of distribution - where of course $\partial_e$ is the derivative in the direction $e$ - so the weak limit of the $\nabla (\Delta_e^{h_m} u)$ is $\nabla \partial_e u$, which is now in $L^2(E)$ and satisfies $\int_E |\nabla \partial_e u|^2 \, dX \leq C_{E,u}$. Since we have the bound for every compact subset $E$ and any direction $e$, we conclude that $W^{2,2}_{loc}(D)$ as desired.

\medskip

So it remains to prove  \eqref{claimW22loc}. Since $u$ is a weak solution to $-\diver A\nabla u = 0$ in $D$, and that $X\to \Psi(X-he)$ still lies in $C^\infty_0(D)$, we deduce
\[\begin{split}
0 & = \iint_D \frac{A(X+he)\nabla u(X+he) - A(X)\nabla u(X)}{h} \cdot \nabla[ \Psi^2 \Delta_h^e u ] \, dX \\
& = \iint_D (\Delta_e^h A)\nabla u(X+he) \cdot \nabla[ \Psi^2 \Delta_h^e u ] \, dX + \iint_D A \nabla \Delta_e^h u \cdot \nabla[ \Psi^2 \Delta_h^e u ] \, dX
\end{split}\]
We use the product rule to write $\nabla[ \Psi^2 \Delta_h^e u ] = 2 \nabla \Psi \,  (\Psi \Delta_h^e u) + \nabla \Delta_h^e u \, (\Psi^2)$ we obtain 
\begin{multline*}
I_h= - 2 \iint_D A \nabla \Delta_e^h u \cdot \nabla \Psi \, (\Psi \Delta_h^e u)  \, dX - 2 \iint_D (\Delta_e^h A)\nabla u(X+he) \cdot \nabla \Psi \,  (\Psi \Delta_h^e u) \, dX \\
 -  \iint_D (\Delta_e^h A)\nabla u(X+he) \cdot \nabla \Delta_h^e u \, (\Psi^2)  \, dX
\end{multline*}
We use the identity $ab \leq \epsilon a^2 + b^2/4\epsilon$ on each of the three terms of the right-hand side above, and we group the similar terms. Afterwards, we get that 
\[I_h \leq \frac12 I_h + C \left( (1+\|A\|_\infty^2) \iint_D |\Delta_h^e u|^2 |\nabla\Psi|^2 \, dX + \|\nabla A\|_{L^\infty(F)}^2 \iint_D |\nabla u(X+he)|^2 \Psi^2 \, dX \right) \]
where $F \Subset D$ is the set of points at distance at most $\dist(\supp\Psi,D^c)/2$ from $\supp \Psi$. Since $I_h$ is finite, we hide the term  $\frac12 I_h$ in the right-hand side, as for the two other terms, we observe that they are easily bounded - up to a constant that depends only on $A$ and $\Psi$ - by $\iint_F |\nabla u|^2 \, dX$. The claim \eqref{claimW22loc} and the proposition follows.
\end{proof}

The global analogue of the previous proposition is the following result.

\begin{Prop}\label{WT01}
Let $L:=-\diver (|t|^{d+1-n}\mathcal{A}\nabla)$ be an elliptic operator that satisfies the uniform ellipticity condition (\ref{ELLIP}). Suppose that $u\in W$ is a weak solution of $Lu=0$ and $\Tr (u)=g\in C^\infty_0(\partial\Omega)$. If $\|\nabla \mathcal{A}_x \|_{L^\infty(\Omega)}<\infty$, then $\nabla_x u \in W$. More precisely, 
\begin{equation} \label{WT01a}
\iint_\Omega |\nabla \nabla_x u|^2\frac{dt}{|t|^{n-d-1}}dx\lesssim 
\int_{\partial\Omega}|\nabla g|^2dx +\int_{\partial\Omega}|\nabla \nabla g|^2dx
+ \iint_{\Omega} |\nabla u|^2\frac{dtdx}{|t|^{n-d-1}},
\end{equation}
where the implicit constant depends on elliptic constant $\lambda$,the  dimensions $d$ and $n$, and $\|\nabla \mathcal{A}\|_{L^\infty(\Omega)}$. 
\end{Prop}

\begin{Rem}
In the codimension 1 case, where we do not have angular derivatives, we can deduce a bound on the full set of second derivatives by using the equation (see Proposition \ref{PpOPER}, that allow us to write the second order radial derivative $\partial_r^2$ of a solution as a linear combination of first order derivatives and $\nabla \nabla_x$). However, we did not succeed to bound globally the angular derivatives, so we did not succeed to show that $\nabla u \in W$. 
\end{Rem}

\begin{proof} Let $v(x, t):=u(x, t)-e^{-|t|}g(x)$. We first notice that $e^{-|t|}g\in W$ and $\Tr(e^{-|t|}g) = g$, hence $v\in W_0$. Consider the difference quotient in the tangential direction $e_i\in \mathbb{R}^d$ such that
\begin{align*}
\Delta_i^h v(x,t)=\frac{v(x+he_i,t)-v(x,t)}{h}, \ \ \  h\neq 0.
\end{align*}
We can easily see that $\Delta_i^h v \in W_0$ for each fixed $h\neq 0$. In particular, the quantity $J(v)$ defined as
\begin{align*}
J(v):=\iint_\Omega|\nabla \Delta^h_i v|^2\frac{dt}{|t|^{n-d-1}}dx.
\end{align*}
is finite. We turn to the bound of $J(v)$. By uniform ellipticity of $\mathcal{A}$, 
\begin{align}\label{J121}
J(v)\lesssim \iint_\Omega \mathcal{A}\nabla (\Delta^h_i v)\cdot \nabla (\Delta^h_i v)\frac{dt}{|t|^{n-d-1}}dx.
\end{align}
Since $\Delta_i^h v\in W_0$ and $u$ is a weak solution to the equation $Lu=0$, we have,
\begin{multline*}
0=\iint_{\Omega} \frac{\mathcal{A}(x+he_i, t)\nabla u(x+he_i, t)-\mathcal{A}(x, t)\nabla u(x, t)}{h}\cdot \nabla (\Delta^h_i v)\frac{dt}{|t|^{n-d-1}}dx\\
=\iint_{\Omega} (\Delta_i^h \mathcal{A})\nabla u(x+he_i, t)\cdot \nabla (\Delta^h_i v)\frac{dt}{|t|^{n-d-1}}dx\\
+\iint_{\Omega} \mathcal{A}(x, t)\Big (\Delta_i^h\nabla u(x, t)\Big )\cdot \nabla (\Delta^h_i v)\frac{dt}{|t|^{n-d-1}}dx.
\end{multline*}
Since two operators $\Delta_i^h, \nabla$ commute and $u(x)=v(x)+e^{-|t|}g(x)$, the identity above implies that
\begin{multline*}
J=-\iint_\Omega \mathcal{A}\nabla (e^{-|t|}\Delta_i^h g )\cdot \nabla (\Delta_i^h v)\frac{dt}{|t|^{n-d-1}}dx\\
-\iint_\Omega (\Delta^h_i \mathcal{A})\nabla u(x+he_i, t)\cdot \nabla (\Delta^h_i v)\frac{dt}{|t|^{n-d-1}}dx=J_{1}+J_{2}.
\end{multline*}

The term $J_{1}$ can be bounded by
\[J_{1}
\leq \frac{\|\mathcal{A}\|^2_{L^\infty(\Omega)}}{4\epsilon}\iint_\Omega |\nabla (e^{-|t|} \Delta_i^h g)|^2\frac{dt}{|t|^{n-d-1}}dx+\epsilon J(v).\]
and 
\[J_{2}
\leq \epsilon J(v)+\frac{\|\Delta^h_i \mathcal{A}\|^2_{L^{\infty}(\Omega)}}{4\epsilon}\iint_\Omega |\nabla u(x+he_i, t)|^2\frac{dt}{|t|^{n-d-1}}dx.\]
The term $\epsilon J(v)$ can be hidden to the left-thand side of (\ref{J121}) by choosing $\epsilon$ small enough. Moreover, the mean value inequality infers that $\|\Delta^h_i \mathcal{A}\|_{\infty} \leq \|\nabla_x \A\|_\infty$. Consequently,
\begin{equation}\label{GTJ12}
J(v)\lesssim \iint_\Omega |\nabla (e^{-|t|} \Delta_i^h g)|^2\frac{dtdx}{|t|^{n-d-1}} +  \iint_{\Omega} |\nabla u(x+he_i, t)|^2\frac{dtdx}{|t|^{n-d-1}}:= J_3 + J_4.
\end{equation}
After a change of variable, $J_4$ is just $\|u\|_W^2 = \iint_\Omega |\nabla u|^2 |t|^{d+1-n} dtdx$. The term $J_3$ is bounded brutally by using cylindrical coordinate as follows
\begin{align}\label{eqWT02}
J_3 \lesssim \left( \int_0^\infty e^{-2r} dr\right) \int_{\R^d} \big(|\Delta^h_i g|^2 + |\nabla \Delta^h_i g|^2 \big) dx \lesssim \int_{\R^d} (|\nabla g|^2 + |\nabla \nabla g|^2) \, dx.
\end{align}
We just prove \eqref{WT01a} with the rate of change $\Delta^h_i$ instead of the derivative $\partial_{x_i}$, and with a constant independent of $h$. So by the same compactness argument as the one given in Proposition \ref{LW22loc}, \eqref{WT01a} follows.  
\end{proof}

We study the sufficient conditions to define $\Tr (\nabla_x u)$ in the next proposition.  

\begin{Prop}\label{APRP03}
Let $L:=-\diver (|t|^{d+1-n}\mathcal{A}\nabla)$ be an elliptic operator that satisfies the uniform ellipticity condition (\ref{ELLIP}). Suppose that $u\in W$ is a weak solution of $Lu=0$ and $\Tr (u)=g\in C^\infty_0(\partial\Omega)$. If $\|\nabla_x \mathcal{A}\|_{L^\infty(\Omega)}<\infty$, then $\Tr (\nabla_x u)=\nabla g$ almost everywhere, and in particular
\begin{equation} \label{limTr=g}
\lim_{\epsilon\to 0} \int_{\R^d} \fint_{\epsilon/2 < |t| \leq \epsilon} |\nabla_x u|^2 \, dt \, dx = \int_{\R^d} |\nabla g|^2 \, dx. 
\end{equation}
\end{Prop}

\begin{proof}
The exact definition of trace is not always the same (but it is well known that the different definitions are equivalent, as we shall show). We prove the result in the case $d<n-1$, which has way less background, with the trace introduced in \cite{david2017elliptic}. We let the reader check that proof in the case $d=n-1$ is analogous with any reasonable notion of trace.

The trace of a function $u$ in $W$ is defined as in \cite{david2017elliptic} by
\[g(x) = \Tr(u)(x) := \lim_{\epsilon \to 0} \fiint_{B(x,\epsilon)} u \, dy \, dt.\]
The definition is valid because, $|B(x,\epsilon)\cap \Omega^c| = 0$ when $d<n-1$, but that does not matter much, because in the case $d=n-1$, we can simply extend $u \in W$ from $\R^n_+$ to $\R^n$ by symmetry. If we set $g_\epsilon(x) := \fiint_{B(x,\epsilon)} u \, dy \, dt$
then (3.24) in \cite{david2017elliptic} shows that 
\begin{equation} \label{gepstog}
\|g_\epsilon - g\|_{L^2(\R^d} \lesssim \epsilon^\alpha \|u\|_W \quad \text{ for any $\alpha \in (0,1/2)$}.
\end{equation}

We pick now a smooth nonnegative function $\theta \in C^\infty_0(\R^{d+1})$ such that $\theta \equiv 0$ outside $B(0,\frac14)$ and $\int_{\R^n} \theta = 1$. We define 
\[\theta_{x,\epsilon}(y,t) := \epsilon^{-n} \theta\Big(\frac{y-x}{\epsilon},\frac{|s|- 3\epsilon/4}{\epsilon}\Big)\]
and then
\[\Tr_\epsilon(u)(x) := \iint_{\R^n} \theta_{x,\epsilon}(y,t) u(y,t) \, dt\, dy.\]
We show that $\Tr_\epsilon(u)$ is a good substitute of $g_\epsilon$. Indeed, since $\fint \theta_{x,\epsilon} = 1$ and $\theta_{x,\epsilon}$ is supported in $B(x,\epsilon)$, we have 
\begin{multline*}|\Tr_\epsilon(u)(x) - g_\epsilon(x)| \leq \iint_{\R^n} \theta_{x,\epsilon} |u(y,t) - g_\epsilon(x)| \, dy\, dt \lesssim \fiint_{B(x,\epsilon)} |u(y,t) - g_\epsilon(x)| \, dy\, dt \\ \lesssim \left(\epsilon^2 \fiint_{B(x,\epsilon)} |\nabla u(y,t) |^2 \, dy\, dt\right)^\frac12 \lesssim \left(\epsilon^{1-d} \iint_{B(x,\epsilon)} |\nabla u(y,t) |^2 \, dy\, \frac{dt}{|t|^{n-d-1}}\right)^\frac12 
\end{multline*}
by using the $L^2$-Poincar\'e inequality and then (2.13) in \cite{david2017elliptic}. So we have by Fubini's theorem
\begin{multline}\|\Tr_\epsilon(u) - g_\epsilon\|_2^2 = \epsilon^{1-d} \int_{\R^d} \iint_{B(x,\epsilon)} |\nabla u(y,t) |^2 \, \frac{dt\, dy}{|t|^{n-d-1}}\, dx \\ \lesssim \epsilon \iint_{\R^n} |\nabla u(y,t) |^2 \, \frac{dt\, dy}{|t|^{n-d-1}} = \epsilon \|u\|_W.\end{multline}
Together with \eqref{gepstog}, we deduce
\begin{equation} \label{gepstog2}
\|\Tr_\epsilon(u) - \Tr(u)\|_{L^2(\R^d} \lesssim \epsilon^\alpha \|u\|_W \quad \text{ for any $\alpha \in (0,1/2)$}.
\end{equation}
The above inequality shows that $\Tr_\epsilon(u)$ converges to $\Tr(u)$ in $L^2$, so also in the sense of distribution. Therefore $\nabla_x \Tr_\epsilon(u)$ converges to $\nabla \Tr(u)$ in the sense of distributions. Moreover, since $\nabla_x u \in W$, we similarly have that $\Tr_\epsilon(\nabla_x u)$ converges to $\Tr(\nabla_x u)$ in the sense of distributions. But since we easily have by definition of $\Tr_\epsilon(u)$ that 
\begin{multline}
\nabla_x \Tr_\epsilon(u) = - \iint_{\R^n}  \nabla_x \Big[ \epsilon^{-n} \theta\Big(\frac{y-x}{\epsilon},\frac{|s|- 3\epsilon/4}{\epsilon}\Big) \Big] u(y,t) \, dt \, dy \\ 
= - \iint_{\R^n} \nabla_y \theta_{x,\epsilon} (y,t)  u(y,t) \, dt \, dy = \iint_{\R^n} \theta_{x,\epsilon} (y,t ) \nabla_y u(y,t) \, dt \, dy = \Tr_\epsilon(\nabla_x u),
\end{multline}
then by uniqueness of the limit, we deduce that $\Tr(\nabla_x u) = \nabla \Tr(u) = g$.

It remains to prove \eqref{limTr=g}. Observe that
\[\int_{\R^d} \fint_{\epsilon/2 < |t| \leq \epsilon} |\nabla_x u|^2 \, dt \, dx \leq \int_{\R^d} \fiint_{W(x,\epsilon)} |\nabla_x u|^2 \, dt\, dy \, dx.\]
The function $h_\epsilon(x) := \fiint_{W(x,\epsilon)} \nabla_x u \, dt\, dy$ is similar to $\Tr_\epsilon(\nabla_x u)$, so we can repeat the argument used to obtain \eqref{gepstog2}, and we have
\[ \lim_{\epsilon \to 0} \int_{\R^d} \left| \fiint_{W(x,\epsilon)} \nabla_x u \, dt\, dy - \nabla g  \right|^2  \, dx = \lim_{\epsilon \to 0} \int_{\R^d} \left| \fiint_{W(x,\epsilon)} \nabla_x u \, dt\, dy - \Tr(\nabla_x u)  \right|^2  \, dx =  0 \]
The combination of the two last computations easily implies \eqref{limTr=g}, which ends the proof of the proposition.
\end{proof}

\begin{Th}\label{PrAPA} Let $L = - \diver (|t|^{d+1-n} \A \nabla)$ be an elliptic operator that satisfies \HH$_{\lambda,\kappa}$. Then there exists a sequence of $\{\mathcal{A}^j\}_{j\in \mathbb{N}}$ such that 
\begin{enumerate}[label=(\alph*)]
\item the convergence $\A^j \to \A$ holds uniformly on compact sets of $\Omega$;
\item $\|\nabla \A^j\|_{L^\infty(\Omega)} \leq Cj\kappa^\frac12$;
\item the operator $L_j:= -\diver (|t|^{d+1-n} \A^j \nabla )$ satisfies \HH$_{\lambda,C\kappa}$.
\end{enumerate}
where in both cases, $C>0$ is a constant that depends only on $d$ and $n$.
\end{Th}

\begin{Rem}\label{REM.APA}
We adapt the construction from Lemma 7.12 in \cite{kenig1993neumann} to the higher co-dimensional boundaries, that is a construction that smoothens the coefficients of $\A$ while preserving the form of the matrix, the (constant of the) Carleson measure conditions on the coefficients, and the ellipticity constant of $\A$. Note that the construction does not rely on the specific structure \eqref{coe.afor}. 
\end{Rem}

\begin{proof}
Suppose that $\psi\in C_0^\infty(\mathbb{R}_+)$ with $0\leq \psi\leq 1$ such that $\psi\equiv 1/2$ on $[2,\infty]$ and $\psi\equiv 0$ on $[0,1]$. For $j\in \mathbb N$, set $\psi_j(t):=\psi(j|t|)$. We construct the matrix $\mathcal{A}^j$ as follows:
\begin{align}\label{eqPrAP10}
    \mathcal{A}^j(x,t)=\psi_j(t)\mathcal{A}(x,t)+(1-\psi_j(t))\fiint_{W(x,\frac{1}{j})}\mathcal{A}(x',t')\, dx'\, dt'.
\end{align}
From the construction above, we observe that $\mathcal{A}^j\rightarrow \mathcal{A}$ uniformly on compact sets of $\Omega$, which is a direct consequence of the uniform convergence $\psi_j\rightarrow 1$ on compact sets. 
The fact that the structure \eqref{coe.afor} is transferred to $\A_j$ and the ellipticity bound \eqref{defellip} on $\A_j$ (with the same constant as $\A$) is an immediate consequence of the fact that each coefficient in $\A_j$ is an average of some value of the same coefficient in $\A$. 

It remains to estimate $\nabla \A^j$, we want to show that $|t||\nabla_{x,\varphi} \A^j| \in CM(C\kappa)$, $|t||\partial_r \A^j| \in CM(CM)$, and  $|\nabla \A^j| \lesssim j$. Since $\psi_j$ is $x$-independent, we have
\[  \nabla_{x,\varphi} \mathcal{A}^j(x,t)=\psi_j(t)\nabla_{x,\varphi}\mathcal{A}(x,t)+(1-\psi_j(t))\nabla_{x} \Big (\fiint_{W(x,\frac{1}{j})}\mathcal{A}(x',t')dx'dt'\Big )\\
 =:\RN{1}_1+\RN{1}_2.\]
According to \eqref{PpCAL} and the fact that $|t||\nabla_{x,\varphi} \A| \in CM(\kappa)$, we have $\|t\nabla_{x,\varphi} \mathcal{A}\|^2_{\infty}\lesssim \kappa$ and thus
\begin{equation}\label{eqPrAP18y}
\|\1_{|t|\geq 1/2j} \nabla_{x,\varphi} \mathcal{A} \|^2_{L^\infty(\Omega)}\lesssim j^2\kappa.
\end{equation}
But for $h$ small enough, we have
\begin{multline}\label{eqPrAP18}
    \frac{1}{|h|}\Big |\fiint_{W_a(x+h,\frac{1}{j})}\mathcal{A}(x',t')dx'dt'-\fiint_{W_a(x,\frac{1}{j})}\mathcal{A}(x',t')dx'dt'\Big |\\
    =\Big |\fiint_{W_a(x,\frac{1}{j})}\frac{\mathcal{A}(x'+h,t')-\mathcal{A}(x',t')}{|h|}dx'dt'\Big |\lesssim \|\1_{|t|\geq j} \nabla_x \mathcal{A}\|_{L^\infty(\Omega)} \lesssim j \kappa^{1/2},
\end{multline}
where the righthand side above is independent of $h$. By taking the limit $h\rightarrow 0$ in (\ref{eqPrAP18}), we obtain that
\[\nabla_x\Big (\fiint_{W_a(x,\frac{1}{j})}\mathcal{A}(x',t')dx'dt'\Big ) \lesssim j \kappa^{1/2} \]
and then
\begin{equation} \label{eqPrAP18z}
|\RN{1}_2| \lesssim (j \kappa^{1/2}) \1_{|t| \leq 2/j}.
\end{equation}
We deduce
\begin{equation} \label{eqPrAP18x}
\|\nabla_{x,\varphi} \A^j\|_\infty \leq \|(\nabla_{x,\varphi} \A) \1_{|t|\geq 1/j} \|_\infty + \|\RN{1}_2\|_\infty \lesssim j\kappa^{1/2}
\end{equation}
by \eqref{eqPrAP18y} and \eqref{eqPrAP18z}, and 
\begin{equation} \label{eqPrAP18w}
|t||\nabla_{x,\varphi} \A^j| \leq |t||\RN{1}| + |t||\RN{1}_2| \lesssim |t||\nabla_{x,\varphi} \A| + \frac{|t|}{j} \kappa^\frac12 \1_{|t| \leq 2/j} \in CM(C\kappa)
\end{equation}
because $|t||\nabla_{x,\varphi} \A| \in CM(\kappa)$ and a simple computation shows that $\frac{|t|}{j} \kappa^\frac12 \1_{|t| \leq 2/j} \in CM(4\kappa c_{n-d})$, where $c_{n-d}$ is the surface of the unit sphere in $\R^{n-d}$. 

We have the two desired bounds \eqref{eqPrAP18x} and \eqref{eqPrAP18w} on $\nabla_{x,\varphi}\A^j$, and it remains to prove the analogue estimates on $\partial_r\A^j$. We have that 
\begin{multline*}  |\partial_r \mathcal{A}^j|= \left| \psi_j \partial_r \mathcal{A} + \partial_r \psi_j \Big ( \A -  \fiint_{W(x,\frac{1}{j})}\mathcal{A}(x',t')dx'dt' \Big ) \right| \\
\leq \1_{|t|\geq 1/j}|\partial_r \A| + j \1_{1\leq j|t| \leq 2} \|\1_{|t|\geq 1/2j}\nabla \A\|. \end{multline*}
And since, similarly to \eqref{eqPrAP18y}, we have $\|\1_{|t|\geq j} \nabla  \mathcal{A} \|^2_{\infty}\lesssim j^2\kappa$. 
So we easily conclude that $\|\partial_r \mathcal{A}^j\|_\infty \leq j\kappa^\frac12$ and 
\[|t||\partial_r \mathcal{A}^j| \lesssim |t||\partial_r \mathcal{A}| + j|t| \1_{1\leq j|t| \leq 2} (M+\kappa)^\frac12 \in CM(C\kappa)\]
because $j|t| \1_{1\leq j|t| \leq 2} \in CM(C)$.
The lemma follows.
\end{proof}

\section{Proof of Theorem \ref{THRE2MA}: The Regularity Problem for a Reduced Class of Operators.}

\label{SThred}

In Theorem \ref{PaprRe}, we study the convergence of the solutions $u_j$ of the approximating operators constructed in Theorem \ref{PrAPA} to the solution $u$ of the initial operator.
Then we solve the Regularity problem for smooth boundary data in Theorem \ref{THRE2MA}, using the bound obtained in Theorem \ref{THMAIN1a}, the convergence of trace provided by Proposition \ref{APRP03}, and of course the convergence of solutions established in Theorem \ref{PaprRe}. 
It is important to understand that we have two convergences (one on the trace given by Proposition \ref{APRP03} and one on the solutions given by Theorem \ref{PaprRe}) and the uniqueness of the double limit is only guaranteed by the uniform convergence of the traces, which is given by \eqref{eqTHM00}. That is, we can prove the identity
\[ \lim_{j\to 0} \int_{\R^d}  |\Tr(\nabla_x u_j)|^2 \, dx = \lim_{\epsilon\to 0}  \int_{\R^d} \fint_{\epsilon/2 < |t| \leq \epsilon} |\nabla_x u|^2 \, dt \, dx\]
only when the assumptions of Theorem \ref{THMAIN1a} are satisfied.

\begin{Th}\label{PaprRe}
Let $L=-\diver (|t|^{d+1-n}\mathcal{A}\nabla)$ be a uniformly elliptic operator 
satisfying (\ref{ELLIP}). Let $\A^j$ be a sequence of matrices that converges pointwise to $\A$ and for which each $\A^j$ satisfies \eqref{ELLIP} with the same constant as $\A$. 
 
If $u$ and $u_j$ to be weak solutions in $W$ of respectively $Lu=0$ and $L_ju_j=0$ with the same trace, i.e. $u_j-u\in W_0$,  then $\|u_j-u\|_W$ converges to 0.
\end{Th}

\begin{proof}
According to Lemma \ref{LaxMilgram}, we have
\begin{align}\label{eqAP1}
\|u\|_W + \|u_j\|_{W} \lesssim \|g\|_H
\end{align}
where $g$ is the common trace $g=\Tr(u) = \Tr(u_j)$ and the constant is independent of $j$. Since $u-u_j\in W_0$ and $Lu= L_ju_j = 0$, we have
\begin{align}\label{eqAP11}
\iint_{\Omega}\mathcal{A}\nabla u\cdot \nabla (u-u_j)\frac{dtdx}{|t|^{n-d-1}}=\iint_{\Omega}\mathcal{A}^j\nabla u_j\cdot \nabla (u-u_j)\frac{dtdx}{|t|^{n-d-1}}=0.
\end{align}
By the uniform ellipticity of matrix $\mathcal{A}^j$ and (\ref{eqAP11}), we have
\begin{multline*}
\|u-u_j\|^2_W=\iint_\Omega |\nabla (u-u_j)|^2\frac{dtdx}{|t|^{n-d-1}}\\
\lesssim \iint_{\Omega}\mathcal{A}^j\nabla (u-u_j)\cdot \nabla (u-u_j)\frac{dtdx}{|t|^{n-d-1}}
=\iint_{\Omega}\mathcal{A}^j\nabla u\cdot \nabla (u-u_j)\frac{dtdx}{|t|^{n-d-1}}\\
=\iint_{\Omega}(\mathcal{A}^j-\mathcal{A})\nabla u\cdot \nabla (u-u_j)\frac{dtdx}{|t|^{n-d-1}}.
\end{multline*}
Furthermore, we apply the Cauchy-Schwarz inequality to obtain that 
\begin{multline}\label{eqAPR0}
\|u-u_j\|^2_W\lesssim \Big (\iint_\Omega |\nabla(u-u_j)|^2\frac{dtdx}{|t|^{n-d-1}}\Big )^{1/2}\Big (\iint_\Omega |\mathcal{A}^j-\mathcal{A}|^2|\nabla u|^2\frac{dtdx}{|t|^{n-d-1}}\Big )^{1/2} \\
= \|u-u_j\|_W \Big (\iint_\Omega |\mathcal{A}^j-\mathcal{A}|^2|\nabla u|^2\frac{dtdx}{|t|^{n-d-1}}\Big )^{1/2},
\end{multline}
hence
\[\|u-u_j\|_W^2 \lesssim \iint_\Omega |\mathcal{A}^j-\mathcal{A}|^2|\nabla u|^2\frac{dtdx}{|t|^{n-d-1}}\]
Since $\A$ and $\A_j$ are bounded by a uniform constant, the functions $|\A^j - \A|^2 |\nabla u|^2 |t|^{d+1-n}$ are bounded (uniformly in $j$) by $(2\lambda)^2 |\nabla u|^2 |t|^{d+1-n}$ which is integrable on $\Omega$. 
So by the Lebesgue's dominated convergence theorem, 
\[\lim_{j\to \infty} \|u-u_j\|_W^2 \lesssim \iint_\Omega \lim_{j\to 0} |\mathcal{A}^j-\mathcal{A}|^2|\nabla u|^2\frac{dtdx}{|t|^{n-d-1}} = 0\]
since $\A^j$ converges pointwise to $\A$. The theorem follows. 
\end{proof}

\begin{Cor}\label{coMAIN1}
Under the hypotheses of Theorem \ref{PaprRe}, if $\|\wt{N}(\nabla u_j)\|_{L^2(\mathbb{R}^d)}\lesssim \|\nabla g\|_{L^2(\mathbb{R}^d)}$ for all $j\in \mathbb{N}$, where the implicit constant is independent of $j$, then
\begin{align*}
\|\wt{N}(\nabla u)\|_{L^2(\mathbb{R}^d)}\lesssim \|\nabla g\|_{L^2(\mathbb{R}^d)}.
\end{align*}
\end{Cor}
\begin{proof}
Let $\epsilon>0$. We first notice that:
\begin{multline}\label{eqCoAP39}
\|\wt{N}(\nabla u|\Psi_\epsilon)\|_{L^2}\leq \|\wt{N}(\nabla u_j|\Psi_\epsilon)\|_{L^2}+\|\wt{N}(\nabla (u-u_j)|\Psi_\epsilon)\|_{L^2} \\ \lesssim \|\nabla g\|_{L^2} + \|\wt{N}(\nabla (u-u_j)|\Psi_\epsilon)\|_{L^2}
\end{multline}
by assumption. Yet, we have $\|\wt{N}(\nabla (u-u_j)|\Psi_\epsilon)\|_2 \leq C\epsilon^{-1/2} \|u-u_j\|_W \rightarrow 0$ by Lemma \ref{LN<infty} and then Theorem \ref{PaprRe}. So by taking the limit as $j$ goes to infinity,  \eqref{eqCoAP39} becomes
\[\|\wt{N}(\nabla u|\Psi_\epsilon)\|_{L^2} \lesssim \|\nabla g\|_{L^2}\]
The corollary follows then from the monotone convergence theorem.
\end{proof}

We conclude our section with the proof of Theorem \ref{THRE2MA}

\medskip

{\noindent \em Proof of Theorem \ref{THRE2MA}.}
Pick $\lambda>0$. Let $C_0$ (that depends only on $\lambda$, $d$ and $n$) be the constant in Theorem \ref{PrAPA}, and then let $\kappa_0<1$ (that depends on $d$, $n$ and $\lambda$) be the ``kappa'' value provided by Theorem \ref{THMAIN1a} for $\lambda$. We pick then $\kappa := \kappa_0/C_0$.

According to Theorem \ref{PrAPA}, there exists a sequence of $\{\mathcal{A}^j\}_{j\in \mathbb{N}}$ such that  $\mathcal{A}^j\rightarrow \mathcal{A}$ pointwise as $j\rightarrow \infty$. Each $\mathcal{A}^j$ satisfies the following conditions,
\begin{enumerate}[label=(\alph*)]
\item $\|\nabla \A^j\|_\infty \leq Cj\kappa^{1/2}$,
\item the operator $L_j := - \diver (|t|^{d+1-n} \A^j\nabla )$ satisfies \HH$_{\lambda,\kappa_0}$.
\end{enumerate}
Let $u_j\in W$ be the solution to $L_ju_j$ with $\Tr(u_j) = \Tr(u) = g \in H$ provided by Lemma \ref{LaxMilgram}. Our choice of $\kappa_0$ is small enough to have the inequality \eqref{eqTHM00} for each $u_j$. So Theorem \ref{THMAIN1a} and then Proposition \ref{APRP03} infer that
\[\|\wt{N}(\nabla u_j)\|^2_{2}\lesssim \limsup_{\epsilon \to 0} \int_{\R^d} \fint_{\epsilon/2 \leq |s| \leq \epsilon} |\nabla_x u_j|^2 \, ds\, dy = \|\nabla g\|_{2}^2,\]
with a constant that depends only on $d$, $n$ and $\lambda$ (in particular is independent of $j$). The theorem follows now from Corollary \ref{coMAIN1}.  
\hfill $\square$

\section{Proof of Theorem \ref{THRE3MA} 
} \label{Sselfimpro}

First, the solvability of the Regularity problem is stable under Carleson perturbations. 

\begin{Th}[Theorem 2.1 in \cite{kenig1993neumann}, Theorem 1.3 in \cite{dai2021carleson}] \label{ThCarlPert}
Let 
$$L_0 = - \diver [|t|^{d+1-n}\A_0 \nabla]\quad \mbox{and}\quad L_1 = - \diver [|t|^{d+1-n}\A_1 \nabla]$$ be two uniformly elliptic operators that satisfy \eqref{ELLIP} with the same constant $\lambda$. 
Assume that: 
\begin{enumerate}
\item the Regularity problem for the operator $L_0$ is solvable in $L^{p_0}$, that is there exists a constant $C_0$ such that for any $g\in C^\infty_0(\Omega)$, the solution $u_0$ to $L_0u_0=0$ constructed as in  \eqref{defug} (or equivalently as in Lemma \ref{LaxMilgram}) verifies
\begin{align}\label{NNu<Ng0}
\|\wt{N}(\nabla u)\|_{L^{p_0}(\R^d)} \leq C_0\|\nabla g\|_{L^{p_0}(\R^d)},
\end{align}
\item the disagreement $\A_1 - \A_0$ satisfies the Carleson measure condition with the constant $M$ - i.e. $|\A_1 - \A_0| \in CM(M)$.
\end{enumerate}
Then the Regularity problem for the operator $L_{1}$ is solvable in $L^{p_1}$ for some $p_1>1$, more precisely there exists $p_1\in (1,p_0]$ and $C_1$ both depending only on $\lambda$, $d$, $n$, $p_0$, $C_0$, and $M$, such that for any $g\in C^\infty_0(\Omega)$, the solution $u_1$ to $L_1u_1=0$ constructed in  \eqref{defug} verifies
\begin{align}\label{NNu<Ng1}
\|\wt{N}(\nabla u_1)\|_{L^{p_1}(\R^d)} \leq C_1 \|\nabla g\|_{L^{p_1}(\R^d)}.
\end{align}

Furthermore, if $M>0$ is small enough (depending on $\lambda$, $n$, $p_0$, and $C_0$), then we can take $p_1 = p_0$ in \eqref{NNu<Ng1}.
\end{Th}

The second result shows that any operator as in Theorem \ref{THRE3MA} can be compared to a Carleson perturbation of an operator satisfying \HH.

\begin{Prop}\label{PPLL1J}
Let $L_0=-\diver[|t|^{d+1-n}\B \nabla]$ be such that $\B$ satisfies \eqref{ELLIP} and can be written as a block matrix in the form
 \[\B = \begin{pmatrix} B_1 & B_2 \frac{t}{|t|} \\ \frac{t^T}{|t|}B_3 & b_4I \end{pmatrix}\]
 where $b_4I$ is the product of the identity matrix of order $n-d$ with a scalar function, $B_2$ is a $d$-dimensional vertical vertor \footnote{Recall that $t$ is a horizontal vector, so $B_2 \frac{t}{|t|}$ is a valid matrix product.}, $B_3$ is a $d$-dimensional horizontal vector \footnote{Since $t^T$ is a vertical vector, so $\frac{t^T}{|t|}B_3$ is a $(n-d)\times d$-matrix.}, and
 \[|t||\nabla B_1| + |t||\nabla B_2| + |t||\nabla B_3| + |t||\nabla b_4| \in CM(\kappa).\]

 If $\kappa>0$ is small enough, there exists a bi-Lipschitz change of variable $\rho$ from $\overline{\Omega}$ to $\overline{\Omega}$ such that 
\begin{enumerate}
\item $\rho(x) = x$ for any $x\in \R^d = \partial \Omega$;
\item there exists $C_\lambda$ such that for any weak solution $u$ to $L_0u=0$, the function $u\circ \rho$ is a weak solution to $L_\rho (u\circ \rho) = 0$ where the operator $L_\rho$ satisfies \eqref{ELLIP} with constant $C_\lambda$ and can be written as $$L_\rho = -\diver[|t|^{d+1-n}(\B_\rho + \C_\rho) \nabla]$$ 
where $\C_\rho \in CM(2\kappa)$ and $L_{\rho,0}:= -\diver[|t|^{d+1-n}\B_\rho \nabla]$ satisfies \HH$_{\lambda/2,2\kappa}$.
\end{enumerate}
Moreover, if $\B$ is symmetric, then $\B_\rho$ and $\C_\rho$ are symmetric.
\end{Prop}

\begin{Rem}
The assumption that $\kappa$ is small can actually be removed, but will make the proof longer. The proposition is a variant of the method presented in \cite{feneuil2021change}, and we refer any reader that wants to remove the condition on the smallness of $\kappa$ to the later article.
\end{Rem}

\begin{proof} As we just said, the proposition is a variant of the result given in \cite{feneuil2021change}. We will try to keep it light and refer to  \cite{feneuil2021change} for the details that we skipped. 

\medskip

{\bf Step 1: Change of variables to cancel the bottom left corner of $\B$.}
We write $v$ for the $d\times n$ matrix function $B_3/b_4$. We define $\rho_v$ as 
\[\rho_v:= \rho(x+ |t|v(x,t),t).\]
Observe that $\rho_v$ maps $\Omega$ to $\Omega$ and is the identity on $\R^d$. Its Jacobian matrix is
\[\Jac_v = \begin{pmatrix} I + |t|\nabla_x v & 0 \\ \frac{t^T}{|t|} v+ |t| \nabla_t v & I\end{pmatrix} = \begin{pmatrix} I & 0 \\ \frac{t^T}{|t|} v & I\end{pmatrix} + \mathcal O(|t||\nabla B_{3,4}|)\]
where $\mathcal O(h)$ denotes a quantity bounded by $Ch$, and $B_{3,4}$ denotes the the couple $(B_3,b_4)$. We have $|t||\nabla B_{3,4}|\in CM(\kappa)$, which implies $|t||\nabla B_{3,4}| \leq C \kappa$ by \eqref{PpCAL}, and since $\kappa$ is small, we deduce that $\Jac_v$ is invertible and $|\Jac_v| + |\Jac^{-1}_v|$ are bounded by $1.1\lambda$. In addition
\[\Jac_v^{-1} = \begin{pmatrix} I & 0 \\ -\frac{t^T}{|t|} v & I\end{pmatrix} + \mathcal O(|t||\nabla B_{3,4}|) \ \text{ and } \ \det(\Jac_v) = 1 + \mathcal O(|t||\nabla B_{3,4}|). \]

We define the conjugate\footnote{By conjugate, we mean that $L_v(u\circ \rho_v) = 0$ whenever $L u = 0$.} operator $L_v = - \diver [|t|^{d+1-n}\A_v \nabla]$ where
\begin{equation} \label{defArho}
\A_v := \left(\frac{\dist(\rho_v(x,t),\R^d)}{|t|}\right)^{d+1-n} \det(\Jac_v) \Jac^{-T}_v (\B\circ \rho_v) \Jac^{-1}_v.
\end{equation}
We check that $\dist(\rho_v(x,t),\R^d) = |t|$ and hence
\begin{multline*}\A_v = \begin{pmatrix} *  &[B_2 \circ \rho_v - (b_4 \circ \rho_v) v^T]\frac{t}{|t|}   \\ \frac{t^T}{|t|}  [B_3\circ \rho_v - (b_4 \circ \rho_v) v] & (b_4 \circ \rho_v) I \end{pmatrix} + \mathcal O(|t||\nabla B_{3/4}| ) \\
= \begin{pmatrix} * & [B_2 - (B_3)^T] \frac{t}{|t|}  \\ 0 & b_4 I \end{pmatrix} + \mathcal O(|t||\nabla B_{1,2,3,4}|+ |B_{1,2,3,4} \circ \rho_v - B_{1,2,3,4}|)
\end{multline*}
because $B_3 - b_4v = 0$ with our choice of $v$. We did not compute the upper left corner in the matrix above to lighten the notation, but we can have $B_1 - B_2B_3/b_4$. We write $\B^v$ for the matrix in the right-hand side above which has 0 in the bottom left corner, and $\C^v$ for $\A^v - \B^v$. We have that 
\[|\C^v| \lesssim |t||\nabla B_{1,2,3,4}| + |B_{1,2,3,4} \circ \rho_v - B_{1,2,3,4}| \in CM(C\kappa)\]
because $|B_{1,2,3,4} \circ \rho_v - B_{1,2,3,4}|(x,t)$ is bounded by the supremum of $|t||\nabla B_{1,2,3,4}|$ in a Whitney region around $(x,t)$, so satisfies Carleson estimate as long as $|t||\nabla B_{1,2,3,4}|$ does. The matrix $\B^v$ has the form
 \[\B^v = \begin{pmatrix} B_1^v & B_2^v \frac{t}{|t|} \\ \frac{t^T}{B_3^v} & b_4^vI  \end{pmatrix}  = \begin{pmatrix} B_1^v & B_2^v \frac{t}{|t|} \\ 0 & b_4I  \end{pmatrix}\]
 and the Carleson bound  $|t||\nabla B^v_{1,2,3,4}|\in CM(C\kappa)$ is consequence of the fact that the coefficients of $\B^v$ are product, quotient, difference, and sums of coefficients of $\B$ (we actually have $|t||\nabla B^v_{1,2,3,4}| \leq C |t||\nabla B_{1,2,3,4}|$).

\medskip

{\bf Step 2: Change of variables to reduce the bottom right corner of $\B_v$ to $I$.} The strategy is very similar to what we did in Step 1. Set $h :=  b_4$ and define $\rho_h$ as 
\[\rho_h:= \rho(x,th(x,t)).\]
As before, observe that $\rho_h$ maps $\Omega$ to $\Omega$ and is the identity on $\R^d$. Its Jacobian matrix is
\[\Jac_h = \begin{pmatrix} I & t\nabla_x h \\ 0 & hI + (\nabla_t h)t\end{pmatrix} = \begin{pmatrix} I & 0 \\ 0 & hI\end{pmatrix} + \mathcal O(|t||\nabla b_4|).\]
Since $\kappa$ is small (depending only on $d$, $n$ and $\lambda$), the matrix $\Jac_h$ is invertible and $|\Jac_h| + |\Jac^{-1}_h|$ are bounded by a constant that depends only on $\lambda$. In this case, we have
\[\Jac_h^{-1} = \begin{pmatrix} I & 0 \\ 0 & h^{-1}\end{pmatrix} + \mathcal O(|t||\nabla b_4|) \ \text{ and } \ \det(\Jac_v) = h^{n-d} + \mathcal O(|t||\nabla b_4|). \]

Note that $\dist(\rho_h(x,t),\R^d) = |t|h(x,t)$. Therefore,  the conjugate operator $L_h = - \diver [|t|^{d+1-n}\A_v \nabla]$ of $L_v$ by $\rho$ is such that
\begin{multline*}\A_h :=  \left(\frac{\dist(\rho_h(x,t),\R^d)}{|t|}\right)^{d+1-n} \det(\Jac_h) \Jac^{-T}_h (\A_v \circ \rho_h) \Jac^{-1}_h \\
= \begin{pmatrix} *  & [(B_2 - B_3^T)\circ \rho_h] \frac{t}{|t|} \\ 0 & h (b_4 \circ \rho_h) I \end{pmatrix} + \mathcal O(|t||\nabla b_4| + |\C_v \circ \rho_h|) \\
= \begin{pmatrix} *  & [B_2 - B_3^T] \frac{t}{|t|} \\ 0 & I \end{pmatrix} + \mathcal O(|t||\nabla b_4| + |\C_v \circ \rho_h| + |\B \circ \rho_h - \B|)
\end{multline*}
with our choice for $h$. We denote the matrix in the right-hand side above as $\B_h$, and $\C_h$ is $\A_h - \B_h$. The matrix $\B_h$ has the desired form, and $L_{\rho,0}:= - \diver[|t|^{d+1-n} \B_\rho \nabla]$ satisfies \HH$_{C_\lambda,C_\lambda\kappa}$. By definition 
\[|\C_h| \lesssim  |t||\nabla b_4| + |\C_v \circ \rho_h| + |\B \circ \rho_h - \B|\]
and the right-hand side above easily satisfies the Carleson measure condition with constant $C\kappa$ for the same reasons as in Step 1 (and the fact that bi-Lipschitz changes of variable preserve the Carleson measure condition).

\medskip

{\bf Conclusion.} The change of variables is $\rho := \rho_v \circ \rho_h$, which is bi-Lipschitz because $\rho_v$ and $\rho_h$ are bi-Lipschitz. The conjugate of $L$ by $\rho$ is $L_h$, and the ellipticity constant of $L_\rho$ is controlled by the ellipticity constant of $L$ (because Jacobian matrices of $\Jac_v$, $\Jac_h$, and their inverses are bounded by constants that depends only on $\lambda$). The top left corner of $\B_h$ does not really matter, but one can check that we (can) have
\[\B_h = \begin{pmatrix} b_4B_1 - B_2B_3  & [B_2 - B_3^T]\frac{t}{|t|} \\ 0 & I \end{pmatrix}\]
so $\B_h$ easily satisfies \HH$_{C_\lambda,C_\lambda\kappa}$. At last, notice that all our operations on the coefficients preserve the symmetry of matrix coefficients, which means that $\B_h$ and $\C_h$ are symmetric as long as $\B$ is symmetric. The proposition follows.
\end{proof}

We are now ready for the proof of our main theorem.

\medskip

\noindent {\em Proof of Theorem \ref{THRE3MA}.}
We consider the elliptic operator $L' = - \diver \B \nabla$, then we construct from it the change of variable from Proposition \ref{PPLL1J}. The conjugated operator of $L'$ by $\rho$ is in the form $L'_\rho:= - \diver [|t|^{d+1-n}(\B_\rho + \C_\rho) \nabla ]$, where $L'_{\rho,0}:= - \diver [|t|^{d+1-n}\B_\rho \nabla]$ satisfies \HH$_{C_\lambda,C_\lambda \kappa}$ and $\C_\rho \in CM(C\kappa)$. Therefore, if $\kappa:=\kappa(\lambda,n)$ is small enough we can apply Theorem \ref{THRE2MA} to say that the Regularity problem for the operator $L'_{\rho,0}$ is solvable in $L^2$ . 

The operator $L'_{\rho}$ is a small Carleson perturbation of $L'_{\rho,0}$, so Theorem \ref{ThCarlPert} gives that the Regularity problem for the operator $L'_{\rho}$ is solvable in the same space $L^{2}$. Since $L'$ and $L'_{\rho}$ are the same operator up to a bi-Lipschitz change of variable, the Regularity problem is also solvable for $L'$ in $L^{2}$. Now, $L$ is a small Carleson perturbation of $L'$, so we use Theorem \ref{ThCarlPert} again to obtain that the Regularity problem is still solvable for $L$ in $L^{2}$.  
\hfill $\square$

\section{A complement of a Lipschitz graph}\label{sLip}

The definition of cones, Whitney regions, non-tangential maximal function, and other objects given in the introduction was adapted to the that fact that the domain $\R^n\setminus \R^d$ is the product space $\R^d \times (\R^{n-d} \setminus \{0\})$. But we did so only for convenience, and equivalent definition can be given in general spaces.

If the domain is more general $\Omega$, we define the cones in $\Omega$ with vertex in $x\in \partial \Omega$ as
\begin{equation} \label{defconeG} 
\Gamma(x) := \{X\in \Omega, |X-x| < 2 \dist(X,\partial \Omega)\}.
\end{equation}
We can change the `aperture' of the cone by replacing the value 2 by any $\alpha >0$. The Whitney box $W(X)$ is defined as
\[W(X) := B(X,\dist(X,\partial\Omega)/2).\]
The definition of the cones and Whitney boxes given here are just example, as many variants exist. From there, we define the averaged non-tangential maximal function as
\[\wt N(u)(x) := \sup_{X\in \Gamma(x)} \left(\iint_{W(X)} |u(Y)|^2 dY\right)^\frac12.\]
From now on, we need a doubling measure on $\partial \Omega$, that we call $\sigma$. When $\partial \Omega$ is the graph of a Lipschitz function as in Corollary \ref{cLip} - $\sigma$ will simply be the $d$-dimensional Hausdorff measure. The Carleson measure condition, that is the substitute of \eqref{defCMM} is
\[f\in CM(M) \Longleftrightarrow \sup_{x\in \partial \Omega, r>0} \iint_{B(x,r) \cap \Omega} \sup_{W(X)} |f|^2 \frac{dX}{\dist(X,\partial \Omega)} \lesssim M \sigma(B(x,r) \cap \partial \Omega).\]
Moreover, we say that the Regularity problem is solvable in $L^p$ if, for any $g\in C_0^\infty(\R^d)$, we have
\[\|\wt N(\nabla u)\|_{L^p(\partial \Omega, \sigma)} \leq C \|\nabla_{\partial \Omega} g\|_{L^p(\partial \Omega,\sigma)}.\]
The gradient $\nabla_{\partial \Omega}$ is a gradient on $\partial \Omega$, which needs to be defined. In the simple case where $\partial \Omega$ is the graph of a Lipschitz function, as in Corollary \ref{cLip}, $\nabla_{\partial \Omega}$ is simply the classical gradient (that can be defined almost everywhere).

\bigskip

\noindent {\it{Proof of Corollary~\ref{cLip}. }}
We can construct a bi-Lipschitz change of variable $\rho$ (with Lipschitz constants close to 1) such that, for any weak solution $u$ to $L_\varphi u = 0$ in $\Omega_\varphi$, the function $u\circ \rho$ is a solution to $L_\rho (u\circ \rho) = 0$ in $\R^n \setminus \R^d$, where the operator $L_\rho$ satisfies the assumptions of Theorem \ref{THRE3MA}. The construction of such change of variable, and the properties of the conjugate operator $L_\rho$ are the main purpose of the article \cite{david2019dahlberg}.

The fact that $\rho$ is bi-Lipschitz and Theorem \ref{THRE3MA} entail then that 
\[ \|\wt N([\nabla u_g] \circ \rho)\|_{L^2(\R^d)}  \leq 2 \|\wt N(\nabla [u_g \circ \rho])\|_{L^2(\R^d)} \leq  C \|\nabla [g\circ \rho]\|_{L^2(\R^d)}.\]
The fact that $\|\wt N(v \circ \rho)\|_{L^2(\R^d)} \approx \|\wt N(v)\|_{L^2(\rho(\R^d),\sigma)}$ is a consequence of the fact that $\rho$ change the shape the regions $W(z,r)$, but preserves the fact that they are Whitney regions. Similarly, $\rho$ change the shape of the cones, but not the fact that they are the union of Whitney regions for a same point at all scale - i.e. a weaker version of ``cones'' variant to \eqref{defconeG} - and we know from \cite[Chapter II, § 2.5.1]{stein2016harmonic} and the various definitions of cones does not change the $L^p$-boundedness of the non-tangential maximal functions $N$ and $\wt N$. 
The fact that $\rho$ is bi-Lipschitz also infers the equivalence $\|\nabla [g\circ \rho]\|_{L^2(\R^d)} \approx \|\nabla g\|_{L^2(\partial \Omega_\varphi)}$. The corollary follows.
\ep

\end{document}